%% file: V5.tex
\documentclass[12pt]{amsart}
\usepackage[english]{babel}
\usepackage[T1]{fontenc}

\usepackage[a4paper,left=2.5cm,right=2.5cm,top=3cm,bottom=3cm]{geometry} 

\usepackage[dvipsnames, svgnames]{xcolor}
\definecolor{vert}{RGB}{15,120,5}
\definecolor{gris}{RGB}{128,128,128}
\definecolor{bleu}{RGB}{0,50,150}
\definecolor{rouge}{RGB}{149,24,24}
\usepackage[colorlinks = true,linkcolor=bleu,citecolor=vert,hyperindex,breaklinks]{hyperref}
\usepackage[nameinlink]{cleveref} 

\usepackage{lmodern}
\usepackage{amsmath}
\usepackage{amsthm}
\usepackage{bookmark}
\usepackage{amsfonts}
\usepackage{amssymb}
\usepackage{bbm}
\usepackage{mathrsfs}
\usepackage[all]{xy}
\usepackage{enumitem}
\usepackage{stmaryrd}
\usepackage{upgreek}
\usepackage{tikz}
\usepackage{tikz-cd}
\usetikzlibrary{cd}
\usepackage{adjustbox}
\usetikzlibrary{matrix,arrows}

\DeclareFontFamily{U}{BOONDOX-calo}{\skewchar\font=45 }
\DeclareFontShape{U}{BOONDOX-calo}{m}{n}{
  <-> s*[1.05] BOONDOX-r-calo}{}
\DeclareFontShape{U}{BOONDOX-calo}{b}{n}{
  <-> s*[1.05] BOONDOX-b-calo}{}
\DeclareMathAlphabet{\mathcalboondox}{U}{BOONDOX-calo}{m}{n}
\SetMathAlphabet{\mathcalboondox}{bold}{U}{BOONDOX-calo}{b}{n}
\DeclareMathAlphabet{\mathbcalboondox}{U}{BOONDOX-calo}{b}{n}

\newtheorem{theorem}{Theorem}
\numberwithin{theorem}{subsection}
\numberwithin{equation}{subsection}
\newtheorem{corollary}[theorem]{Corollary}
\newtheorem{lemma}[theorem]{Lemma}
\newtheorem{definition}[theorem]{Definition}
\newtheorem{proposition}[theorem]{Proposition}
\newtheorem{remark}[theorem]{Remark}
\newtheorem{example}[theorem]{Example}
\newtheorem{notations}[theorem]{Notations}

\newtheorem*{theorem*}{Theorem}
\newtheorem*{proposition*}{Proposition}
\newtheorem*{remark*}{Remark}
\newtheorem*{definition*}{Definition}

\newtheorem*{theoremf*}{Th\'eorème}
\newtheorem*{definitionf*}{Définition}

\newtheorem{emptypar}[theorem]{}

\newcommand{\Z}{\mathbb{Z}}
\newcommand{\Q}{\mathbb{Q}}

\newcommand{\mr}{\mathscr}
\newcommand{\mb}{\mathbb}
\newcommand{\mc}{\mathcal}
\newcommand{\un}{\mathbbm{1}}
\newcommand{\und}{\underline}
\newcommand{\rar}{\rightarrow}
\newcommand{\heart}{\ensuremath\heartsuit}

\newcommand{\id}{\operatorname{Id}}
\newcommand{\Hom}{\operatorname{Hom}}
\newcommand{\Map}{\operatorname{map}}

\newcommand{\et}{\operatorname{\acute{e}t}}
\newcommand{\Sm}{\operatorname{Sm}}
\newcommand{\Sh}{\operatorname{Sh}}
\newcommand{\DM}{\operatorname{\mathcal{DM}}}

\newcommand{\D}{\operatorname{D}}

\newcommand{\coh}{\operatorname{coh}}
\newcommand{\BM}{\operatorname{BM}}
\newcommand{\Hl}{\mathrm{H}}
\newcommand{\Spec}{\operatorname{Spec}}
\newcommand{\colim}{\operatorname{colim}}
\newcommand{\Rep}{\operatorname{Rep}}
\newcommand{\codim}{\operatorname{codim}}
\newcommand{\In}{\operatorname{Ind}}

\newcommand{\Loc}{\operatorname{Loc}}
\newcommand{\lis}{\operatorname{lisse}}
\newcommand{\Shv}{\operatorname{\mathcalboondox{Sh}}}

\newcommand{\pp}{\mathfrak{p}}
\newcommand{\Hlh}{{}^{hp}\Hl}

\newcommand{\ord}{\mathrm{ord}}

\title{Integral Artin motives II: Perverse motives and Artin Vanishing Theorem}
\author{Raphaël Ruimy}
\date{}

\newtheorem{introthm}{Theorem}

\newtheorem{introprop}[introthm]{Proposition}
\newtheorem{introdef}[introthm]{Definition}

\begin{document}
\input{Resume.tex}
\maketitle
\tableofcontents

\section*{Introduction}
\addtocontents{toc}{\protect\setcounter{tocdepth}{1}}

Thanks to \cite{thesevoe,ayo07,tcmm,ayo14,em,rob,khan,anabelian}, we now have at our disposal stable $\infty$-categories of mixed étale motives $\mc{DM}_{\et}$ endowed with Grothendieck's six functors formalism and various realization functors. 
However, to connect those categories to Grothendieck's conjectural theory of motives, the motivic t-structure is still missing.  Working by analogy with étale sheaves, there are two possible versions of the motivic t-structure: the \emph{perverse motivic t-structure} and the \emph{ordinary motivic t-structure}.

Indeed, if $\Lambda$ is an $n$-torsion ring and if $S$ is a scheme over which $n$ is invertible, the stable $\infty$-category $\mc{D}^b_c(S,\Lambda)$ of constructible complexes of étale sheaves with coefficients in $\Lambda$ is naturally endowed with its ordinary t-structure; the perverse t-structure then is defined on  by setting a complex $M$ to be 
\begin{itemize}
\item t-negative whenever for any point $x$ of $S$, letting $i_x\colon \{ x\} \rar S$ be the inclusion, the cohomology groups of the complex $i_x^*M$ vanish in degree higher than $\dim(\overline{\{ x\}})$. 
\item t-positive whenever for any point $x$ of $S$, the cohomology groups of the complex $i_x^!M$ vanish in degree lower than $\dim(\overline{\{ x\}})$. 
\end{itemize}

One of the central features of the perverse t-structure is that it satisfies Artin's vanishing Theorem:

\begin{theorem*}(Artin vanishing Theorem, \cite[XV.1.1.2]{travauxgabber}) 
    Let $f\colon X\rar Y$ be an affine morphism between quasi-excellent scheme. Then, the functor $$f_*\colon \mc{D}^b_c(X,\Lambda)\rar \mc{D}^b_c(Y,\Lambda)$$ is perverse right t-exact (meaning that it preserves negative objects of the perverse t-structure). By duality, the functor $f_!$ is perverse left t-exact.
\end{theorem*}
This result generalizes Artin's original vanishing Theorem from \cite[XIV.3.2]{sga4} which asserts that the étale cohomology group $\Hl^n_{\et}(X,\mc{F})$ vanishes when $X$ is a smooth affine variety over an algebraically closed field and $\mc{F}$ is a torsion étale sheaf and $n>\dim(X)$.

Furthermore, the conjectural existence of the perverse t-structure is critical in the theory of motives: its heart would be the abelian category of mixed motivic sheaves which should then satisfy Beilinson's conjectures on special values of $L$ functions (see \cite{Jannsen}).





In its full generality, this problem is beyond the scope of our current technology. On subcategories of étale motives, namely Artin motives, Artin-Tate motives and $1$-motives, the ordinary motivic t-structure is however now well understood thanks to \cite{orgo,ayo11,abv,bvk,plh,plh2,vaish,AM1,1mot}.


This leaves open the case of the perverse t-structure, even in the case of motives with rational coefficients.
In this paper, we construct the perverse motivic t-structure on Artin motives in dimension $2$ or less and we explain why it cannot exist in other cases. We show that there is however a variant: the \emph{perverse homotopy t-structure} which exists unconditionally with rational coefficients (but only in dimension $2$ or less integrally) and coincides with the perverse motivic t-structure whenever the latter is defined. 

This allows to formulate (and ultimately prove) Artin's Vanishing Theorem for Artin motives which is \Cref{Affine Lefschetz Q-cons}. Our approach in fact actually reverses the traditional philosophy: we first prove the existence of the t-structure and Artin's vanishing theorem rationally which is one of the main ingredients to prove (or disprove for higher dimensional schemes) the existence of the t-structure integrally.

The main player of this paper is the category of Artin étale motives:
\begin{definition*}Let $S$ be a scheme and let $R$ be a commutative ring. The category $\DM^A_{\et,c}(S,R)$ of \emph{constructible Artin étale motives} (or constructible $0$-motives) is defined as the thick subcategory\footnote{This means the smallest subcategory closed under finite limits, finite colimits and retracts.} of $\DM_{\et}(S,R)$ generated by the cohomological motives of finite $S$-schemes.
\end{definition*}

The goal is to investigate the existence of the perverse motivic t-structure as well as to prove properties of Artin vanishing type. The perverse motivic t-structure should be the one such that realization functors are t-exact.

\begin{introdef}\label{reduced l-adic real} Let $S$ be a scheme and $R$ be a localization of $\Z$. Let $t_0$ be a t-structure on $\DM_{\et,c}^A(S,R)$. For any prime $\ell$, consider the \emph{reduced $\ell$-adic realization functor}: 
$$\bar{\rho}_\ell\colon \mc{DM}_{\et,c}^A(S,R)\to \mc{DM}_{\et,c}^A(S[1/\ell],R)\overset{\rho_\ell}{\rar} \mc{D}^b_c(S[1/\ell],R\otimes_\Z \Z_\ell),$$ where $\rho_\ell$ is the usual $\ell$-adic realization from \cite{em}. We say that 
\begin{enumerate}
    \item the t-structure $t_0$ is the \emph{perverse motivic t-structure} if for any prime $\ell$, the functor $$\overline{\rho}_\ell\colon \DM^A_{\et,c}(S,R)\rar \mc{D}^b_c(S[1/\ell],R\otimes_\Z \Z_\ell)$$ is t-exact when the left hand side is endowed with $t_0$ and the right hand side is endowed with the perverse t-structure.
    \item the t-structure $t_0$ is the \emph{perverse homotopy t-structure} if for any constructible Artin étale motive $M$, the motive $M$ is $t_0$-negative if and only if for any prime $\ell$, the complex $\overline{\rho}_\ell(M)$ is perverse-negative.
\end{enumerate}
\end{introdef}

Note that if $t_0$ is the perverse t-structure it is also the perverse homotopy t-structure because the reduced $\ell$-adic realizations are jointly conservative on Artin motives (\cite[3.4.2]{AM1}). The perverse homotopy t-structure is inspired by Voevodsky's homotopy t-structure and by the $\delta$-homotopy t-structure of \cite{bondarko-deglise}. An $\ell$-adic analog was also considered in \cite{aps}. When it exists, the perverse homotopy t-structure is unique because t-negative objects determine a t-structure uniquely.

The perverse homotopy t-structure can actually be recovered from a t-structure on the category of (non-necessarily constructible) Artin motives $\DM_{\et}^A(S,R)$.
\begin{definition*}(\Cref{intro_perverse}) Let $S$ be a scheme endowed with a dimension function $\delta$ (see \Cref{dimension function}) and let $R$ be a commutative ring. The \emph{perverse homotopy t-structure} $t_{hp}$ on $\DM_{\et}^A(S,R)$ is the t-structure generated in the sense of \Cref{AM.t-structure generated} by the family of the Borel-Moore motives $M_S^\mathrm{BM}(X)[\delta(X)]$ with $X$ quasi-finite over $S$.
\end{definition*}
The above t-structure restricts to a t-structure on constructible Artin motives if and only if the perverse homotopy t-structure exists (and then both t-structure coincides) which justifies the terminology; this is \Cref{hp is hp}.

\subsection*{Summary of the paper}
Here are the main points of our investigation:
\begin{itemize}
	\item The perverse homotopy t-structure admits a pointwise description in the spirit of \cite{bbd}. Furthermore, over smooth Artin motives, the perverse homotopy t-structure exists and is simply a shift of the standard t-structure given by the identification from \cite{AM1} of smooth Artin motives with lisse étale sheaves. See \Cref{intro locality}
    \item $R=\Q$: the perverse homotopy t-structure in the sense of \Cref{reduced l-adic real} exists: see \Cref{Intro_MT_rat}. This can be derived from the same ideas as \cite{vaish} or \cite{aps}. 
	\item $R=\Q$: we give 3 analogs of the Artin vanishing theorem for Artin motives: see \Cref{Intro_lefschetz affine 2Q}, \Cref{affine lef direct image} and \Cref{Affine Lefschetz direct image}.
	\item $R=\Q$: we study the heart of the t-structure proving formal properties close to those of perverse sheaves and establishing links to known constructions on Artin motives from \cite{az,NV}: see \Cref{intro_simple objects EMX}.
    \item $R=\Z$: we prove the existence of the perverse homotopy t-structure (in the sense of \Cref{reduced l-adic real}) for schemes of dimension $2$ or less (we give counterexamples in dimension $4$). This is \Cref{intro_MT_Z}. The construction relies on the first version of Artin vanishing from the first part and on explicit computations of the Artin truncation functor $\omega^0$ (\Cref{def omega^0}). More precisely, we show that this functor has very pathological behaviors integrally: see \Cref{omega^0 sur un corps}. 
	The first version of the Artin vanishing theorem transfers readily to this setting.
    \item We prove that the $\ell$-adic realization is t-exact for schemes of dimension $2$ or less: this is done by using some gluing techniques from \cite{aps}. Thus in that case, the perverse motivic t-structure exists. The main result of \cite{aps} shows that this result is optimal. See \Cref{Intro_t-exactness}.
\end{itemize}
\subsection*{Pointwise description of the t-structure}
One of the key points for us will be the existence of the ordinary motivic t-structure on the subcategory $\DM^{smA}_{\et,c}(S,R)$ of \emph{constructible smooth Artin motives} from \cite{AM1}. The latter is the thick subcategory generated by cohomological motives of finite étale $S$-schemes. Local systems can be seen as perverse sheaves with a shift given by the dimension of the base. Here, constructible smooth Artin motives which are in the ordinary heart can also be seen as perverse Artin motives with the same shift. In fact, constructible smooth Artin motives play in the theory of Artin motives the same role as locally constant sheaves with perfect fibers in the theory of étale sheaves. By a result of \cite{AM1}, the analogy goes even deeper: constructible smooth Artin motives are identified to étale sheaves with perfect fibers over regular schemes in \cite{AM1}. 

Our other main technical tool is the Artin truncation functor. Recall from \cite{az} that for $S$ a noetherian scheme, the inclusion 
\[\iota \colon \DM^A_{\et}(S,R)\to \DM^\mathrm{coh}_{\et}(S,R)\]
of Artin motives into \emph{cohomological motives} has a right adjoint $\omega^0$: the \emph{Artin truncation functor}. This functor yields a trace of the six functors on Artin motives. We can use this to give a pointwise description of the t-structure in the spirit of \cite{bbd}:
\begin{introprop}\label{intro locality}(\Cref{AM.locality}) Let $S$ be a scheme with a dimension function $\delta$ and $R$ be a localization of $\Z$. If $x$ is a point of $S$, denote by $i_x\colon\{ x\}\rar S$ the inclusion. Let $M$ be an Artin motive over $S$ with coefficients in $R$. Then, \begin{enumerate}
\item $M\geqslant 0$ for the perverse homotopy t-structure if and only if it is bounded below with respect to the ordinary homotopy t-structure of \cite{AM1} and for any point $x$ of $S$, we have $$\omega^0i_x^!M\geqslant -\delta(x).$$
\item If $M$ is constructible, then $M\leqslant 0$ for the perverse homotopy t-structure if and only if for any point $x$ of $S$, we have $$i_x^*M\leqslant -\delta(x).$$
\item The perverse homotopy t-structure induces the ordinary t-structure (given by the identification with étale sheaves from \cite{AM1}) on smooth Artin motives up to a shift of $\delta(S)$.
\end{enumerate}
\end{introprop}
\subsection*{Rational coefficients} 
When the ring of coefficients is $\Q$, following the ideas of \cite{vaish,aps}, \Cref{intro locality} allows us to prove:
\begin{introprop}\label{Intro_MT_rat}(\Cref{main thm rationnel}) Let $S$ be an excellent scheme. The perverse homotopy t-structure exists on $\DM^A_{\et,c}(S,\Q)$.
\end{introprop}
Furthermore, if $p$ is a prime number, if the scheme $S$ is of finite type over $\mb{F}_p$, we defined in \cite{aps} 
the perverse homotopy t-structure on the category $\mc{D}^A(S,\Q_\ell)$ of Artin $\ell$-adic sheaves. We show that the $\ell$-adic realization $$\DM^A_{\et,c}(S,\Q)\rar \mc{D}^A(S,\Q_\ell)$$ is t-exact (\Cref{main thm rationnel}).

We can then prove our first analogue of Artin's vanishing Theorem. 
\begin{introthm}\label{Intro_lefschetz affine 2Q}(\Cref{lefschetz affine 2Q}) , let $f\colon X\rar S$ be a quasi-finite affine morphism. Assume that the scheme $X$ is nil-regular.


Assume furthermore that we are in one of the following cases 
\begin{enumerate}[label=(\alph*)]
    \item We have $\dim(S)\leqslant 2.$
    \item There is a prime number $p$ such that the scheme $S$ is of finite type over $\mb{F}_p$. 
\end{enumerate}
Then, the functor $$f_!\colon \DM^{smA}_{\et,c}(X,\Q)\rar \DM^A_{\et,c}(S,\Q)$$ is t-exact.
\end{introthm}

We give a second result which is "of Artin vanishing type". To state it, recall from \cite{plh} that the functor $\omega^0$ preserves constructible objects \cite{plh}. This yields for $f$ of finite type a right adjoint $\omega^0 f_*$ to the pullback functor $f^*$ on constructible Artin motives. When $f$ is smooth, Pepin Lehalleur showed that on Artin motives, the image of the unit object $\un_X$ through $\omega^0f_*$ can be described using the Stein factorization of $f$. We build on his result to prove a proper version of the Artin vanishing Theorem:

\begin{proposition*}(\Cref{affine lef direct image}) Let $S$ be a scheme allowing resolution of singularities by alterations and let $f\colon X\rar S$ be a proper morphism. Then, the functor $$\omega^0f_*\colon \DM^A_{\et,c}(X,\Q)\rar \DM^A_{\et,c}(S,\Q)$$ is right $t$-exact for the perverse homotopy t-structure. 
\end{proposition*}

Finally, denote the heart of the perverse homotopy t-structure as $\rm{M}^A_{\mathrm{perv}}(S,\Q)$. It shares similarities with the category of perverse sheaves: as in the case of perverse sheaves, we can define an analog of the intermediate extension functor denoted by $j_{!*}^A$. Furthermore, we have the following results:
\begin{introprop}\label{intro_simple objects EMX}(\Cref{simple objects,EMX}) Let $S$ be an excellent scheme.
\begin{enumerate} \item The abelian category $\rm{M}^A_{\mathrm{perv}}(S,\Q)$ is artinian and noetherian: every object is of finite length.
\item If $j\colon V\hookrightarrow S$ is the inclusion of a regular connected subscheme and if $L$ is a simple object of $\Loc_V(K)$, then the perverse Artin motive $j_{!*}^A(\rho_!L[\delta(V)])$ is simple. Every simple perverse Artin motive is obtained this way.
\item Let $d=\delta(S)$. Recall the motivic weightless complex $\mb{E}_S$ from \cite[3.21]{az}. Then, the motive $\mb{E}_S [d]$ is a simple perverse Artin motive over $S$.
\end{enumerate}
\end{introprop}
The last assertion relies on a third analog of the Artin vanishing theorem, namely that if $j$ is as above, the functor $$\omega^0j_*\colon \DM^{smA}_{\et,c}(V,\Q)\rar \DM^A_{\et,c}(S,\Q)$$ is perverse-homotopy t-exact (\Cref{Affine Lefschetz direct image}).

\subsection*{Integral coefficients}
When $R=\Z$, the situation is very different. First, we have a positive result for schemes of dimension $2$ or less. 

\begin{introthm}\label{intro_MT_Z}(\Cref{AM.main theorem} and \Cref{Affine Lefschetz}) Let $S$ be an excellent scheme of dimension $2$ or less endowed with a dimension function, let $R$ be a localization of $\Z$. Then, the perverse homotopy t-structure induces a t-structure on $\DM^A_{\et,c}(S,R)$.
	
Furthermore, if $f\colon T\rar S$ is a quasi-finite and affine morphism of schemes. Then, the functor $$f_!\colon \DM^A_{\et,c}(T,R)\rar \DM^A_{\et,c}(S,R)$$ is perverse homotopy t-exact.
\end{introthm}
However, unlike the case of rational Artin motives, we have a negative result for higher dimensional schemes. If $k$ is a field, the perverse homotopy t-structure of the stable $\infty$-category $\mc{DM}^A_{\et}(\mb{A}^4_k,\Z)$ does not induce a t-structure on the subcategory $\mc{DM}^A_{\et,c}(\mb{A}^4_k,\Z)$ (see \Cref{exemple dim 4}). The case of $3$-dimensional schemes remains open.

One of the main ingredients of \Cref{Intro_MT_rat}, is that when $R=\Q$, the functor $\omega^0$ preserves constructible objects. Unfortunately, when $R=\Z$, we show that this results fails completely: if $f\colon X\rar S$ is a morphism of finite type which is not quasi-finite, then, the Artin motive $\omega^0 f_! \un_X$ is not constructible (see \Cref{quand c'est qu'c'est constructible}). However, this functor is still tractable, even with integral coefficients when the base scheme is the spectrum of a field and we can generalize some of the computations of \cite{az}. To be more precise, when $f\colon X\to \Spec(k)$ is a proper morphism with $X$ regular and $k$ a field, it is possible to compute $\omega^0f_*\un_X$ explicitly purely in terms of étale cohomology. Roughly speaking, what happens is that as torsion motives are all Artin motives, a shadow of motives of higher dimension appears even though we stay in the realm of Artin motives. They often for instance appear as copies of $\Q/\Z$ which is non-constructible. This is \Cref{omega^0 sur un corps}. 
The key idea to prove the theorem is that in low dimension, truncation allows to ignore some of those contributions. This is reminiscent of some computations in \cite{bvk} such as \cite[5.1.1]{bvk}. 
\subsection*{Gluing and t-exactness of the realization }
In the last part of the paper, we show how to recover a perverse Artin motive from Artin representations and gluing data (\Cref{description d=2}). This allows for a proof of t-exactness of the realization functor:

\begin{introthm}\label{Intro_t-exactness}(\Cref{t-ex of l-adic real}) Let $S$ be an excellent scheme of dimension $2$ or less endowed with a dimension function, let $R$ be a localization of $\Z$, let $\ell$ be a prime number. Then, the reduced $\ell$-adic realization functor 
$$\overline{\rho}_\ell\colon \DM_{\et,c}^A(S,R)\rar \mc{D}^b_c(S[1/\ell],R\otimes_\Z \Z_\ell)$$ of \Cref{reduced l-adic real} is t-exact when the left hand side is endowed with the perverse homotopy t-structure and the right hand side is endowed with the perverse t-structure.
\end{introthm}
This proves that the perverse motivic t-structure exists for schemes of dimension at most $2$. The main result of \cite{aps} shows that this is optimal.

\subsection*{Acknowledgements}

This paper is part of my PhD thesis, done under the supervision of Frédéric Déglise. I would like to express my deepest gratitude to him for his constant support, for his patience and for his numerous suggestions and improvements. 

My sincere thanks also go to Joseph Ayoub, Jakob Scholbach and the anonymous referee for their very detailed reports that allowed me to greatly improve the quality of this paper.
I would also like to thank warmly Marcin Lara, Sophie Morel, Riccardo Pengo, Simon Pepin Lehalleur and Jörg Wildeshaus for their help, their kind remarks and their questions which have also enabled me to improve this work.

Finally, I would like to thank Olivier Benoist, Robin Carlier, Mattia Cavicchi, Adrien Dubouloz, Wiesława Nizioł, Fabrice Orgogozo, Timo Richarz, Wolfgang Soergel, Markus Spitzweck, Olivier Taïbi, Swann Tubach and Olivier Wittenberg for their interest in my work and for some helpful questions and advices. 

\section*{Notations and conventions}
\subsection*{Higher category theory}
In this text, we freely use the language of $\infty$-categories from \cite{htt,ha}. We will say \emph{category} instead of $\infty$-category.

If $\mc{C}$ is a stable category, we will denote by $\Map_\mc{C}$ the mapping spectrum for $\mc{C}$. The homomorphism functor $\Hom_\mc{C}$ identifies with $\Hl^0\Map_\mc{C}$. We say that a spectrum is $n$-connected if its $\Hl^k$ vanish for $k<n$.

\subsection*{t-structures}
We adopt the cohomological convention for t-structures (\textit{i.e} the convention of \cite[1.3.1]{bbd} which is the opposite of the convention of \cite[1.2.1.1]{ha}): a t-structure on a stable category $\mc{D}$ is a pair $(\mc{D}^{\leqslant 0},\mc{D}^{\geqslant 0})$ of full subcategories of $\mc{D}$ which are closed under isomorphisms and have the following properties: 
\begin{itemize}
	\item For any $M$ in $\mc{D}^{\leqslant 0}$ and any $N$ in $\mc{D}^{\geqslant 0}$, the abelian group $\Hom_\mc{D}(M,N[-1])$ vanishes.
	\item We have inclusions $\mc{D}^{\leqslant 0} \subseteq \mc{D}^{\leqslant 0}[-1]$ and $\mc{D}^{\geqslant 0}[-1] \subseteq \mc{D}^{\geqslant 0}$.
	\item For any $M$ in $\mc{D}$, there is an exact triangle \[M'\rar M \rar M''\] where $M'$ is an object of $\mc{D}^{\leqslant 0}$ and $M''$ is an object of $\mc{D}^{\geqslant 0}[-1]$.
\end{itemize}
A stable category endowed with a t-structure is called a t-category. If $\mc{D}$ is a t-category, we denote by $\mc{D}^\heart=\mc{D}^{\geqslant 0} \cap \mc{D}^{\leqslant 0}$ the heart of the t-structure (which is an abelian category). 
If $t_{?}$ is a t-structure, we denote by ${}^{?}\tau_{\geqslant 0}, {}^{?} \mathrm{H}^0,{}^? \mc{D}^\heart,\ldots$ the various notions attached to $t_?$. 

Finally, if $\mc{D}$ is a t-category and $\mc{D}'$ is a full stable subcategory of $\mc{D}$ such that  $(\mc{D}^{\leqslant 0} \cap \mc{D}', \mc{D}^{\geqslant 0} \cap \mc{D}')$ defines a t-structure on $\mc{D}'$, we say that the t-structure of $\mc{D}$ \emph{induces a t-structure} on $\mc{D}'$ and call the latter the \emph{induced t-structure}.
\subsection*{Schemes}
All schemes are assumed to be \textbf{noetherian} and of \textbf{finite dimension}; furthermore all smooth (and étale) morphisms and all quasi-finite morphisms are implicitly assumed to be separated and of finite type. The following notations will be used very often:
\begin{itemize}
	\item We let $\Sm$ be the class of smooth morphisms of schemes. For a scheme $S$, we let $S_{\et}$ (resp. $\Sm_S$) be the category of étale (resp. smooth) $S$-schemes.
	\item If $x$ is a point of a scheme $S$, we denote by $k(x)$ the residue field of $S$ at the point $x$.
	\item If $k$ is a field, we denote by $G_k$ its absolute Galois group.
	\item If $S$ is a scheme and $\xi$ is a geometric point of $S$, we denote by $\pi_1^{\et}(S,\xi)$ the étale fundamental group of $S$ with base point $\xi$ defined in \cite{sga1}.
\end{itemize}

Finally, if $S$ is a scheme. A \emph{stratification} of $S$ is a partition of $S$ into non-empty equidimensional locally closed subschemes called \emph{strata} such that the topological closure of any stratum is a union of strata.

\subsection*{Sheaves}
If $\mc{C}$ is a site and $R$ is a commutative ring, we denote by $\Shv(\mc{C},R)$ the category of hypersheaves $\mc{C}$ with value in the derived category of $R$-modules and by $\Sh(\mc{C},R)$ its heart which is the 1-category of sheaves of $R$-modules on $\mc{C}$.








\section{The perverse homotopy t-structure}
\addtocontents{toc}{\protect\setcounter{tocdepth}{2}}

\subsection{Artin motives}
We will extensively use the results and notations of \cite{em}. 
Their definitions and results are formulated in the setup of triangulated categories but can readily be adapted to our framework following the ideas of \cite{anabelian,khan,rob,agv}.


Let $S$ be a scheme and $R$ be a commutative ring. There are several models for the category of étale motives. We will use Ayoub's model developed in \cite{ayo07}: the stable category $\DM_{\et}(S,R)$ of \emph{étale motives} over $S$ with coefficients in $R$ is the $\mb{P}^1$-stabilization of the category of $\mb{A}^1$-local objects of $\Shv(\Sm_S,R)$. We denote its unit by $\un_S$.



We will often use references in \cite{em} which are formulated there in terms of h-motives. 
This is justified because the category of h-motives defined in \cite[5.1.1]{em} is equivalent over noetherian schemes of finite dimension to the category of étale motives (\cite[5.5.7]{em} applies in this generality, replacing \cite[4.1]{ayo14} by \cite[3.2]{bachmanrigidity} in the proof). 
\'Etale motives are endowed with the six functors by \cite{em}. In particular, we have localization exact triangles which will be central in this paper. More precisely, letting $i\colon Z\rar X$ be a closed immersion and $j\colon U\rar X$ be the complementary open immersion, we call \emph{localization triangles} the exact triangles of functors: 
\begin{equation}\label{AM.localization}j_!j^*\rar Id \rar i_*i^*.
\end{equation}
\begin{equation}\label{AM.colocalization}i_!i^!\rar Id \rar j_*j^*.
\end{equation}

To define Artin motives, we first recall the following definitions:
\begin{definition}\label{thickloc} Let $\mc{C}$ be a stable category 
\begin{enumerate}\item A \emph{thick} subcategory of $\mc{C}$ is a full subcategory $\mc{D}$ of $\mc{C}$ which is closed under finite limits, finite colimits and retracts.
\item A \emph{localizing} subcategory of $\mc{C}$ is a full subcategory $\mc{D}$ of $\mc{C}$ which is closed under finite limits and arbitrary colimits.
\item Let $\mc{E}$ be a set of objects. We call \emph{thick} (resp. \emph{localizing}) \emph{subcategory generated by} $\mc{E}$ the smallest thick (resp. localizing) subcategory of $\mc{C}$ whose set of objects contains $\mc{E}$.
 \end{enumerate}
\end{definition}

Let $S$ be a scheme and $R$ be a commutative ring. Recall for $f\colon X \to S$ a morphism of finite type the relative motives associated to $X$ in $\DM_{\et}(S,R)$:
\begin{itemize}
\item the cohomological motive $h_S(X)=f_*\un_X$.
\item the homological motive $M_S(X)=f_!f^!\un_S$.
\item the Borel-Moore motive $M_S^\mathrm{BM}(X)=f_!\un_X$.
\end{itemize}

We define Artin motives as follows:
\begin{definition}\label{def AM} Let $S$ be a scheme and $R$ be a commutative ring. We define
\begin{enumerate}
\item The category $\DM_{\et,(c)}^{A}(S,R)$ of \emph{(constructible) Artin étale motives} to be the localizing (resp. thick) subcategory of $\DM(S,R)$ generated by the $h_S(X)$ for $X$ finite over $S$.
\item The category $\DM_{\et,(c)}^{smA}(S,R)$ of \emph{(constructible) smooth Artin étale motives} over $S$ to be the localizing (resp. thick) subcategory of $\DM_{\et}(S,R)$ generated by the $h_S(X)$ for $X$ finite étale over $S$.
\end{enumerate}
\end{definition}

The functors $\otimes$ and $f^*$ (where $f$ is any morphism) restrict to the $\DM_{\et,(c)}^{(sm)A}(-,R)$. In the non-smooth case, we have an additional exceptional functoriality:

\begin{proposition}\label{f_!}(\cite[1.17]{plh}) Let $R$ be a commutative ring. The categories $\mc{DM}^A_{\et,(c)}(-,R)$ are closed under $f_!$, where $f$ is a quasi-finite morphism.

Hence, $\DM_{\et,(c)}(-,R)$ satisfies the localization property \eqref{AM.localization}.
\end{proposition}
Beware that if $j\colon U\rar S$ is an open immersion, the motive $j_*\un_U$ need not be Artin and thus the fibered category $\DM_{\et}(-,R)$ does not satisfy the localization property \eqref{AM.colocalization}.

For any constructible Artin motive, there is a stratification along which it becomes smooth (\cite[3.5.1]{AM1}). Hence smooth Artin motives can be seen as building blocks of Artin motives. 
They can furthermore be identified to the categories of étale sheaves.

\begin{definition}
    Let $S$ be a scheme and let $R$ be a commutative ring.
\begin{enumerate}
	\item The \emph{category of lisse étale sheaves} $\Shv_{\mathrm{lisse}}(S,R)$ is the subcategory of dualizable objects of $\Shv(S_{\et},R)$.
    \item The \emph{category of Ind-lisse étale sheaves} $\Shv_{\In\mathrm{lisse}}(S,R)$ is the localizing subcategory of $\Shv(S_{\et},R)$ generated by $\Shv_{\mathrm{lisse}}(S,R)$.
\end{enumerate}
\end{definition}

\begin{emptypar}\label{small et site}
To formulate the link between étale sheaves and Artin motives, recall that the inclusion of sites $S_{\et}\to \Sm_S$ induces a functor 
\[\rho_!\colon \Shv(S_{\et},R)\to \DM_{et}(S,R)\]
by \cite[4.4.2]{em}. Furthermore, the functor $\rho_!$ has its essential image contained in $\DM^A_{\et}(S,R)$ (see \cite[1.5.5]{AM1}).
\end{emptypar}

\begin{theorem}\label{smooth Artin motives}(\cite[3.1.6, 3.3.1]{AM1}) Let $R$ be a localization of $\Z$ and let $S$ be a regular scheme. Assume that the residue characteristic exponents of $S$ are invertible in $R$. Then, the functor $\rho_!$ induces monoidal equivalences
$$\Shv_{\In\lis}(S,R)\longrightarrow \DM^{smA}_{\et}(S,R)$$
$$\Shv_{\lis}(S,R)\longrightarrow \DM^{smA}_{\et,c}(S,R).$$

Furthermore, the ordinary t-structure on the stable category $\Shv(S_{\et},R)$ induces t-structures on the subcategories $\Shv_{\lis}(S,R)$ and $\Shv_{\In\lis}(S,R)$.
The heart those t-structures are the categories $\Loc_S(R)$ (resp. $\In \Loc_S(R)$) of locally constant sheaves of $R$-modules with finitely presented fibers (resp. filtered colimits of such sheaves). 
This yields t-structures on smooth Artin motives and their constructible counterpart.
\end{theorem}






The most important technical tool of the paper is the Artin truncation functor $\omega^0$. To define it, first recall the category of cohomological motives:

\begin{definition}
Let $S$ be a scheme and let $R$ be a commutative ring. The category $\DM_{\et,(c)}^{\coh}(S,R)$ of \emph{(constructible) cohomological étale motives} over $S$ is the localizing (resp. thick) subcategory generated by the motives of the form $h_S(X)$ for $X$ proper over $S$. 
\end{definition}

\begin{definition}\label{def omega^0}
The adjoint functor theorem of \cite[5.5.2.9]{htt} ensures that the inclusion of Artin motives into cohomological motives has a right adjoint:
$$\omega^0\colon\DM^{\coh}_{\et}(S,R)\rar \DM^A_{\et}(S,R).$$ which is called the \emph{Artin truncation functor}.
\end{definition}

This functor was first introduced in \cite[2.2]{az} and was further studied in \cite{plh} and \cite{vaish2}. We will use the compatibilities of $\omega^0$ with the six functors given in \cite[3.3]{plh}; in \textit{loc. cit.} they are stated with rational coefficients but the proofs work verbatim with any commutative ring of coefficients. As a consequence, we for instance get an analog of the localization triangle \eqref{AM.colocalization} for Artin motives:

\begin{proposition} Let $R$ be a commutative ring. Let $i\colon Z\rar S$ be a closed immersion, let $j\colon U\rar S$ be the complementary open immersion. Then, we have an exact triangle
$$i_!\omega^0i^!\rar \mathrm{id}\rar \omega^0j_* j^*$$
in $\DM_{\et}^A(S,R)$.
\end{proposition}
\begin{proof} Apply the functor $\omega^0$ to the localization triangle \eqref{AM.colocalization} and use the fact that the functor \[\eta_i\colon i_!\omega^0\to \omega^0 i_!\] of \cite[3.3]{plh} is an equivalence to get the result.
\end{proof}

We will see later that this exact triangle need not induce an exact triangle in $\DM_{\et,c}^A(S,R)$ in general. We can rephrase the above result in terms of gluing in the sense of \cite{bbd}.

\begin{corollary}\label{glueing_0}
    Keep the same notations. Then, the category $\mc{DM}^A_{\et}(S,R)$ is a gluing of the pair $(\mc{DM}^A_{\et}(U,R),\mc{DM}^A_{\et}(Z,R))$ along the fully faithful functors $i_*$ and $\omega^0j_*$ in the sense of \cite[A.8.1]{ha}, \textit{i.e.} the functors $i_*$ and $\omega^0j_*$ have left adjoint functors $i^*$ and $j^*$ such that
\begin{enumerate}
    \item We have $j^*i_*=0$.
    \item The family $(i^*,j^*)$ is conservative.
\end{enumerate}
In particular, by \cite[A.8.5, A.8.13]{ha}, the sequence $$\mc{DM}^A_{\et}(F,R)\overset{i_*}{\rar}\mc{DM}^A_{\et}(S,R)\overset{j^*}{\rar}\mc{DM}^A_{\et}(U,R)$$ satisfies the axioms of the gluing formalism of \cite[1.4.3]{bbd}.
\end{corollary}

In \cite{az}, Ayoub and Zucker proved that if $S$ is a quasi-projective scheme over a field of characteristic $0$ and if $R=\Q$, then, the functor $\omega^0_S$ preserves constructible objects (see \cite[2.15 (vii)]{az}). In \cite[3.7]{plh}, Pepin Lehalleur extended the result to schemes allowing resolutions of singularities by alterations. In the proof, Pepin Lehalleur uses this hypothesis only to ensure that the localization triangle \eqref{AM.colocalization} belongs to $\DM^{\coh}_{\et,c}(-,R)$. Using \Cref{main_thm_appendix}, 
we see that this assumption is in fact unnecessary and we deduce the following statement. 

\begin{proposition}\label{omega^0} Let $S$ be a quasi-excellent scheme and $R=\Q$. Then, the functor $\omega^0$ preserves constructible objects over $S$.
\end{proposition}


Using the Artin truncation functor $\omega^0$, there is a trace of the six functors formalism on constructible Artin motives with rational coefficients ; we will see that this is not true with integral coefficients. Namely, if $k$ is a field, and $i$ is the inclusion of a point in $\mb{P}^1_k$, the motive $\omega^0i^!\un_{\mb{P}^1_k}$ is not constructible (see \Cref{omega^0 corps} below).

\begin{proposition}\label{AM.six foncteurs constructibles} The six functors on $\mc{DM}_{\et}(-,\Q)$ induce 
\begin{enumerate}\item a functor $f^*$ for any morphism $f$,
\item a functor $\omega^0 f_*$ for any separated morphism of finite type $f$ whose target is quasi-excellent,
\item a functor $f_!$ for any quasi-finite morphism $f$,
\item a functor $\omega^0 f^!$ for any quasi-finite morphism $f$ with a quasi-excellent source,
\item functors $\otimes$ and $\omega^0\underline{\Hom}$ (where $\underline{Hom}$ is the internal Hom of étale motives)
\end{enumerate}
on $\DM^A_{\et,c}(-,\Q)$. Moreover, we have the localization property \eqref{AM.localization} on $\DM_{\et,c}^A(-,\Q)$: if $i\colon Z\rar S$ is a closed immersion and $j\colon U\rar S$ is the complementary open immersion, we have an exact triangle
$$i_!\omega^0i^!\rar \mathrm{id}\rar \omega^0j_* j^*$$
in $\DM_{\et,c}^A(S,\Q)$. 
\end{proposition}
\begin{proof}We already proved the proposition in the case of left adjoint functors in \Cref{f_!}. The case of right adjoint functors follows from the fact that $\DM^{\coh}_{\et,c}(-,R)$ is endowed with the six functors formalism and from \Cref{omega^0}.
\end{proof}

\subsection{Definition and basic facts}
First recall the definition of dimension functions (see for instance \cite[1.1.1]{bondarko-deglise}).
\begin{definition}\label{dimension function}
 Let $S$ be a scheme. A \emph{dimension function} on $S$ is a function $$\delta\colon S\rar \Z$$ such that for any immediate specialization $y$ of a point $x$, \textit{i.e} a specialization such that $$\codim_{\overline{\{ x\}}} (y)=1,$$ we have $\delta(x)=\delta(y)+1$. 
 
A pair $(S,\delta)$ with $S$ a scheme and $\delta$ a dimension function is called a \emph{d-scheme}.
\end{definition}

\begin{emptypar}Two dimension functions on a scheme differ by a Zariski-locally constant function. 
Moreover, if a scheme $S$ is universally catenary and integral, the formula $$\delta(x)=\dim(S)-\codim_S(x)$$ defines a dimension function on $S$. 
In addition, since by convention all the schemes considered here are finite-dimensional and noetherian, our dimension functions will always be bounded.
Finally, any $S$-scheme of finite type $f\colon X\rar S$ inherits a canonical dimension function by setting 
$$\delta(x)=\delta(f(x))+\mathrm{tr.deg}\hspace{1mm} k(x)/k(f(x)).$$
By convention, when we consider a d-scheme and a morphism of finite type to the underlying scheme, we will always endow the source with the induced dimension function.
\end{emptypar}

The t-structure will be defined from a set of generators in the following sense: 
\begin{proposition}(\cite[1.4.4.11]{ha}) \label{AM.t-structure generated} Let $\mc{C}$ be a presentable stable category. Given a small family $\mc{E}$ of objects, the smallest subcategory $\mc{E}_-$ closed under small colimits and extensions is the set of non-positive objects of a t-structure.

In this setting, we will call this t-structure the \emph{t-structure generated by} $\mc{E}$. 
\end{proposition}
We can now define the perverse homotopy t-structure.

\begin{definition}\label{intro_perverse} Let $(S,\delta)$ be a d-scheme and let $R$ be a commutative ring. The \emph{perverse homotopy t-structure} $t_{hp}$ on $\DM_{\et}^A(S,R)$ is the t-structure generated in the sense of \Cref{AM.t-structure generated} by the family of the motives $M_S^\mathrm{BM}(X)[\delta(X)]$ with $X$ quasi-finite over $S$.
\end{definition}

Using the same method as \cite[1.28]{plh}, this t-structure is also generated by family of the motives $h_S(X)[\delta(X)]$ with $X$ finite over $S$. Thus, an object $M$ of $\DM_{\et}^A(S,R)$ belongs to $\DM_{\et}^A(S,R)^{hp \geqslant n}$ if and only if for any finite $S$-scheme $X$ the complex $\Map_{\DM_{\et}(S,R)}(h_S(X),M)$ is $(n-\delta(X)-1)$-connected. Let us now derive the formal properties of the t-structure that will be constantly used in the rest of the text. First note that over the spectrum of a field, the t-structure is just given by a shift of the ordinary homotopy t-structure of \cite{AM1}:
\begin{proposition}\label{fields} Let $k$ be a field and let $R$ be a localization of $\Z$. Then, the ordinary homotopy t-structure of \cite{AM1} and the perverse homotopy t-structure on $\DM^A_{\et}(k,R)$ coincide up to a shift of $\delta(k)$. In particular, when $R$ is regular, it induces a t-structure on constructible Artin motives which coincides up to the same shift with the ordinary t-structure on $\Shv_\mathrm{lisse}(k_{\et},R)$ through the equivalence $\rho_!$.
\end{proposition}
\begin{proof} The ordinary homotopy t-structure and $t_{hp}$ have the same set of generators up to a shift of $\delta(k)$. 
\end{proof}

\begin{proposition}\label{AM.t-adj} Let $R$ be a commutative ring. Let $f\colon X \to S$ be a quasi-finite morphism of schemes and let $g\colon Y \to S$ be a finite-type morphism of relative dimension $d$. Assume given a dimension function on $S$ and endow $Y$ and $X$ with the induced dimension functions.
\begin{enumerate} 
\item The adjunction $(f_!,\omega^0f^!)$ is a $t_{hp}$-adjunction.
\item If $M$ is a $t_{hp}$-non-positive Artin motive, then, the functor $-\otimes_S M$ is right $t_{hp}$-exact.
\item If $\dim(g)\geqslant d$, then, the adjunction $(g^*[d],\omega^0g_*[d])$ is a $t_{hp}$-adjunction.
\end{enumerate}
\end{proposition}
\begin{proof} The proof is the same as the proof of \cite[4.1.2]{AM1}
\end{proof}
\begin{corollary}\label{AM.t-adj2} Let $f$ be a morphism of schemes with a given dimension function. Then, 
\begin{enumerate} 
\item If $f$ is étale, the functor $f^*=f^!=\omega^0f^!$ is $t_{hp}$-exact.
\item If $f$ is finite, the functor $f_!=f_*=\omega^0f_*$ is $t_{hp}$-exact.
\end{enumerate}
\end{corollary}

\begin{proposition}\label{glueing perverse} Let $(S,\delta)$ be a d-scheme, let $R$ be a commutative ring, let $i\colon F\rar S$ be a closed immersion and $j\colon U\rar S$ be the open complement. The perverse homotopy t-structure on $\mc{DM}^A_{\et}(S,R)$ is obtained by gluing the t-structures of $\mc{DM}^A_{\et}(U,\Z_\ell)$ and $\mc{DM}^A_{\et}(F,R)$ (see \cite[1.4.9]{bbd}) \textit{i.e.} for any object $M$ of $\mc{DM}^A_{\et}(S,R)$, we have
\begin{enumerate}\item $M\geqslant_{hp} 0$ if and only if $j^*M \geqslant_{hp} 0$ and $\omega^0 i^! M \geqslant_{hp} 0$.
\item $M \leqslant_{hp} 0$ if and only if $j^*M\leqslant_{hp} 0$ and $i^*M \leqslant_{hp} 0$.
\end{enumerate}
\end{proposition}
\begin{proof} This follows from the usual properties of the six functors, \Cref{AM.t-adj} and \Cref{AM.t-adj2}.
\end{proof}

\begin{corollary}\label{weak locality perverse} Let $(S,\delta)$ be a d-scheme, let $R$ be a commutative ring and let $M$ be an object of $\mc{DM}^A_{\et}(S,R)$. Then, we have
\begin{enumerate}\item $M\geqslant_{hp} 0$ if and only if there is a stratification of $S$ such that for any stratum $i\colon T\hookrightarrow S$, we have $\omega^0 i^! M \geqslant_{hp} 0$.
\item $M \leqslant_{hp} 0$ if and only if there is a stratification of $S$ such that for any stratum $i\colon T\hookrightarrow S$, we have $i^*M \leqslant_{hp} 0$.
\end{enumerate}
\end{corollary}

\begin{proposition}\label{boundedness} Let $(S,\delta)$ be a d-scheme and let $R$ be a commutative ring. Then, the constructible objects are bounded with respect to the perverse homotopy t-structure on $\DM^A_{\et}(S,R)$.
\end{proposition}
\begin{proof} The subcategory of bounded objects is thick. Therefore, it suffices to show that the motives of the form $f_*\un_X$ with $f\colon X\rar S$ finite are bounded. By \Cref{AM.t-adj2}, the functor $f_*$ is $t_{hp}$-exact when $f$ is finite. Therefore, it suffices to show that the object $\un_S$ is bounded with respect to $t_{hp}$ for any scheme $S$. We already know that it lies in degree at most $\delta(S)$ by definition of $t_{hp}$.

Let $Y$ be a finite $S$-scheme. The objects $h_S(Y)$ and $\un_S$ belong to the heart of the ordinary homotopy t-structure by \cite[4.1.4]{AM1}. Therefore, the complex $\Map(h_S(Y),\un_S)$ is $(-1)$-connected. Therefore, if $c$ is a lower bound for $\delta$, the complex $\Map(h_S(X),\un_S[c])$ is $(\delta(X)-1)$-connected and therefore, the object $\un_S$ is bounded below by $c$. 
\end{proof}

\subsection{Pointwise description}
In the following section, we give a pointwise description of the perverse homotopy t-structure as well as its behavior on smooth Artin motives. It will be the basis for our understanding of the t-structure. This pointwise description is very analogous to the defining property of the perverse t-structure on étale sheaves (with an extra $\omega^0$ that ensures that we stay in the world of Artin motives).

We begin by exploring how the t-structure behaves when we change the ring of coefficients.

\begin{emptypar}\label{Change of coefficients}
Let $\sigma\colon R \rar R'$ be a commutative ring homomorphism. 
We have a functorial adjunction 
$$\sigma^* \colon \DM_{\et}(-,R)\leftrightarrows \DM_{\et}(-,R') \colon \sigma_*,$$ which is compatible with pullbacks and $\otimes$.

If $S$ is a scheme, we can identify $\DM_{\et}(S,R)$ with the category of $R$-modules in $\DM_{\et}(S,\Z)$. The functor $\sigma_*$ is then the forgetful functor, and the functor $\sigma^*$ is the functor $-\otimes_R R'$.

When $R$ is a localization of $\Z$, is $n$ an integer and let $A$ is another localization of $\Z$, we will denote by $\sigma_n\colon R\rar R \otimes_\Z \Z/n\Z$ and by $\sigma_A\colon R\rar R\otimes_{\Z} A$ the natural maps.
\end{emptypar}


\begin{lemma}\label{conservativity}(\cite[5.4.12]{em}) 
Let $S$ be a scheme and let $R$ be commutative ring, then with the above notations, the family $(\sigma_\Q^*,\sigma_p^* \mid p \text{ prime })$ is conservative.
\end{lemma}

We have another similar conservativity lemma. If $p$ is a prime number, denote by $\Z_{(p)}$ the localization of $\Z$ with respect to the prime ideal generated by $p$. Denote by $r_p$ the ring morphism $\sigma_{\Z_{(p)}}$. Then, the morphisms $\sigma_{\Q}$ and $\sigma_p$ factor through $r_p$. This yields the following corollary.

\begin{corollary}\label{rl} Let $S$ be a scheme and let $R$ be a commutative ring, then with the above notations, the family $(r_p^* \mid p \text{ prime })$ is conservative.
\end{corollary}

\begin{proposition}\label{coeff change perverse} Let $(S,\delta)$ be a d-scheme, let $R$ be a commutative ring, 
Recall the notations of \Cref{Change of coefficients}.
Then,
\begin{enumerate}\item The functors $(\sigma_A)_*$ and $(\sigma_n)_*$ are $t_{hp}$-exact.
\item the functor $\sigma_A^*$ is $t_{hp}$-exact.
\end{enumerate}
\end{proposition}
\begin{proof} This follows from \cite[1.1.5]{AM1} and from \cite[1.2.1]{AM1}.
\end{proof}

We can now prove the main result of this section.

\begin{proposition}\label{AM.locality}
Let $(S,\delta)$ be a d-scheme and $R$ be a localization of $\Z$. If $x$ is a point of $S$, denote by $i_x\colon\{ x\}\rar S$ the inclusion. Then, the following properties hold.
\begin{enumerate} \item Let $M$ be in $\DM^A_{\et}(S,R)$. Then, 
    \begin{enumerate}
\item $M\geqslant_{hp}0$ if and only if $M$ is bounded below with respect to the ordinary homotopy t-structure $t_\mathrm{ord}$ of \cite{AM1} and for any point $x$ of $S$, we have $$\omega^0i_x^!M\geqslant_{hp} 0,$$ \textit{i.e.} $\omega^0i_x^!M\geqslant_{\ord} -\delta(x)$.
\item Assume that $M$ is constructible. Then, $M\leqslant_{hp}0$ if and only if for any point $x$ of $S$, we have $$i_x^*M\leqslant_{hp} 0,$$ \textit{i.e.} $i_x^*M\leqslant_{\ord} -\delta(x)$.
\end{enumerate}
\item Assume that the scheme $S$ is regular and connected. Then, the functor $\rho_!$ induces a t-exact functor $$\rho_!\colon \Shv_{\lis}(S,R)\rar \DM_{\et}^A(S,R)$$ when the left hand side is endowed with the ordinary t-structure shifted by $\delta(S)$ and the right hand side is endowed with the perverse homotopy t-structure.
\end{enumerate}
\end{proposition}
\begin{remark}\label{hp is hp} 
The above result shows that the perverse homotopy t-structure in the sense of \Cref{reduced l-adic real} exists if and only if the perverse homotopy t-structure restricts to $\DM^A_{\et,c}(S,R)$ (and then both t-structure coincide) which justifies the terminology.

Indeed, let $M$ be in $\DM^A_{\et,c}(S,R)$. Then, 
\begin{align*}M\leqslant_{hp} 0 &\Longleftrightarrow \forall x \in S, i_x^*M\leqslant_{hp} 0 \\ &\Longleftrightarrow \forall x \in S, \forall \ell \text{ prime }i_x^*\overline{\rho}_\ell(M)\leqslant_{p} 0  \\
&\Longleftrightarrow \forall \ell \text{ prime }\overline{\rho}_\ell(M)\leqslant_{p} 0.
\end{align*}
with $t_p$ the perverse t-structure on $\ell$-adic sheaves.
The first equivalence indeed follows from Assertion (1)(b); the second equivalence follows from \Cref{fields} and the fact that the $\ell$-adic realization respects the six functors; finally the third equivalence is a characterization of the perverse t-structure (see \cite[2.2.12]{bbd}).

Hence, if the perverse homotopy t-structure induces a t-structure on $\DM^A_{\et,c}(S,R)$, it is the perverse homotopy t-structure in the sense of \Cref{reduced l-adic real}.
Conversely, assume that the perverse homotopy t-structure in the sense of \Cref{reduced l-adic real} exists and denote it by $t_{hp}'$. Then, its t-non-positive objects are the $t_{hp}$-non-negative objects which are constructible. Therefore, the $t_{hp}'$-positive objects are those constructible objects $N$ such that for all constructible object $M$ with $M\leqslant_{hp} 0$, we have 
$$\Hom_{\DM^A_{\et}(S,R)}(M,N)=0.$$
But the generators of the $t_{hp}$ are constructible and thus, those objects are exactly the $t_{hp}$-positive objects which are constructible.
\end{remark}

\begin{proof} 

\textbf{Proof of Assertion (1)(a):} 

If $M\geqslant_{hp} 0$, we already know from \Cref{AM.t-adj} that for any point $x$ of $S$, we have $\omega^0i_x^!M\geqslant_{hp} 0$. Furthermore, if $c$ is a lower bound for the function $\delta$ on $S$, the complex $\Map(M_S(X),M)$ is $(c-1)$-connected for any étale $S$-scheme $X$. Therefore, the Artin motive $M$ is bounded below with respect to the ordinary homotopy t-structure. 

Conversely, assume that $M$ is $t_{\ord}$-bounded below and that for any point $x$ of $S$, the Artin motive $\omega^0i_x^!M$ is $t_{hp}$-non-negative. Let $f\colon X\rar S$ be quasi-finite. Then, by \cite[3.1.5]{bondarko-deglise}, we have the $\delta$-niveau spectral sequence 
$$E^1_{p,q}\Rightarrow \Hl^{p+q}\Map_{\DM_{\et}(S,R)}(M_S^{\BM}(X),M)$$
with
$$E^1_{p,q}=\bigoplus_{y\in X,\hspace{1mm} \delta(y)=p} \colim\limits_{U\in \mathrm{Z}(y)}\Hl^{p+q}\Map_{\DM_{\et}(S,R)}(M_S^{\BM}(X\times_S U),M)$$
and for $\in X$, and for any point $y$ of $X$, $\mathrm{Z}(y)$ denotes the set of open neighborhoods of $y$ in the closure $\overline{\{ y\}}$ of $\{ y\}$.

As the sheaf $M$ is bounded below, \Cref{continuity0} below yields
\begin{align*}\colim\limits_{U\in \mathrm{Z}(y)}\Map_{\DM_{\et}(S,R)}(M_S^{\BM}(X\times_S U),M)&=\Map_{\DM_{\et}(S,R)}(M_S^{\BM}(X\times_S \{ y\}),M)\\
&=\Map_{\DM_{\et}(k(f(y)),R)}(M_{f(y)}^{\BM}(y),\omega^0i_{f(y)}^!M).
\end{align*}


Thus, 
$$E^1_{p,q}=\bigoplus_{y\in X,\hspace{1mm} \delta(y)=p} \Hl^{p+q}\Map_{\DM_{\et}(k(f(y)),R)}(M_{f(y)}^{\BM}(y),\omega^0i_{f(y)}^!M).$$

Since for any point $x$ of $S$, we have $\omega^0i_x^! M\geqslant_{hp} 0$, the complex $$\Map(M_{f(y)}^{\BM}(y),\omega^0i_{f(y)}^!M)$$ is $(-\delta(X)-1)$-connected for any point $y$ of $X$. Hence, if $n>\delta(X),$ the $R$-module $\Hl^{-n}\Map_{\DM_{\et}(S,R)}(M^{\BM}_S(X),M)$ vanishes. Therefore, the Artin motive $M$ is $t_{hp}$-non-negative.

\textbf{Proof of Assertion (2)(a):} 

\cite[3.1.11]{AM1} implies that the functor $$\rho_!\colon \Shv_{\In\lis}(S,R)\rar \DM_{\et}^A(S,R)$$ is right t-exact when the left hand side is endowed with the ordinary t-structure shifted by $-\delta(S)$ and the right hand side is endowed with the perverse homotopy t-structure. Since the scheme $S$ is regular and connected $\delta(x)=-\codim_S(x)$ is a dimension function on $S$. In particular, we can assume that $\delta(S)=0$.

By dévissage, it therefore suffices to show that any étale sheaf which belongs to $\Loc_S(R)$ is send through the functor $\rho_!$ to a $t_{hp}$-non-negative object. Let $M$ be such an étale sheaf. We have an exact triangle $$M\otimes_\Z \Q/\Z[-1]\rar M\rar M\otimes_\Z \Q.$$

The sheaf $M\otimes_\Z \Q$ lies in $\Loc_S(R\otimes_\Z \Q)$. 

Furthermore, if $n$ is a positive integer, the sheaf $M\otimes_\Z \Z/n\Z[-1]$ is t-non-negative with respect to the ordinary t-structure. Let $x$ be a point of $S$. \cite[1.2.4]{AM1} yields $$i_x^!(M\otimes_\Z \Q/\Z[-1])=\colim_n i_x^!(M\otimes \Z/n\Z[-1]).$$

Therefore, since the scheme $S$ is regular, and since the sheaf $M\otimes \Z/n\Z[-1]$ is lisse, we get by absolute purity \cite[XVI.3.1.1]{travauxgabber} that 
$$i_x^!(M\otimes \Z/n\Z[-1])=i_x^*(M\otimes \Z/n\Z[-1])(\delta(x))[2\delta(x)].$$

Therefore, we get $$i_x^!(M\otimes_\Z \Q/\Z[-1])=i_x^*(M\otimes \Q/\Z[-1])(\delta(x))[2\delta(x)].$$ 

Finally, we get $$\omega^0 i_x^!\rho_!(M\otimes \Q/\Z[-1])=\rho_!\left[i_x^*(M\otimes \Q/\Z[-1])(\delta(x))\right][2\delta(x)]$$ where the left hand side in degree $-2\delta(x)$ or more. It is therefore in degree $-\delta(x)$ or more and therefore, the motive $\rho_!(M\otimes \Q/\Z[-1])$ is perverse t-non-negative.


We can therefore assume that the ring $R$ is a $\Q$-algebra. In that case, we claim that $\omega^0i_x^!(M)$ vanishes when $x$ is not the generic point of $S$. This can indeed either be derived from \Cref{omega^0 corps} below or alternatively, the absolute purity property shows that $\omega^0i_x^!$ vanishes on generators of smooth Artin motives which proves that it vanishes on smooth Artin motives.

When $x$ is the generic point of $S$, we get $$\omega^0i_x^!(M)=i_x^*(M).$$ Therefore, the sheaf $\omega^0i_x^!(M)$ is in degree $0=-\delta(x)$. Hence, Assertion (1)(a) ensures that $M$ is perverse homotopy t-non-negative

\textbf{Proof of Assertion (1)(b):} 

If $M\leqslant_{hp} 0$, we already know from \Cref{AM.t-adj} that for any point $x$ of $S$, we have $i_x^* M\leqslant_{hp} 0$. 

We now prove the converse. \Cref{rl} and \Cref{coeff change perverse} ensure that the family $(-\otimes_\Z \Z_{(p)})_{p \text{ prime}}$ is conservative family which is made of $t_{hp}$-exact functors. Using \cite[1.1.6]{AM1}, we can therefore assume that the ring $R$ is a $\Z_{(p)}$-algebra where $p$ is a prime number. 

Write $$S=S[1/p]\sqcup S_{p}$$ where $S_p=S\times_{\Z}\Z/p\Z$. The subscheme $S[1/p]$ of $S$ is open while the subscheme $S_{p}$ is closed. 

Let $M$ be a constructible Artin motive and such that for all point $x$ of $S$, we have $$i_x^*M\leqslant_{hp} 0.$$

\cite[3.5.4]{AM1} ensures that we have a stratification $\mc{S}[1/p]$ of $S[1/p]$ with regular strata and a stratification $\mc{S}_p$ of $S_{p}$ with regular strata, lisse étale sheaves $M_T\in \Shv_{\lis}(T,R[1/p])$ for $T\in \mc{S}_{p}$ and lisse étale sheaves $M_T\in \Shv_{\lis}(T,R)$ for $T\in \mc{S}[1/p]$ such that for all any stratum $T$, we have $$M|_T=\rho_! M_T.$$

\Cref{lemme du point generique} below then ensures that $M|_T\leqslant_{hp} 0$ for each $T\in\mc{S}$. Therefore, using \Cref{weak locality perverse}, we get $$M\leqslant_{hp} 0.$$
\end{proof}

\begin{lemma}\label{continuity0} Let $\mr{D}$ be one of the fibered categories $\DM_{\et}$ or $\Shv(-_{\et})$ and let $R$ be a commutative ring. Consider a scheme $X$ which is the limit of a projective system of schemes with affine transition maps $(X_i)_{i\in I}$.

Let $(M_i)_{i\in I}$ and $(N_i)_{i \in I}$ be two cartesian sections of the fibered category $\mr{D}(-,R)$ over the diagram of schemes $(X_i)_{i\in I}$ and denote by $M$ and $N$ the respective pullback of $M_i$ and $N_i$ along the map $X\rar X_i$. If each $M_i$ is constructible and each $N_i\otimes_\Z \Q/\Z$ is bounded below with respect to the ordinary t-structure on $\mr{D}_{\mathrm{tors}}(X_i,R)$, then the canonical maps 
$$\colim_{i \in I} \Hom_{\mr{D}(X_i,R)}(M_i,N_i) \rar \Hom_{\mr{D}(X,R)}(M,N)$$
$$\colim_{i \in I} \Map_{\mr{D}(X_i,R)}(M_i,N_i) \rar \Map_{\mr{D}(X,R)}(M,N)$$
are equivalences.
\end{lemma}
\begin{proof} Following the proof of \cite[6.3.7]{em}, we can assume that the ring $R$ is a $\Q$-algebra or a $\Z/n\Z$-algebra with $n$ a power of some prime number $p$.

If $R$ is a $\Q$-algebra, the result is \cite[5.2.5]{em} when $\mr{D}=\mc{DM}_{\et}$ and the result follows from \cite[VI.8.5.7]{sga4} when $\mr{D}=\Shv(-_{\et})$.

If $R$ is a $\Z/n\Z$-algebra, we can assume that $p$ (and therefore $n$) is invertible on $S$ by virtue of \cite[A.3.4]{em}. Applying the Rigidity Theorem \cite[3.2]{bachmanrigidity}, we can assume that $\mr{D}=\Shv(-_{\et})$. In this case, the result follows from \cite[IX.2.7.3]{sga4}.
\end{proof}

\begin{lemma}\label{lemme du point generique} Let $S$ be a regular connected scheme and $R$ be a regular ring and let $$\eta\colon\Spec(K)\rar S$$ be the generic point of $S$. Then, the functor $$\eta^*\colon\Shv_{\lis}(S,R)\rar \Shv_{\lis}(K,R)$$ is conservative and t-exact with respect to the ordinary t-structure. 
\end{lemma}
\begin{proof} The functor $\eta^*$ is t-exact.
Assume that $f\colon C\rar D$ is a map in $\Shv_{\lis}(S,R)$ such that $\eta^*(f)$ is invertible. Then, for any integer $n$, the map $\Hl^n(\eta^*(f))$ is invertible. Since the functor $\eta^*$ is t-exact, the map $\eta^*\Hl^n(f)$ is invertible. 

Let $\xi$ be a geometric point with image $\eta$. By \cite[V.8.2]{sga1}, the canonical map $G_K\rar \pi_1^{\et}(S,\xi)$ is surjective. Thus, the map $$\eta^*\colon \Sh(S_{\et},R)\to \Sh(K_{\et},R)$$ is conservative. Therefore, for any integer $n$, the map $\Hl^n(f)$ is invertible. Thus, the map $f$ is invertible. 
\end{proof}

\begin{corollary}
	Let $(S,\delta)$ be a d-scheme with $S$ regular and connected and assume that the residue characteristic exponents of $S$ are invertible in $R$, then the perverse homotopy t-structure induces a t-structure on $\DM^{smA}_{\et,c}(S,R)$ which is the same as its ordinary t-structure shifted by $\delta(S)$.
\end{corollary}


\begin{corollary}\label{delta(S)} Let $S$ be a regular connected scheme and let $R$ be a regular ring. Then, the Artin motive $\un_S$ is in degree $\delta(S)$ with respect to the perverse homotopy t-structure.    
\end{corollary} 




\begin{definition}\label{perverse smooth Artin motives} Let $S$ be a regular scheme and let $R$ be a localization of $\Z$. Assume that the residue characteristic exponents of $S$ are invertible in $R$. Then, \Cref{AM.locality} implies that the perverse homotopy t-structure restricts to $\DM^{smA}_{\et}(S,R)$. 
The \emph{abelian category of perverse smooth Artin motives} $\mathrm{M}^{smA}_{\mathrm{perv}}(S,R)$ is the heart of the induced t-structure.     
\end{definition}

\section{Rational coefficients}
\subsection{The t-structure}
When $R=\Q$, we show that the perverse homotopy t-structure always exists on constructible Artin motives. 

\begin{proposition} Let $(S,\delta)$ be an excellent d-scheme (meaning that $S$ is excellent).
\begin{enumerate}\label{main thm rationnel}\item  Then, the perverse homotopy t-structure on $\DM^A_{\et}(S,\Q)$ induces a t-structure on the subcategory $\DM^A_{\et,c}(S,\Q)$. 
\item Let $i\colon F\rar S$ be a closed immersion and let $j\colon U\rar S$ be the open complement. 
The perverse homotopy t-structure on $\mc{DM}^A_{\et,c}(S,\Q)$ is obtained by gluing the perverse homotopy t-structures of $\mc{DM}^A_{\et,c}(U,\Q)$ and $\mc{DM}^A_{\et,c}(F,\Q)$ (see \cite[1.4.9]{bbd}).
\item Let $p$ be a prime number. Assume that the scheme $S$ is of finite type over $\mb{F}_p$. Then, the $\ell$-adic realization functor $$\rho_\ell\colon\DM^A_{\et,c}(S,\Q)\rar \mc{D}^A(S,\Q_\ell)$$ is t-exact and conservative when the right hand side is endowed with the perverse homotopy t-structure from \cite{aps}.
\end{enumerate}
\end{proposition}
\begin{proof} We can assume that $S$ is reduced. 
	We adapt the ideas of \cite[2.2.10]{bbd} to our setting. This technique is very similar to 
\cite[4.7.1]{aps}.

We first need a few definitions:
\begin{itemize}
    \item If $\mc{S}$ is a stratification of $S$, we say that a locally closed subscheme $X$ of $S$ is an $\mc{S}$-subscheme of $S$ if $X$ is a union of strata of $\mc{S}$.
    \item A pair $(\mc{S},\mc{L})$ is \emph{admissible} if 
        \begin{itemize}
            \item $\mc{S}$  is a stratification of $S$ with regular strata everywhere of the same dimension. 
            \item for every stratum $T$ of $\mc{S}$, $\mc{L}(T)$ is a finite set of isomorphism classes of objects of $\mathrm{M}^{smA}_{\mathrm{perv}}(S,\Q)$ (see \Cref{perverse smooth Artin motives}).
        \end{itemize}
    \item If $(\mc{S},\mc{L})$ is an admissible pair and $X$ is an $\mc{S}$-subscheme of $S$, the category $\mc{DM}_{\mc{S},\mc{L}}(X,\Q)$ of \emph{$(\mc{S},\mc{L})$-constructible Artin motives} over $X$ is the full subcategory of $\mc{DM}^A_{\et}(X,\Q)$ made of those objects $M$ such that for any stratum $T$ of $\mc{S}$ contained in $X$, the restriction of $M$ to $T$ is a smooth Artin motive and its perverse cohomology sheaves are successive extensions of objects whose isomorphism classes lie in $\mc{L}(T)$. 
    \item A pair $(\mc{S}',\mc{L}')$ refines a pair $(\mc{S},\mc{L})$ if every stratum $S$ of $\mc{S}$ is a union of strata of $\mc{S}'$ and any perverse smooth Artin motive $M$ over a stratum $T$ of $\mc{S}$ whose isomorphism class lies in $\mc{L}(T)$ is $(\mc{S}',\mc{L}')$-constructible. 
\end{itemize}

If $(\mc{S},\mc{L})$ is an admissible pair and if $X$ is an $\mc{S}$-subscheme of $S$, the category $\mc{DM}_{\mc{S},\mc{L}}(X,\Q)$ is a stable subcategory of $\mc{DM}^A_{\et}(S,\Q)$ (\textit{i.e.} it is closed under finite (co)limits). 

If $(\mc{S},\mc{L})$ is an admissible pair and  $i\colon U \rar V$ is an immersion between $\mc{S}$-subschemes of $S$, the functors $i_!$ and $i^*$ preserve $(\mc{S},\mc{L})$-constructible objects. 

Now, we say that an admissible pair $(\mc{S},\mc{L})$ is \emph{superadmissible} if letting $i\colon T \rar S$ be the immersion of a stratum $T$ of $\mc{S}$ into $S$ and letting $M$ be a perverse smooth Artin motive over $T$ whose isomorphism class lies in $\mc{L}(T)$, the complex $\omega^0i_*M$ is $(\mc{S},\mc{L})$-constructible. 

We now claim that an admissible pair $(\mc{S},\mc{L})$ can always be refined into a superadmissible one. 

Assume that this is true when we replace $S$ with the union of the strata of dimension $n$ or more. Let $i:T \rar S$ and $i':T'\rar S$ be immersions of strata of $\mc{S}$. Let $M$ be a perverse smooth Artin motive on $T$ whose isomorphism class lies in $\mc{L}(T)$. If $T$ and $T'$ are of dimension at least $n$, then, the perverse cohomology sheaves of $(i')^*\omega^0i_*M$ are obtained as successive extensions of objects of whose isomorphism classes lie in $\mc{L}(T')$ by induction.

If $T=T'$, then, by \cite[3.3]{plh}, we have $$(i')^*\omega^0i_*M=\omega^0 i^*\omega^0i_*M=\omega^0i^*i_*M=M$$ whose isomorphism class lies in $\mc{L}(T)$. 

Otherwise, if $T$ is of dimension $n-1$, we can always replace $T$ with an open subset, so that the closure of $T$ is disjoint from $T'$. Then, by \cite[3.3]{plh}, we have
$$(i')^*\omega^0i_*M=\omega^0(i')^*i_*M=0.$$ 

Assume now that $T$ is of dimension $n$ or more and that $T'$ is of dimension $n-1$. We can replace $T'$ with an open subset so that for any $M$ in $\mc{L}(T)$, the motive $(i')^*\omega^0i_*M$ is smooth Artin by \Cref{AM.six foncteurs constructibles} and \cite[3.5.1]{AM1}. We can also add all the isomorphism classes of the $\Hlh^n\left((i')^*\omega^0i_*M\right)$ to $\mc{L}(T')$ as there are finitely many of them. 

Hence, any admissible pair can be refined into a superadmissible one.

If $(\mc{S},\mc{L})$ is a superadmissible pair and if $i\colon U \rar V$ is an immersion between $\mc{S}$-subschemes of $S$, we claim that the functors $\omega^0i_*$ and $\omega^0i^!$ preserve $(\mc{S},\mc{L})$-constructibility. The proof is the same as \cite[2.1.13]{bbd} using as above the commutation of $\omega^0$ with the six functors described in \cite[3.3]{plh}.

Hence, if $(\mc{S},\mc{L})$ is a superadmissible pair, we can define a t-structure on $(\mc{S},\mc{L})$-constructible objects by gluing the perverse homotopy t-structures on the strata. By \Cref{AM.locality}, the positive (resp. negative) objects of this t-structure are the positive (resp. negative) objects of the perverse homotopy t-structure that are $(\mc{S},\mc{L})$-constructible. 

Thus, the perverse homotopy t-structure induces a t-structure on $(\mc{S},\mc{L})$-constructible objects. Therefore, the subcategory of $(\mc{S},\mc{L})$-constructible objects is stable under the truncation functors of the perverse homotopy t-structure. 

Any object $M$ of $\mc{DM}^A_{\et}(S,\Q)$ is $(\mc{S},\mc{L})$-constructible for some superadmissible pair $(\mc{S},\mc{L})$. Indeed, by \cite[3.5.1]{AM1}, there is an admissible pair $(\mc{S},\mc{L})$ such that $M$ is $(\mc{S},\mc{L})$-constructible. Such a pair is can always be refined into one which is superadmissible. 

Hence, $\mc{DM}^A_{\et}(S,\Q)$ is stable under the truncation functors of the perverse homotopy t-structure. Therefore, the perverse homotopy t-structure induces a t-structure on $\mc{DM}^A_{\et}(S,\Q)$.

Now, (2) follows from \Cref{glueing perverse}. Finally the third assertion follows from the construction of the perverse homotopy t-structure on constructible objects that we used to prove the first assertion, from the same construction in the $\ell$-adic setting of 
\cite[4.7.1]{aps} 
and from the fact that the functor $\omega^0$ commutes with the $\ell$-adic realization functor by 
\cite[4.6.3]{aps}.
\end{proof}

\begin{remark} Another possible proof of the first assertion is to show that the perverse homotopy t-structure can be constructed on constructible Artin motives by using Vaish's formalism of punctual gluing \cite[3]{vaish}.
\end{remark}

\subsection{Artin vanishing for \texorpdfstring{$f_!$}{f!}}
We now prove our first version of Artin's vanishing Theorem for Artin motives. It is an analog of \cite[XIV.3.1]{sga4}, \cite[XV.1.1.2]{travauxgabber} and \cite[4.1.1]{bbd} in the setting of Artin motives. 

\begin{theorem}\label{lefschetz affine 2Q} Let $(S,\delta)$ be an excellent scheme, let $f\colon X\rar S$ be a quasi-finite affine morphism. Assume that 
$X$ is nil-regular and that we are in one of the following cases 
\begin{enumerate}[label=(\alph*)]
    \item We have $\dim(S)\leqslant 2.$
    \item There is a prime number $p$ such that the scheme $S$ is of finite type over $\mb{F}_p$. 
\end{enumerate}

Then, the functor $$f_!\colon \DM^{smA}_{\et,c}(X,\Q)\rar \DM^A_{\et,c}(S,\Q)$$ is $t_{hp}$-exact. 
\end{theorem}
\begin{proof}
	This is the combination of \Cref{Ass_b} and \Cref{Ass_a}
\end{proof}
\begin{lemma}\label{Ass_b}
	\Cref{lefschetz affine 2Q} holds when Assumption (b) holds.
\end{lemma}
\begin{proof}
	Suppose that Assumption (b) holds. 
We have a commutative diagram
$$\begin{tikzcd} \DM^{smA}_{\et,c}(U,\Q)\ar[r,"j_!"]\ar[d,"\rho_{\ell}"] & \DM^A_{\et,c}(S,\Q)\ar[d,"\rho_{\ell}"] \\
\mc{D}^{smA}(U,\Q_\ell)\ar[r,"j_!"] & \mc{D}^A(S,\Q_\ell)
\end{tikzcd}$$
such that the vertical arrows are conservative and t-exact when the categories which appear in the bottom row are endowed with the perverse homotopy t-structure. Using \cite[1.1.6]{AM1}, it suffices to show that the functor $j_!$ is t-exact which follows from 
\cite[4.8.12]{aps}.
\end{proof}
We continue with the case when $\dim(S)\leqslant 1$. To that end, we study an analog of the nearby cycle functor (or more accurately residue functor) in our setting. Namely, let $C$ be an excellent $1$-dimensional scheme, let $i\colon F\rar C$ be a closed immersion of positive codimension and let $j\colon U\rar C$ be the complementary open immersion. Assuming that $U$ is nil-regular, we study the functor

$$\omega^0i^*j_*\colon \DM^{smA}_{\et,c}(U,\Q)\rar \DM^A_{\et}(F,\Q).$$

First, notice that the scheme $F$ is $0$-dimensional. Therefore, the scheme $F$ is the disjoint union of its points. Therefore, denoting by $$i_x\colon\{ x\}\rar F$$ the inclusion for $x\in F$, we get $$\omega^0i^*j_*=\bigoplus\limits_{x\in F} (i_x)_* i_x^*\omega^0 i^*j_*.$$
Now, by \cite[3.3]{plh}, the $2$-morphism $i_x^*\omega^0 \rar \omega^0 i_x^*$ is an equivalence. Therefore, the functor $\omega^0i^*j_*$ is equivalent to the functor $$\bigoplus\limits_{x\in F} (i_x)_* \omega^0 (i\circ i_x)^*j_*.$$ 
Hence, studying the case where $F$ is a point will yield a description in the general case. Assume from now on that $F=\{ x\}$. 

From \Cref{smooth Artin motives}, we can deduce that 
through the equivalence $\rho_!$ of \Cref{small et site}, the functor 
$\omega^0i^*j_*$ can be seen as a functor
$$\Shv_{\lis}(U,\Q)\rar \Shv_{\lis}(k(x)_{\et},\Q).$$

We have an equivalence $\rho_! i^*\overset{\sim}{\rar} i^*\rho_!$ and an exchange transformation $\rho_! j_* \rar j_* \rho_!$ by \cite[4.4.2]{em}.
This gives a transformation $$\rho_!i^*j_*\rar i^*j_*\rho_!$$ of functors $\Shv_{\lis}(U,\Q)\rar \DM_{\et}(F,\Q).$

Furthermore, if $M$ is a lisse étale sheaf over $U$, then, the motive $\rho_!i^*j_* M$ is Artin by \Cref{small et site}. Therefore, we have a transformation $${\Theta}\colon\rho_!i^*j_*\rar \omega^0 i^*j_*\rho_!.$$
of functors $\Shv_{\lis}(U,\Q)\rar \DM^A_{\et,c}(F,\Q).$


\begin{lemma}\label{nearby cycle} 
	Let $C$ be an excellent $1$-dimensional scheme, let $x$ be a closed point of $C$, let $i\colon\{ x\}\rar C$ be the closed immersion and $j\colon U\rar C$ be the complementary open immersion and let $p$ be the characteristic exponent of $k(x)$. 


Then, the transformation $\Theta$
is an equivalence. In particular, the square $$\begin{tikzcd}\Shv_{\lis}(U,\Q)\ar[r,"\rho_!"] \ar[d,swap,"i^*j_*"]& \DM_{\et,c}^A(U,\Q)\ar[d,"\omega^0i^*j_*"] \\
\Shv_{\lis}(\{ x\},\Q)\ar[r,"\rho_!"]& \DM_{\et,c}^A(\{ x\},\Q)
\end{tikzcd}$$
is commutative.
\end{lemma}
\begin{proof} 



Let $\nu\colon\widetilde{C}\rar C$ be the normalization of $C$. 
Write
$$\begin{tikzcd}
U_\mathrm{red} \arrow[d, swap,"\nu_U"] \arrow[r, "\gamma"] & \widetilde{C} \arrow[d, "\nu"] & F \arrow[d, "\nu_F"] \arrow[l, swap,"\iota"] \\
U \arrow[r, "j"]                               & C                              & \{ x\} \arrow[l, swap, "i"]                  
\end{tikzcd}$$
the commutative diagram such that both squares are cartesian. Notice that $U_\mathrm{red}$ is the reduced subscheme associated to $U$. Therefore, the functor  $(\nu_U)_*$ is an equivalence and commutes with the functor $\rho_!$. Hence, it suffices to show that $$\Theta (\nu_U)_*:\rho_!i^*j_*(\nu_U)_*\rar \omega^0 i^*j_*(\nu_U)_*\rho_!$$ is an equivalence. 
Therefore, it suffices to show that the map $\rho_!i^*\nu_*\gamma_*\rar \omega^0 i^*\nu_* \gamma_*\rho_!$ is an equivalence.

Using proper base change, this is equivalent to the fact that the map $$\rho_!(\nu_F)_*\iota^*\gamma_*\rar \omega^0 (\nu_F)_* \iota^*\gamma_*\rho_!$$
is an equivalence. Since the map induced by $\nu_F$ over the reduced schemes is finite, the functor $(\nu_F)_*$ commutes with the functor $\rho_!$ by \Cref{442em} below and with the functor $\omega^0$ by \cite[3.3]{plh}. Thus, it suffices to show that the map
$$\rho_!\iota^*\gamma_*\rar \omega^0 \iota^*\gamma_*\rho_!$$ is an equivalence. 
Moreover, since the scheme $F$ is $0$-dimensional, we can assume as before that $F$ is a point. Therefore, we can assume that the $1$-dimensional scheme $C$ is normal and therefore regular.

Since by \Cref{smooth Artin motives} the stable category $\Shv_{\lis}(U,\Q)$ is generated as a thick subcategory of itself by the sheaves of the form $\Q_U(V)=f_*\underline{\Q}$ for $f\colon V\rar U$ finite and étale and since the functors $\rho_!i^*j_*$ and $\omega^0 i^*j_*$ are exact and therefore compatible with finite limits, finite colimits and retracts, it suffices to show that the map $$\Theta(\Q_U(V))\colon\rho_!i^*j_*f_*\underline{\Q} \rar \omega^0 i^*j_*f_*\un_V$$ is an equivalence.
To that end, we introduce $\overline{f}\colon C_V\rar C$ the normalization of $C$ in $V$. 
Write 
$$\begin{tikzcd}
V \arrow[d, "f"] \arrow[r, "\gamma"] & C_V \arrow[d, "\overline{f}"] & Z \arrow[d, "p"] \arrow[l, swap, "\iota"] \\
U \arrow[r, "j"]                               & C                              & \{ x\} \arrow[l, swap, "i"]                  
\end{tikzcd}$$
the commutative diagram such that both squares are cartesian. The map $\Theta(\Q_U(V))$ is an equivalence if and only if the map
$$\rho_!i^*\overline{f}_*\gamma_*\underline{\Q} \rar \omega^0 i^*\overline{f}_*\gamma_*\un_V$$
is an equivalence. The morphism $p$ is finite 
and therefore the functor $p_*$ commutes with the functor $\rho_!$ by \Cref{442em} and with the functor $\omega^0$ by \cite[3.3]{plh}. Hence, using proper base change, it suffices to show that the map
$$\rho_!\iota^*\gamma_*\underline{\Q} \rar \omega^0 \iota^*\gamma_*\un_V$$
is an equivalence.

Hence, we can assume that $V=U$. Therefore, we need to prove that the map
$$\Theta(\underline{\Q})\colon\rho_! i^*j_*\underline{\Q}\rar \omega^0i^*j_*\un_U$$ is an equivalence. Using \cite[3.3]{plh}, the exchange transformation $$i^*\omega^0j_*\rar \omega^0i^*j_*$$ is an equivalence. Furthermore, by \cite[4.4.2]{em}, the exchange transformation $\rho_! i^*\rar i^*\rho_!$ is an equivalence. Therefore, it suffices to show that the exchange map
$$\rho_! j_* \underline{\Q} \rar \omega^0 j_* \un_U$$ is an equivalence which is a direct consequence of the two lemmas below.
\end{proof}
\begin{lemma}\label{az211} 
	Let $S$ be a regular scheme, and $j\colon U\rar S$ an open immersion with complement a simple normal crossing divisor. Then, the canonical map $$\un_S \rar \omega^0 j_*\un_U$$ in $\DM_{\et}^A(S,\Q)$ is an equivalence.
\end{lemma}
\begin{proof} This result generalizes \cite[2.11]{az}. Let $i\colon D\rar S$ be the closed immersion which is complementary to $j$. Using \Cref{AM.six foncteurs constructibles}, we have an exact triangle $$i_*\omega^0i^!\un_S\rar \un_S\rar \omega^0j_*\un_U.$$
It therefore suffices to show that $\omega^0i^!\un_S=0.$
Using \cite[5.7]{ddo}, we can reduce to the case when $D$ is a regular subscheme of codimension $c>0$. In this setting, we have $$\omega^0i^!\un_S=\omega^0(\un_D(-c))[-2c]$$
and the latter vanishes by \cite[3.9]{plh}.
\end{proof}
\begin{lemma}\label{az211 etale coh} 
	Let $S$ be a normal scheme, and let $j\colon U\rar S$ be an open immersion. Then, the canonical map $$\underline{\Q} \rar j_*\underline{\Q}$$ is an equivalence.
\end{lemma}
\begin{proof} The canonical map $\underline{\Q} \rar j_*\underline{\Q}$ is an equivalence if and only if for any geometric point $\xi$ of $S$, the canonical map $$\Q \rar \xi^*j_*\underline{\Q}$$ is an equivalence. If $n$ is an integer, \cite[VIII.5]{sga4} yields  $$\mathrm{H}^n(\xi^*j_*\underline{\Q})=\mathrm{H}^n_{\et}(\Spec(\mc{O}^{sh}_{S,x}) \times_S U,\Q)$$ where $\mc{O}^{sh}_{S,x}$ is the strict henselization of the local ring $\mc{O}_{S,x}$ of $S$ at $x$. 
Since the scheme $S$ is normal, so is the scheme $\Spec(\mc{O}^{sh}_{S,x}) \times_S U$. Moreover, this scheme is also connected. Therefore \cite[2.1]{deninger} yields the result.
\end{proof}

\begin{lemma}\label{442em} Let $R$ be a commutative ring. Recall the functor $\rho_!$ from \Cref{small et site}. Then, for any quasi-finite morphism $f$, there is a functorial isomorphism $$\rho_!f_!\rar f_!\rho_!.$$ 

\end{lemma}
\begin{proof} \cite[4.4.2]{em} yields the assertion in the case of étale morphisms. Therefore, using Zariski's Main Theorem \cite[18.12.13]{ega4}, it suffices to prove that for any finite morphism $f$, there is a canonical isomorphism $$\rho_!f_*\rar f_*\rho_!.$$

Assume first that $f\colon X\rar S$ is a closed immersion and let $j\colon U\rar S$ be the complementary open immersion. The localization triangles \eqref{AM.localization} in the fibered category of étale sheaves and in the fibered category of étale motives yield a morphism of exact triangles
$$\begin{tikzcd}
 \rho_!j_!j^*\ar[r]\ar[d] & \rho_! \ar[d,equals] \ar[r]& \rho_! f_* f^* \ar[d] \\
 j_!j^*\rho_!\ar[r] & \rho_! \ar[r]& f_* f^* \rho_!
\end{tikzcd}.$$

The left vertical arrow is an equivalence using the first assertion and the third assertion in the case of the map $j$. Therefore, we get an isomorphism $$\rho_! f_* f^*\rar f_* f^*\rho_!$$ and therefore, using the first assertion, we get an equivalence $$\rho_! f_* f^*\rar f_* \rho_!f^*.$$ Since the functor $f^*$ is essentially surjective, this yields the result for the map $f$.

If the map $f$ is purely inseparable the functor $f_*$ is an equivalence with inverse $f^*$ using \cite[6.3.16]{em} in the motivic setting and using the topological invariance of the small étale site \cite[04DY]{stacks} in the case of étale sheaves. Since the functors $\rho_!$ and $f^*$ commute, so do the functors $\rho_!$ and $f_*$.

Assume now that $f\colon X\rar S$ is a general finite morphism. We prove the claim by noetherian induction on $S$. We can assume that the scheme $X$ is reduced using the case of closed immersions. 

Let $\Gamma$ be the subscheme of $S$ made of its generic points. The pullback map $X\times_S \Gamma\rar \Gamma$ is finite over a finite disjoint union of spectra of fields and can therefore be written as a composition $$X\times_S\Gamma\xrightarrow{g}\Gamma'\xrightarrow{h}\Gamma$$ where $g$ is finite étale and $h$ is purely inseparable. Hence, by \cite[8.8.2]{ega4}, if $j\colon U\rar S$ is a small enough dense open immersion, the pullback map $f_U\colon X\times_S U\rar U$ can be written as a composition $$X\times_S U\xrightarrow{g_U}U'\xrightarrow{h_U}U$$ such that the pullback of $g_U$ (resp. $h_U$) to $\Gamma$ is $g$ (resp. $h$). 

Shrinking $U$ further, we can assume that the map $g_U$ is finite étale by \cite[8.10.5(x)]{ega4} and \cite[1.A.5]{ayo15} and that the map $h_U$ is purely inseparable by \cite[8.10.5(vii)]{ega4}. Hence, there is a dense open subscheme $U$ such that the map $f_U\colon X\times_S U\rar U$ is the composition of a finite étale map and a purely inseparable morphism. Therefore, the transformation $$\rho_! (f_U)_*\rar (f_U)_* \rho_!$$ is an equivalence. 

Let $i\colon Z\rar S$ be the reduced closed immersion which is complementary to $j$ and let $f_Z\colon X\times_S Z\rar Z$ be the pullback map. By noetherian induction, the transformation $$\rho_! (f_Z)_*\rar (f_Z)_* \rho_!$$ is an equivalence.

The localization triangles \eqref{AM.localization} in the fibered category of étale sheaves and in the fibered category of étale motives yield a morphism of exact triangles
$$\begin{tikzcd}
 \rho_!j_!j^*f_*\ar[r]\ar[d] & \rho_!f_* \ar[d] \ar[r]& \rho_! i_* i^*f_* \ar[d] \\
 j_!j^*f_*\rho_!\ar[r] & f_*\rho_! \ar[r]& i_* i^*f_* \rho_!
\end{tikzcd}.$$

The discussion above and proper base change ensure that the leftmost and rightmost vertical arrows are equivalences. Therefore, the middle vertical arrow is an equivalence which finishes the proof.
\end{proof}

\begin{lemma}\label{lefschetz affine 1Q} \Cref{lefschetz affine 2Q} holds when $\dim(S)\leqslant 1$.
\end{lemma}
\begin{proof} 
\Cref{AM.t-adj2} implies that when $g$ is a finite morphism, the functor $g_!$ is perverse homotopy t-exact . When the scheme $S$ is $0$-dimensional, the morphism $f$ is finite and the result holds.

Assume that the scheme $S$ is $1$-dimensional. Zariski's Main Theorem \cite[18.12.13]{ega4} provides a finite morphism $g$ and an open immersion $j\colon U\rar X$ such that $f=g\circ j$. Therefore it suffices to show that the functor $$j_!\colon \DM^{A}_{\et,c}(U,\Q)\rar \DM^A_{\et,c}(X,\Q)$$ is $t_{hp}$-exact. 
We already know by \Cref{AM.t-adj} that $j_!$ is right t-exact.

Hence, letting $M$ be a constructible Artin motive over $U$ which is $t_{hp}$-non-negative, it suffices to show that the Artin motive $j_!(M)$ is $t_{hp}$-non-negative.

Let $x$ be a point of $X$. Using \Cref{AM.locality}, we can replace $X$ with a neighborhood of $x$. Since the scheme $X$ is $1$-dimensional, we can therefore assume that the complement of $U$ is the closed subscheme $\{ x\}$ and replace $U$ with any neighborhood of the generic points. In particular, we can assume that $U$ is nil-regular and that $M$ is smooth Artin using the continuity property of \cite[1.4.2]{AM1}.

Let $i\colon \{ x\}\rar X$ be the canonical immersion. Applying the localization triangle \eqref{AM.localization} to the Artin motive $\omega^0 j_*(M)$ yields an exact triangle $$j_!(M) \rar \omega^0 j_*(M) \rar i_* i^* \omega^0 j_*(M).$$

By \Cref{AM.t-adj}, the motive $\omega^0 j_*(M)$ is $t_{hp}$-non-negative. Furthermore, \cite[3.3]{plh} ensures that $$i^* \omega^0 j_*(M)= \omega^0 i^* j_*(M).$$ Using \Cref{nearby cycle}, the Artin motive $\omega^0 i^* j_*(M)$ is in degree at least $\delta(X)$ with respect to the ordinary homotopy t-structure. Therefore, we get $$\omega^0 i^* j_*(M)\geqslant_{hp} -1.$$ 

This ensures that the Artin motive $j_!(M)$ is perverse homotopy t-non-negative.
\end{proof}

\begin{lemma}\label{Ass_a}
	\Cref{lefschetz affine 2Q} holds when $\dim(S)\leqslant 2$.
\end{lemma}

\begin{proof} We can assume that the scheme $S$ is connected. As in the proof of \Cref{lefschetz affine 1Q}, it suffices to show that if $j\colon U\rar S$ is an affine open immersion such that the scheme $U$ is nil-regular and if $M$ is a $t_{hp}$-non-negative smooth Artin motive over $U$, the Artin motive $j_!M$ is $t_{hp}$-non-negative.

We now suppose that Assumption (a) holds \textit{i.e.} that $\dim(S)\leqslant 2$. By \Cref{lefschetz affine 1Q}, we can assume that $\dim(S)=2$.
Take the convention that $\delta(S)=2$. 
Let $i\colon F\rar S$ be the reduced complementary closed immersion of $j$. As in the proof of \Cref{lefschetz affine 1Q}, it suffices to show that the functor $\omega^0i^*j_*$ has cohomological amplitude bounded below by $-1$. 

\textbf{Step 1:} In this step, we show that we can assume that the scheme $S$ is normal.

Let $\nu\colon \hat{S}\rar S$ be the normalization of $S$. Write 
$$\begin{tikzcd}
U_{\mathrm{red}} \arrow[d, swap, "\nu_U"] \arrow[r, "\gamma"] & \hat{S} \arrow[d, "\nu"] & \hat{F} \arrow[d, "\nu_F"] \arrow[l, swap, "\iota"] \\
U \arrow[r, "j"]                               & S                              & F \arrow[l, swap,"i"]                 
\end{tikzcd}$$
the commutative diagram such that both squares are cartesian. Notice that $U_{\mathrm{red}}$ is the reduced scheme associated to $U$ and therefore, the functor $(\nu_U)_*$ is a t-exact equivalence. Hence, it suffices to show that the functor $\omega^0 i^*j_*(\nu_U)_*$ has cohomological amplitude bounded below by $-1$. This functor is equivalent to $\omega^0 (\nu_F)_* \iota^*\gamma_*$. Since $\nu_F$ is a finite morphism, the functor $(\nu_F)_*$ is $t_{hp}$-exact by \Cref{AM.t-adj2} and commutes with $\omega^0$ by \cite[3.3]{plh}. Therefore, suffices to show that the functor $\omega^0\iota^*\gamma_*$ has cohomological amplitude bounded below by $-1$. Therefore, we can assume that the scheme $S$ is normal.

\textbf{Step 2:} In this step, we show that it suffices to show that $\omega^0i^*j_*\un_U\geqslant_{hp} 1$.

Using \Cref{boundedness}, the objects of $\DM^{smA}_{\et,c}(U,\Q)$ are bounded with respect to the perverse homotopy t-structure. Therefore, it suffices to show that the objects of the heart of the t-category $\DM^{smA}_{\et,c}(U,\Q)$ are sent to objects of degree at least $-1$ with respect to the perverse homotopy t-structure. Letting $\xi$ be a geometric point of $U$, the heart of $\DM^{smA}_{\et,c}(U,\Q)$ is equivalent to the category $\Rep^A(\pi_1^{\et}(U,\xi),\Q)$ of Artin representations of $\pi_1^{\et}(U,\Q)$ by \Cref{smooth Artin motives} and \cite[3.1.7]{AM1}. The latter a semi-simple category by Maschke's Theorem. Thus, every object of the heart is a retract of an object of the form $f_* \un_V[2]$ with $f\colon V\rar U$ finite and étale. Therefore, it suffice to show that $\omega^0 i^*j_*f_*\un_V\geqslant_{hp} 1$.

Let $\overline{f}\colon S_V\rar S$ be the normalization of $S$ in $V$. Write 
$$\begin{tikzcd}
V \arrow[d, swap, "f"] \arrow[r, "\gamma"] & S_V \arrow[d, "\overline{f}"] & F_V \arrow[d, "p"] \arrow[l, swap, "\iota"] \\
U \arrow[r, "j"]                               & S                           & F \arrow[l, swap,"i"]                 
\end{tikzcd}$$
the commutative diagram such that both squares are cartesian. By proper base change, we get $$\omega^0 i^*j_*f_*\un_V= \omega^0 p_* \iota^*\gamma_*\un_V.$$ 

Since the map $p$ is finite, the functor $p_*$ is $t_{hp}$-exact by \Cref{AM.t-adj2} and commutes with the functor $\omega^0$ by \cite[3.3]{plh}. Therefore, it suffices to show that $$\omega^0 \iota^*\gamma_*\un_V\geqslant_{hp}1.$$ Hence, we can assume that $V=U$.

\textbf{Step 3:} Since the scheme $S$ is normal its singular points are in codimension $2$.  Let $x$ be a point of $S$. Using \Cref{AM.locality}, we can replace $S$ with a neighborhood of $x$. We can therefore assume that the singular locus of $S$ is either empty or the single point $x$. Let $f\colon \widetilde{S}\rar S$ be a resolution of singularities of $S$ such that $E$ is a simple normal crossing divisor. Since the scheme $S$ is excellent, such a resolution exists by Lipman's Theorem on embedded resolution of singularities (see \cite[0BGP,0BIC,0ADX]{stacks}).  Write

$$\begin{tikzcd}
U \arrow[d, equal] \arrow[r, "\gamma"] & \widetilde{S} \arrow[d, "f"] & E \arrow[d, "p"] \arrow[l, "\iota"] \\
U \arrow[r, "j"]                               & S                           & F \arrow[l, swap,"i"]                 
\end{tikzcd}$$
the commutative diagram such that both squares are cartesian.

Then, by proper base change, we get $$\omega^0i^*j_*\un_U=\omega^0p_*\iota^*\gamma_*\un_U.$$ The latter is by \cite[3.3]{plh} equivalent to $\omega^0p_*\iota^*\omega^0\gamma_*\un_U$ which by \Cref{az211} is equivalent to $\omega^0p_*\un_{E}$. 
Therefore, it suffices to show that $\omega^0p_*\un_E\geqslant_{hp} 1$.

Write $E=\bigcup\limits_{i \in J} E_i$ with $J$ finite, where the $E_i$ are regular connected and of codimension $1$, the $E_{ij}=E_i\cap E_j$ are of codimension $2$ and regular if $i\neq j$ and the intersections of $3$ distinct $E_i$ are empty. By cdh-descent, we have an exact triangle
$$\omega^0p_*\un_E\rar \bigoplus\limits_{i \in J} \omega^0 (p_i)_*\un_{E_i}\rar \bigoplus\limits_{\{ i,j\} \subseteq J} (p_{ij})_* \un_{E_{ij}}.$$

Since we have $\delta(E_i)=1$ for any $i$ and since $\delta(E_{ij})=0$ if $i\neq j$, \Cref{delta(S),AM.t-adj} imply that the Artin motives $\omega^0 (p_i)_*\un_{E_i}$ and $ (p_{ij})_* \un_{E_{ij}}$ are $t_{hp}$-non-negative. Hence, we get $\omega^0p_*\un_E\geqslant_{hp} 0$ and that $\Hlh^0(\omega^0p_*\un_E)$ is the kernel of the map $$\bigoplus\limits_{i \in J} \Hlh^0(\omega^0 (p_i)_*\un_{E_i})\rar \bigoplus\limits_{\{ i,j\} \subseteq J} (p_{ij})_* \un_{E_{ij}}.$$
Hence, to finish the proof, it suffices to show that the kernel of the above map vanishes.

Let $$F_0=\{ y \in F \mid \dim(f^{-1}(y))>0\}.$$ The scheme $F_0$ is $0$-dimensional. Since we can work locally around $x$, we may assume that $F_0$ is either empty or the single point $x$ and that the image of any $E_{ij}$ through the map $p$ is $\{ x\}$.

Assume that the scheme $F_0$ is empty, then, the morphism $p$ is finite. Furthermore, for any index $i$ in $J$, the Artin motive $\un_{E_i}$ is in degree $1$ with respect to the perverse homotopy t-structure by \Cref{delta(S)}, therefore, the Artin motive $\Hlh^0(\omega^0 (p_i)_*\un_{E_i})$ vanishes. Thus, the motive $\Hlh^0(\omega^0p_*\un_E)$ vanishes.

Assume now that $F_0=\{ x\}$. Let $$I=\{i \in J\mid p(E_i)=\{ x\} \}$$ and let $k\colon \{ x \}\rar F$ be the inclusion. If an index $i$ belongs to $ J\setminus I$, the morphism $p_i$ is finite and therefore, the motive $\Hlh^0(\omega^0 (p_i)_*\un_{E_i})$ vanishes.

If an index $i$ belongs to $ I$, let $$E_i\rar G_i \overset{g_i}{\rar} \{ x\}$$ be the Stein factorization of $p_i$. \Cref{omega^0 corps} yields $$\Hlh^0(\omega^0 (p_i)_*\un_{E_i})=k_* (g_i)_*\un_{G_i}$$ and therefore $$\bigoplus\limits_{i \in I} \Hlh^0(\omega^0 (p_i)_*\un_{E_i})=k_*\bigoplus\limits_{i \in I} (g_i)_*\un_{G_i}.$$

Now, if $i\neq j$, the morphism $p_{ij}$ factors through $\{ x\}$, so that we have a morphism $g_{ij}\colon E_{ij}\rar \{ x\}$ with $p_{ij}=k\circ g_{ij}$.

We have 
$$\bigoplus\limits_{\{ i,j\} \subseteq J} (p_{ij})_* \un_{E_{ij}}
=k_*\bigoplus\limits_{\{ i,j\} \subseteq J} (g_{ij})_* \un_{E_{ij}}.$$



Hence, the Artin motive $\Hlh^0(\omega^0p_*\un_E)$ is the kernel of the map $$k_*\bigoplus\limits_{i \in I} (g_i)_*\un_{G_i} \rar k_*\bigoplus\limits_{\{ i,j\} \subseteq J} (g_{ij})_* \un_{E_{ij}}.$$

Since the functor $k_*$ is t-exact, we get $\Hlh^0(\omega^0p_*\un_E)=k_* P$ with $P$ the kernel of the map $$\bigoplus\limits_{i \in I} (g_i)_*\un_{G_i} \rar \bigoplus\limits_{\{ i,j\} \subseteq J} (g_{ij})_* \un_{E_{ij}}.$$

Let $\kappa=k(x)$ and let $\Gamma=G_{\kappa}$. Recall that if $g\colon X\rar \Spec(\kappa)$ is a finite morphism, the Artin motive $g_*\un_X$ corresponds to the Artin representation $\Q[X_{\overline{\kappa}}]$ of $\Gamma$ through $\rho_!$. 
Hence, we have  $$P=\rho_!P'$$ where $P'$ is the kernel of the morphism

$$\bigoplus\limits_{i \in I} \Q[(G_i)_{\overline{\kappa}}] \rar \bigoplus\limits_{\{ i,j\} \subseteq J} \Q[(E_{ij})_{\overline{\kappa}}]$$ of Artin representations of  $\Gamma$.

Now, we have  $P'=P_0\otimes_{\Z} \Q$ where $P_0$ is the kernel of the map $$\bigoplus\limits_{i \in I} \Z[(G_i)_{\overline{\kappa}}] \rar \bigoplus\limits_{\{ i,j\} \subseteq J} \Z[(E_{ij})_{\overline{\kappa}}].$$
Since the underlying $\Z$-module of $P_0$ is of finite type, 
it suffices to show that $P_0$ is of rank $0$.
Denote by $N$ the $\Z[\Gamma]$-module $\bigoplus\limits_{i \in I} \Z[(G_i)_{\overline{\kappa}}]$, by $Q$ the $\Z[\Gamma]$-module $\bigoplus\limits_{\{ i,j\} \subseteq J} \Z[(E_{ij})_{\overline{\kappa}}]$ and by $R$ the image of $N$ through the map $N\rar Q$. 
We have an exact sequence $$0\rar P_0\rar N \rar R \rar 0.$$

Let $\ell$ be a prime number, let $n$ be a positive integer and denote by $\Lambda$ the ring $\Z/\ell^n\Z$. We get an exact sequence 
$$\mathrm{Tor}^1_{\Z}(R,\Lambda)\rar P_0\otimes_\Z \Lambda \rar N \otimes_\Z \Lambda \rar R \otimes_\Z \Lambda\rar 0.$$

Assume that the map $N\otimes_\Z \Lambda\rar Q\otimes_\Z \Lambda$ is injective. Then, the induced map $N\otimes_\Z \Lambda\rar R\otimes_\Z \Lambda$ is also injective and we get a surjection $$\mathrm{Tor}^1_{\Z}(R,\Lambda)\rar P_0\otimes_\Z \Lambda.$$
Since $\mathrm{Tor}^1_{\Z}(R,\Lambda)$ is the $\ell^n$-torsion subgroup of $R$ and since the latter is of finite type, there is an integer $n_0$ such that if $n\geqslant n_0$, the group $\mathrm{Tor}^1_{\Z}(R,\Lambda)$ is of $\ell^{n_0}$-torsion. If $P_0$ is not of rank $0$, the group $P_0\otimes_\Z \Lambda$ is not of $\ell^{n_0}$-torsion for $n>n_0$. 

Hence, to show that $P_0$ is of rank $0$ and to finish the proof, it suffices to show that there is a prime number $\ell$ such that for any positive integer $n$, letting $\Lambda=\Z/\ell^n\Z$, the map $$\bigoplus\limits_{i \in I} \Lambda[(G_i)_{\overline{\kappa}}] \rar \bigoplus\limits_{\{ i,j\} \subseteq J} \Lambda[(E_{ij})_{\overline{\kappa}}]$$ has a trivial kernel.

\textbf{Step 4:} Let $\ell$ be a prime number which is invertible in $\kappa$. Replacing $S$ with a neighborhood of $x$ if needed, we can assume that $\ell$ is invertible on $S$. Let $n$ be a positive integer. Denote by $\Lambda$ the ring $\Z/\ell^n\Z$. In this step, we show that the map $$\bigoplus\limits_{i \in I} \Lambda[(G_i)_{\overline{\kappa}}] \rar \bigoplus\limits_{\{ i,j\} \subseteq J} \Lambda[(E_{ij})_{\overline{\kappa}}]$$ has trivial kernel.

The classical Artin Vanishing Theorem \cite[XV.1.1.2]{travauxgabber} on $\mc{D}^b_c(S,\Lambda)$ endowed with its perverse t-structure ensures that $${}^{p} \Hl^0(i^*j_*\underline{\Lambda})=0.$$

As before, the sheaf $i^*j_*\underline{\Lambda}$ is equivalent to the sheaf $p_* \iota^*\gamma_*\underline{\Lambda}$. Now, the localization triangle in $\mc{D}^b_c(S,\Lambda)$ yields an exact triangle
$$\iota_* \iota^! \underline{\Lambda} \rar \underline{\Lambda} \rar \gamma_*\underline{\Lambda}.$$
Applying $p_* \iota^*$, we get an exact triangle
$$p_* \iota^! \underline{\Lambda} \rar p_*\underline{\Lambda} \rar p_* \iota^*\gamma_*\underline{\Lambda}.$$
Furthermore, by cdh descent, we have an exact triangle
$$\bigoplus\limits_{\{ i,j\} \subseteq J} (p_{ij})_* \underline{\Lambda}(-2)[-4] \rar \bigoplus\limits_{i \in J} (p_i)_*\underline{\Lambda}(-1)[-2] \rar  p_* \iota^! \underline{\Lambda}.$$

Thus, the perverse sheaf ${}^{p} \Hl^k(p_* \iota^! \underline{\Lambda})$ vanishes if $k<2$. Therefore, the map $${}^{p} \Hl^0(p_*\underline{\Lambda}) \rar {}^{p} \Hl^0(p_* \iota^*\gamma_*\underline{\Lambda})$$ is an equivalence and therefore, the perverse sheaf ${}^{p} \Hl^0(p_*\underline{\Lambda})$ vanishes.

Finally, the same method as in the third step of our lemma shows that $${}^{p} \Hl^0(p_*\underline{\Lambda})=k_*P_\Lambda$$ where $P_\Lambda$ is the étale sheaf of $\kappa_{\et}$ which corresponds through Galois-Grothendieck theory to the kernel $P_\Lambda'$ of the map 

$$\bigoplus\limits_{i \in I} \Lambda[(G_i)_{\overline{\kappa}}] \rar \bigoplus\limits_{\{ i,j\} \subseteq J} \Lambda[(E_{ij})_{\overline{\kappa}}]$$ of $\Gamma$-modules. Therefore, the $\Gamma$-module $P_\Lambda'$ vanishes which finishes the proof.
\end{proof}

\begin{remark} Let $k$ be a field characteristic $0$, let $S$ be a $k$-scheme of finite type, let $f\colon X\rar S$ be a quasi-finite affine morphism. By using the Betti realization instead of the $\ell$-adic realization and by adapting the proof of \Cref{Ass_b} 
to the setting of Nori motives (the proof also works with mixed Hodge modules), it is possible to prove that the functor $$f_!\colon \DM^{smA}_{\et,c}(X,\Q)\rar \DM^A_{\et,c}(S,\Q)$$ is perverse homotopy t-exact (this is \cite[5.3.5]{IN}). 
\end{remark}

\begin{remark} \Cref{nearby cycle} is specific to curves. A reasonable equivalent to \Cref{nearby cycle} would then be that if $S$ is an excellent scheme, $i\colon F\rar S$ a closed immersion and $j\colon U\rar S$ is the complementary open immersion, with $F$ a Cartier divisor (so that the morphism $j$ is affine) and $U$ nil-regular, then, the $2$-morphism $$\Theta\colon \rho_!i^*j_* \rar \omega^0 i^*j_*\rho_!$$ is an equivalence.
However, this statement is not true. We now give a counterexample. 

Let $k$ be a field. Take $S=\Spec(k[x,y,z]/(x^2+y^2-z^2))$ to be the affine cone over the rational curve of degree $2$ in $\mb{P}^2$ given by the homogeneous equation $x^2+y^2=z^2$. The scheme $S$ has a single singular point $P=(0,0,0)$.

Let $F$ be the Cartier divisor of $S$ given by the equation $x=0$, so that $F$ is the union of the two lines $$L_1\colon y-z=0 \text{ and } L_2\colon y+z=0$$ which contain $P$ and therefore the complement $U$ of $F$ is regular.

The blow up $B_P(S)$ of $S$ at $P$ is a resolution of singularities of $S$ and the exceptional divisor $E$ above $F$ is the reunion of $\mb{P}^1_k$ and of the strict transform $L_i'$ of the line $L_i$ for $i\in \{ 1,2\}$. Each $L_i'$ crosses $\mb{P}^1_k$ at a single point and $L_1' $ and $L_2'$ do not cross each other.

We therefore have a commutative diagram

$$\begin{tikzcd}
U \arrow[d, equal] \arrow[r, "\gamma"] & \widetilde{S}=B_P(S) \arrow[d, "f"] & E=L'_1 \cup \mb{P}^1_k \cup L'_2 \arrow[d, "p"] \arrow[l, "\iota"] \\
U \arrow[r, "j"]                               & S                           & F=L_1 \cup L_2 \arrow[l, swap,"i"]                 
\end{tikzcd}$$
where both squares are cartesian.
Now, \Cref{az211 etale coh} yields $$\rho_!i^*j_*\underline{\Q}=\un_F.$$ 
On the other hand, we have $$\omega^0 i^*j_*\un_U=\omega^0i^*f_* \gamma_* \un_U=\omega^0p_* \iota^*\omega^0 \gamma_*\un_U$$ by \cite[3.3]{plh}. But \Cref{az211} ensures that $$\omega^0\gamma_*\un_U=\un_{\widetilde{S}}$$ and thus, we get $$\omega^0 i^*j_*\un_U=\omega^0p_* \un_E.$$ Finally, notice that $\omega^0p_* \un_E$ does not coincide with $\un_F$ since by \cite[3.9]{plh} $\omega^0i_x^!\un_F$ is trivial while $\omega^0i_x^!\omega^0p_* \un_E$ is equivalent to $\un_x$.
\end{remark}
\subsection{Artin vanishing for \texorpdfstring{$f_*$}{f*}}
We also have a $t_{hp}$-exactness result about the functor $\omega^0f_*$ when $f$ is a proper morphism. This result is a bit unusual (and very specific to Artin motives) as it requires $f$ to be proper and not affine. When $f\colon X \to S$ is smooth it can be easily explained by \cite[3.7]{plh} which states that 
\[\omega^0f_*\un_X\simeq g_*\un_Y\]
with 
\[X \to Y \xrightarrow{g} S\] 
the Stein factorization $f$.

\begin{proposition}\label{affine lef direct image} Let $(S,\delta)$ be a d-scheme allowing resolution of singularities by alterations, let $f\colon X\rar S$ be a proper morphism. Then, the functor $$\omega^0f_*\colon \DM^A_{\et,c}(X,\Q)\rar \DM^A_{\et,c}(S,\Q)$$ is right $t_{hp}$-exact. 
\end{proposition}
\begin{proof} Recall that since $f$ is proper, we have $$\omega^0f_*=\omega^0f_!.$$

Since the functor $f_!$ is a left adjoint functor, it respects small colimits. Furthermore, with rational coefficients, the functor $\omega^0$ also respects small colimits by \cite[3.5]{plh}. Hence, it suffices to show that the Artin motives of the form $\omega^0 f_* h_X(Y)[\delta(Y)]$ with $Y$ finite over $S$ are perverse homotopy t-non-positive. Hence, replacing $X$ with $Y$, it suffices to show that the Artin motive  $\omega^0 f_* \un_X[\delta(X)]$ is perverse homotopy t-non-positive.

This follows from \Cref{lemme aff lef direct}.
\end{proof}
\begin{lemma}\label{lemme aff lef direct} Let $(S,\delta)$ be a d-scheme. Let $f\colon X\rar S$ be a morphism of finite type. Then, the Artin motive $\omega^0 f_* \un_X[\delta(X)]$ belongs to the smallest subcategory of $\DM^A_{\et,c}(S,\Q)$ which is closed under finite colimits, extensions and retracts and contains the motives of the form $g_*\un_Y[\delta(Y)]$ with $g\colon Y\rar S$ finite.
\end{lemma}
\begin{proof}
We now proceed by noetherian induction on $X$. If $\dim(X)=0$, the morphism $f$ is finite and the result follows from \Cref{AM.t-adj2}.

\textbf{Step 1:} Assume first that the map $f$ is proper.

 Using the induction hypothesis and cdh-descent, we can assume the scheme $X$ to be normal. Let $q\colon \widetilde{X}\rar X$ be a resolution of singularities by alterations and let $q'\colon X'\rar X$ be the relative normalization of $X$ in $\widetilde{X}\setminus B$ where $B$ is the inverse image of the singular locus of $X$. Using \cite[2.1.165]{ayo07}, the motive $\omega^0f_*q'_*\un_{X'}$ is a retract of the motive $\omega^0f_*q_*\un_{\widetilde{X}}$. Hence, using the induction hypothesis and cdh-descent, we can assume the scheme $X$ to be regular and connected.

Let $\eta=\Spec(\mb{K})$ be a generic point of $S$. Let $f_\mb{K}\colon X_\mb{K} \rar \mb{K}$ be the pullback of $f$ to $\mb{K}$. De Jong's resolution of singularities yield a proper alteration $X''_\mb{K}\rar X_{\mb{K}}\rar X_\mb{K}$ such that the map $X_{\mb{K}}\rar X_\mb{K}$ is proper and birational, the map $X''_\mb{K}\rar X_{\mb{K}}$ is a finite morphism which is generically the composition of an étale cover and of a purely inseparable morphism and such that the structural map $X''_\mb{K}\rar \mb{K}$ is the composition of a purely inseparable morphism and of a smooth morphism. Therefore by \cite[8]{ega4-partie3}, we have a dense open subset $U$ and a commutative diagram 
$$\begin{tikzcd} X''_U \ar[r,"q"] \ar[dr,swap,"{f''_U}"]& X'_U \ar[r,"p"]\ar[d,"{f'_U}"] & X_U \ar[dl,"f_U"] \\
&U&
\end{tikzcd}$$ 
such that the map $f_U$ is the pullback of the map $f$ to $U$, the map $q$ is proper and birational, the map $q$ is finite and is generically the composition of an étale Galois cover and of a purely inseparable morphism and the map $f''_U$ is the composition of a purely inseparable morphism and of a smooth morphism.


Shrinking $U$ if needed, we can assume that it is regular. 
Denote by $j\colon U\rar S$ the immersion, by $i\colon F\rar S$ its reduced closed complement and by $f_F\colon X_F\rar F$ the pullback of $f$ along the map $i$. By cdh-descent and induction, we can assume that the scheme $X_F$ is simple normal crossing in $X$ by replacing $X$ with some abstract blow-up.

By localization, we have an exact triangle 
$$i_*i^!f_*\un_X\rar f_* \un_X \rar j_* (f_U)_* \un_{X_U}.$$
Using \cite[3.9]{plh}, a quick induction on the number of branches of the normal crossing divisor $X_F$ inside the regular scheme $X$ combined with the absolute purity property shows that the motive $\omega^0 i^!f_*\un_X$ vanishes. 
This yields an equivalence $$\omega^0f_* \un_X = \omega^0j_* (f_U)_* \un_{X_U}$$

By \cite[3.3]{plh}, we therefore get that $$\omega^0f_* \un_X = \omega^0j_* \omega^0(f_U)_* \un_{X_U}.$$

Let $Z$ be the complement of the open subset of $X_U$ over which the map $p$ is an isomorphism and let $Z'$ be its pullback along the map $p$. By cdh-descent, and by using the induction hypothesis on the structural maps $Z\rar U\rar S$ and $Z'\rar U\rar S$, it suffices to show that $$\omega^0j_* \omega^0(f_U')_* \un_{X_U'}[\delta(X)]\leqslant_{hp} 0.$$

Furthermore, \cite[2.1.165]{ayo07} implies that the motive $\omega^0(f_U')_* \un_{X'_U}$ is a retract of the motive $\omega^0(f_U'')_* \un_{X''_U}$ and therefore, it suffices to show that $$\omega^0j_* \omega^0(f_U'')_* \un_{X_U''}[\delta(X)]\leqslant_{hp} 0.$$

Since the map $f_U''$ is the composition of a smooth and proper morphism and of a purely inseparable morphism, letting $$X_U''\rar Y\overset{g}{\rar} U$$ be the Stein factorization of $f_U''$ yields by \cite[3.3]{plh}(5) that $$\omega^0(f_U'')_* \un_{X_U''}=g_*\un_Y.$$

On the other hand, the localization property also gives an exact triangle $$j_!g_*\un_Y\rar \omega^0j_* g_*\un_Y \rar i_* \omega^0 i^*j_* g_*\un_Y$$
and therefore, using \Cref{AM.t-adj}, it suffices to show that $$\omega^0 i^*j_* g_*\un_Y[\delta(X)-1]\leqslant_{hp} 0.$$

Let $Y_1$ be a relative normalization of $S$ in $Y$ and let $\overline{Y}$ be a resolution of singularities by alterations of $Y_1$. Let $k\colon \overline{Y}\rar S$ be the structural map. The scheme $\overline{Y}$ is regular and there is a map $e\colon Y'\rar Y$ which is the composition of a finite étale cover and of a purely inseparable morphism such that we have a commutative diagram
$$\begin{tikzcd}
Y' \arrow[d, swap, "g\circ e"] \arrow[r, "\gamma"] & \overline{Y} \arrow[d, "k"] & Y_F \arrow[d, "g_F"] \arrow[l, swap, "\iota"] \\
U \arrow[r, "j"]                               & S                           & F \arrow[l, swap, "i"]                  
\end{tikzcd}$$ 
which is made of cartesian squares.

The Artin motive $\omega^0 i^*j_* g_*\un_Y$ is by \cite[2.1.165]{ayo07} a retract of the motive $\omega^0 i^*j_* (g\circ e)_*\un_{Y'}$ and by proper base change, we get
$$\omega^0 i^*j_* (g\circ e)_*\un_{Y'}=\omega^0(g_F)_* \omega^0 \iota^*\gamma_*\un_Y.$$

But using \Cref{az211} and \cite[3.3]{plh}, we have 
$$\omega^0 \iota^*\gamma_*\un_Y=\iota^*\omega^0\gamma_*\un_Y=\un_{Y_F}.$$ 

Therefore, it suffices to show that $$\omega^0 (g_F)_*\un_{Y_F}[\delta(X)-1]\leqslant_{hp} 0.$$ 

Since $$\dim(Y_F)\leqslant\dim(Y)-1\leqslant \dim(X)-1,$$ the result follows by induction.

\textbf{Step 2:} Assume now that the map $f$ is a general morphism.

Using Nagata's Theorem \cite[0F41]{stacks}, we can write $$f=j\circ \overline{f}$$ where $\overline{f}$ is proper and $j$ is an open immersion. Therefore, we can assume that $f=j\colon U\rar S$ is an open immersion.

We want to show that $$\omega^0 j_*\un_U[\delta(U)]\geqslant_{hp} 0.$$
Let now $p\colon Y\rar S$ be a resolution of singularities by alterations. We can proceed as before to show that the motive $\omega^0 j_* \un_U$ is a retract of the motive $\omega^0p_* \un_{Y}$. The result then follows from the case of proper morphisms.
\end{proof}
\subsection{Perverse Artin motives}
\begin{definition} Let $(S,\delta)$ be an excellent d-scheme. 
	We define the \emph{abelian category of perverse Artin étale motives over $S$ with rational coefficients}, as the heart of the perverse homotopy t-structure on $\mc{DM}^A_{\et,c}(S,\Q)$. We denote by $\mathrm{M}^A_{\mathrm{perv}}(S,\Q)$ this category.
\end{definition}

Our understanding of this category will rely on a third and last version of the Artin vanishing Theorem:
\begin{lemma}\label{Affine Lefschetz direct image} Let $(S,\delta)$ be a d-scheme allowing resolution of singularities by alteration, let $f\colon X\rar S$ be a quasi-finite morphism of schemes. Assume that $X$ is nil-regular.
Then, the functor $$\omega^0 f_*\colon \DM^{smA}_{\et,c}(X,\Q)\rar \DM^A_{\et,c}(S,\Q)$$ is t-exact when both sides are endowed with the perverse homotopy t-structure.
\end{lemma}
\begin{proof} 

Using \Cref{boundedness}, the objects of $\DM^{smA}_{\et,c}(X,\Q)$ are bounded with respect to the perverse homotopy t-structure. Therefore, it suffices to show that the objects of the heart of the t-category $\DM^{smA}_{\et,c}(X,\Q)$ are sent to objects of degree $0$ with respect to the perverse homotopy t-structure.  Since the latter is a semi-simple category by Maschke's Theorem, every object of the heart is a retract of an object of the form $g_* \un_V[\delta(X)]$ with $g\colon V\rar X$ finite and étale. Therefore, it suffice to show that $$\omega^0 f_* g_*\un_V[\delta(V)]\geqslant_{hp} 0$$ which follows from \Cref{lemme aff lef direct}.
\end{proof}

We now show that the abelian category of perverse Artin motives with rational coefficients has similar features as the category of perverse sheaves.

First, we recall the definition of the intermediate extension functor given by the gluing formalism (see \cite[1.4.22]{bbd}):
\begin{definition}
 Let $j\colon U\rar S$ be an open immersion. The \emph{intermediate extension functor} $j_{!*}^A$ is the functor which to an object $M$ in $\mathrm{M}^A_{\mathrm{perv}}(S,\Q)$ associates the image of the morphism $$\Hlh^0 j_! M\rar \Hlh^0 \omega^0 j_* M.$$
\end{definition}

Using \cite[1.4.23]{bbd}, we have the following description of this functor (compare with \cite[2.1.9]{bbd}):
\begin{proposition}\label{AM.carext} Let $i\colon F\rar S$ be a closed immersion, let $j\colon U\rar S$ be the open complementary immersion and let $M$ be in $\mathrm{M}^A_{\mathrm{perv}}(U,\Q)$. Then, the Artin motive $j_{!*}^A (M)$ is the only extension $P$ of $M$ to $S$ such that $i^* P<_{hp}0$ and $\omega^0 i^! P>_{hp}0$.
\end{proposition}

The following result is similar to \cite[4.3.1]{bbd}.

\begin{proposition}\label{simple objects} Let $(S,\delta)$ be an excellent d-scheme. 
\begin{enumerate}
\item The abelian category of perverse Artin motives with rational coefficients on $S$ is artinian and noetherian: every object is of finite length.
\item If $j\colon V\hookrightarrow S$ is the inclusion of a regular connected subscheme and if $L$ is a simple object of $\Loc_V(K)$, then the perverse Artin motive $j_{!*}^A(\rho_!L[\delta(V)])$ is simple. Every simple perverse Artin motive is obtained this way.
\end{enumerate}
\end{proposition}
\begin{proof} The proof is the same as in \cite[4.3.1]{bbd} replacing \cite[4.3.2, 4.3.3]{bbd} with the following lemmas.
\end{proof}
\begin{lemma} Let $V$ be a regular connected scheme. Let $L$ be a simple object of $\Loc_V(K)$. Then, letting $F=\rho_! L[\delta(V)]$, for any open immersion $j\colon U\hookrightarrow V$, we have $$F=j_{!*}^A j^*F.$$
\end{lemma}
\begin{proof} Let $i\colon F\rar V$ be the reduced complementary closed immersion to $j$. Then, the motive $i^*F$ belongs to $\rho_!\Loc_F(R)[\delta(X)]$ at is therefore in degree at most $\delta(F)-\delta(X)<0$ with respect to the perverse homotopy t-structure.

Now, if $x$ is a point of $F$, we have $$\omega^0i_x^! F=0$$ by absolute purity and \cite[3.9]{plh}. Therefore, by \Cref{AM.locality}, we get $$\omega^0i^!F>_{hp}0.$$
Therefore, the motive $F$ is an extension of $j^*F$ such that $i^* F<_{hp}0$ and $\omega^0 i^! F>_{hp}0$ and the proposition follows from \Cref{AM.carext}.
\end{proof}
\begin{lemma} Let $V$ be a regular connected scheme.  If $L$ is a simple object of $\Loc_V(K)$, then $\rho_! L[\delta(V)]$ is simple in $\mathrm{M}^A_{\mathrm{perv}}(V,\Q)$.
\end{lemma}
\begin{proof}The proof is exactly the same as \cite[4.3.3]{bbd}.
\end{proof}

\begin{proposition}\label{EMX}(compare with \cite[4.9.7]{aps}
) Let $(X,\delta)$ be a d-scheme allowing resolution of singularities by alterations, let $j\colon U\rar X$ be an open immersion with $U$ nil-regular and let $d=\delta(X)$. Ayoub and Zucker defined (see \cite[2.20]{az}) a motive $$\mb{E}_X=\omega^0j_*\un_U.$$ 
Then,
\begin{enumerate} \item $\mb{E}_X[d]=j_{!*}^A (\un_U[d])$.
\item $\mb{E}_X [d]$ is a simple perverse Artin motive over $X$. In particular, $\omega^0j_*\un_U$ is concentrated in degree $d$ with respect to the perverse homotopy t-structure.
\end{enumerate}
\end{proposition}
\begin{proof} First, notice that Assertion (2) follows from Assertion (1) and \Cref{simple objects}. \Cref{Affine Lefschetz direct image} ensures that $\omega^0j_*\un_U$ is in degree $d$ with respect to the perverse homotopy t-structure. Moreover, we have an exact triangle:
$$j_!\un_U \rar \omega^0j_*\un_U \rar i_*\omega^0i^*j_*\un_U.$$

thus, it suffices to show that $\Hlh^d(\omega^0i^*j_*\un_U)=0$. To prove this, note that we can assume $X$ to be normal and connected as in the first step of the proof of \Cref{lefschetz affine 2Q}. By assumption, there is a proper and surjective map $p\colon Y\rar X$ such that $Y$ is regular and integral and $p$ is generically the composition of an étale Galois cover and of a finite surjective purely inseparable morphism. 

Therefore, there is a closed subscheme $S$ of $X$ of positive codimension such that $p:Y\setminus p^{-1}(S)\rar X\setminus S$ is the composition of an étale Galois cover of group $G$ and of a finite surjective purely inseparable morphism.

Let $X'$ be the relative normalization of $X$ in $Y\setminus p^{-1}(S)$. Then, the scheme $X'$ is finite over $X$ and endowed with a $G$-action and the morphism $X'/G\rar X$ is purely inseparable since $X$ is normal. We have a commutative diagram

$$\begin{tikzcd}U' \ar[d,swap, "p_U"] \ar[r,"j'"] & X' \ar[d,"p"] & F' \ar[l,swap,"i'"] \ar[d,"p_F"] \\
U \ar[r,"j"] & X & F \ar[l,"i"] 
\end{tikzcd}$$
made of cartesian squares. Using \cite[2.1.165]{ayo07}, the motive $\un_U$ is a retract of the motive $(p_U)_* \un_{U'}$. Moreover, we have $$\omega^0i^*j_*(p_U)_* \un_{U'}=(p_F)_* \omega^0 (i')^*j'_* \un_{U'}.$$ Thus, we can replace $X$ with $X'$ and assume that the group $G$ is trivial. In this case, the morphism $p$ is birational. Let $\pi\colon E\rar F$ be the pullback of $p$ along $i$. Using the same argument as in the first step of the proof of \Cref{lefschetz affine 2Q}, we get that $$\omega^0i^*j_*\un_U=\omega^0\pi_* \un_E.$$

The proof then follows from \Cref{affine lef direct image}.
\end{proof}

\section{Integral coefficients}
\subsection{The t-structure on torsion motives}
The goal of this section is to link the perverse homotopy t-structure with the classical perverse t-structure in the torsion case. 

Let us start by some recollection on the perverse t-structure on the category of torsion étale sheaves over base schemes endowed with a dimension function following \cite{notesgabber}. Let $R$ be a commutative ring. The self-dual perversity $p_{1/2}\colon Z\mapsto -\delta(Z)$ induces two t-structures $[p_{1/2}]$ and $[p_{1/2}^+]$ on $\Shv(S_{\et},R)$. In this paper, we only consider the t-structure $[p_{1/2}]$. 

\begin{definition} (\cite[2]{notesgabber}) 
Let $(S,\delta)$ be a d-scheme 
and let $R$ be a commutative ring.
Define
\begin{itemize}
\item the full subcategory ${}^{p} \Shv^{\leqslant 0}(S_{\et},R)$ made of those $M$ such that for any point $x$ of $S$, we have $$i_x^*M\leqslant p_{1/2}(\overline{\{ x\}})=-\delta(x).$$ 
\item the full subcategory ${}^{p} \Shv^{\geqslant 0}(S_{\et},R)$ made of those bounded below étale sheaves $M$ such that for any point $x$ of $S$, we have $$i_x^!M\geqslant p_{1/2}(\overline{\{ x\}})=-\delta(x).$$ 
\end{itemize}
\cite[6]{notesgabber} implies that this defines a t-structure on $\Shv(S_{\et},R)$: the \emph{perverse t-structure} $t_p$.
\end{definition}

\begin{definition}
 Let $S$ be a scheme, let $R$ be a commutative ring and let $p$ be a prime number.
 We define stable categories as follows.
\begin{enumerate}
 \item The stable category of \emph{torsion étale motives} $\mc{DM}_{\et,\mathrm{tors}}(S,R)$ as the subcategory of $\mc{DM}_{\et}(S,R)$ made of those objects $M$ such that the motive $M\otimes_\Z \Q$ vanishes.
 \item The stable category of \emph{$p^\infty$-torsion étale motives} $\mc{DM}_{\et,p^\infty-\mathrm{tors}}(S,R)$ as the subcategory of the stable category $\mc{DM}_{\et}(S,R)$ made of those objects $M$ such that the motive $M\otimes_\Z \Z[1/p]$ vanishes.
 \item Likewise, we define the category of \emph{torsion étale sheaves} $\Shv_{\mathrm{tors}}(S_{\et},R)$ and the category of \emph{$p^\infty$-torsion étale sheaves} $\Shv_{p^\infty-\mathrm{tors}}(S_{\et},R)$. 
\end{enumerate}
\end{definition}

\begin{proposition}
Let $(S,\delta)$ be a d-scheme and let $R$ be a commutative ring. Then, the perverse t-structure of $\Shv(S_{\et},R)$ induces a t-structure on $\Shv_{p^\infty-\mathrm{tors}}(S_{\et},R)$. 
\end{proposition}
\begin{proof} We only need to show that the functor $-\otimes_\Z \Z[1/p]$ is perverse t-exact: indeed it suffices to show that the perverse truncation functor preserves $p^\infty$-torsion sheaves (compare with the proof of \cite[2.2.2]{AM1}). The functor is indeed t-exact because if $x$ is a point of $S$, denoting by $i_x\colon\{ x\}\rar S$ the inclusion, the functors $i_x^*$ and $i_x^!$ commute with the functor $-\otimes_\Z \Z[1/p]$.
\end{proof}

\begin{definition}\label{perverse torsion} Let $(S,\delta)$ be a d-scheme and let $R$ be a commutative ring. The \emph{perverse t-structure} over the stable category $\mc{DM}_{\et,\mathrm{tors}}(S,R)$ is the only t-structure such that the equivalence $$\prod_{p \text{ prime}} \Shv_{p^\infty-\mathrm{tors}}(S[1/p]_{\et},R) \rar \mc{DM}_{\et,\mathrm{tors}}(S,R)$$
which sends an object $(M_p)_{p \text{ prime}}$ to the object $\bigoplus\limits_{p \text{ prime}} j_p^*(\rho_! M_p)$ is t-exact when the left hand side is endowed with the product of the perverse t-structures (recall that it is indeed an equivalence by \cite[2.1.4]{AM1}).
\end{definition}

\begin{proposition}\label{perverse homotopy torsion} Let $(S,\delta)$ be a d-scheme and let $R$ be a commutative ring. The inclusion functor $$\DM_{\et,\mathrm{tors}}(S,R)\rar \DM_{\et}^A(S,R)$$ is t-exact when the left hand side is endowed with the perverse t-structure and the right hand side is endowed with the perverse homotopy t-structure.
\end{proposition}
\begin{proof} If $f\colon X\rar S$ is quasi-finite, \Cref{442em} implies that the functor $\rho_!$ maps the étale sheaf $f_!\underline{R}$ to the motive $M_S^{BM}(X)$. \Cref{generators perv} below asserts that the sheaves of the form $f_!\underline{R}[\delta(X)]$ with $f\colon X\rar S$ quasi-finite generate the perverse t-structure on the stable category $\Shv(S_{\et},R)$. Since the exact functor $$\rho_!\colon \Shv(S_{\et},R)\rar \DM_{\et}(S,R)$$ is compatible with small colimits, it is therefore right t-exact.

We can prove the left t-exactness using the same proof as the proof of the left t-exactness in \cite[4.1.6]{AM1}. We also give a shorter proof. 

Let $M$ be a torsion étale motive which is perverse t-non-negative. By definition of the t-structure, $M$ is bounded below (with respect to the ordinary t-structure) and for any point $x$ of $S$, we have $$i_x^!M\geqslant -\delta(x)$$ with respect to the ordinary t-structure. \Cref{omega^0 torsion} then ensures that $$i_x^!M=\omega^0 i_x^!M.$$
Therefore, using \cite[4.1.6]{AM1}, we have $$\omega^0 i_x^!M\geqslant_{\ord} -\delta(x).$$ 
\Cref{AM.locality}(1)(a) ensures that the Artin motive $M$ is therefore $t_{hp}$-non-negative.
\end{proof}


\begin{lemma}
    \label{omega^0 torsion} Let $S$ be a scheme and let $R$ be a commutative ring. The functor $\omega^0\rar \id$ induces an equivalence over $\mc{DM}_{\et,\mathrm{tors}}(S,R)$.
\end{lemma}
\begin{proof}This follows from the fact that all torsion motives are Artin motives by \cite[1.5.7]{AM1}.
\end{proof}

\begin{lemma}\label{generators perv} Let $(S,\delta)$ be a d-scheme and let $R$ be a commutative ring. The perverse t-structure on $\Shv(S_{\et},R)$ is generated in the sense of \Cref{AM.t-structure generated} by the étale sheaves of the form $f_!\underline{R}[\delta(X)]$ with $f\colon X\rar S$ quasi-finite.
\end{lemma}
\begin{proof} When $S$ is the spectrum of a field, the result follows from \cite[2.2.1]{AM1}.
Denote by $t_0$ the t-structure generated by the étale sheaves of the form $f_!\underline{R}[\delta(X)]$ with $f\colon X\rar S$ quasi-finite.
It suffices to show that a sheaf $M$ is perverse t-non-negative if and only if it is $t_0$-non-negative.
Hence, it suffices to show that a sheaf $M$ is $t_0$-non-negative if and only if it is bounded below and for any point $x$ of $S$, the étale sheaf $i_x^!M$ is $t_0$-non-negative.

To prove the "only if" part, notice that if $f$ is a quasi-finite morphism, the adjunction $(f_!,f^!)$ is a $t_0$-adjunction. Indeed, the full subcategory of those étale sheaves $N$ such that $f_!N$ is $t_0$-non-positive is closed under extensions and small colimits and contains the set of generators of $t_0$. Thus, this subcategory contains all $t_0$-non-positive objects. Hence, the functors $f_!$ is right $t_0$-exact and therefore, the functor $f^!$ is left-$t_0$-exact. Hence, if $M$ is $t_0$-non-negative and if $x$ is a point of $S$, the étale sheaf $i_x^!M$ is $t_0$-non-negative which yields 
$$i_x^!M\geqslant -\delta(x)$$ 
with respect to the ordinary t-structure on $\Shv(k(x)_{\et},R)$.

Furthermore, if $M$ is $t_0$-non-negative, if $c$ is a lower bound for the function $\delta$ on $S$, the complex $\Map(R_S(X),M)$ is $(c-1)$-connected for any étale $S$-scheme $X$. Hence, the étale sheaf sheaf $M$ is bounded below and it is therefore perverse t-non-negative.

We now prove the converse. Let $M$ be a bounded below étale sheaf such that for all point $x$ of $S$, the étale sheaf $i_x^!M$ is $t_0$-non-negative. Let $f\colon X\rar S$ be quasi-finite. 

As in the proof of \Cref{AM.locality}, we have the $\delta$-niveau spectral sequence 
$$E^1_{p,q}\Rightarrow \Hl^{p+q}\Map_{\Shv(S_{\et},R)}(f_!\underline{R},M)$$
with
$$E^1_{p,q}=\bigoplus_{y\in X,\hspace{1mm} \delta(y)=p} \Hl^{p+q} \Map_{\Shv(k(f(y))_{\et},R)}((f_y)_!\underline{R},i_{f(y)}^!M).$$

Since for any point $x$ of $S$, the sheaf $i_x^!M$ is $t_0$-non-negative, for any point $y$ of $X$, the complex $\Map_{\Shv(k(f(y))_{\et},R)}((f_y)_!\underline{R},i_{f(y)}^!M)$ is $(-\delta(y)-1)$-connected and thus it is $(-\delta(X)-1)$-connected. 

Thus, if $n>\delta(X),$ the $R$-module $\Hl^{-n}\Map_{\Shv(S_{\et},R)}(f_!\underline{R},M)$ vanishes. Hence, the étale sheaf $M$ is $t_0$-non-negative.
\end{proof}




\subsection{Pathological behaviors of Artin truncation}
\Cref{omega^0} is on of the key points for the existence of the t-structure with rational coefficients. It becomes (very) false integrally. Namely we have \Cref{quand c'est qu'c'est constructible} below which gives an idea of how bas things can get. We can however still make a lot of computations as is shown in \Cref{omega^0 sur un corps}. This will be our lifeline in order to prove that the t-structure still exists in dimension $2$ or less (see \Cref{AM.main theorem}).

\begin{proposition}\label{quand c'est qu'c'est constructible} Let $R$ be a localization of $\Z$ and let $f\colon X\rar S$ be a separated morphism of finite type. If the object $\omega^0 f_! \un_X$ of $\DM^A_{\et}(S,R)$ is constructible, then we are in one of the following cases:
\begin{enumerate}\item $R=\Q$.
\item the map $f$ is quasi-finite.
\item there is a prime number $p$ such that $R=\Z_{(p)}$ and all the points $x$ of $S$ such that $f^{-1}(x)$ is infinite are of residue characteristic $p$.
\end{enumerate}
\end{proposition}
\begin{proof}
	Let $x$ be a point such that $f^{-1}(x)$ is infinite. We can chose $x$ to be of minimal codimension in $S$. Since pullback by an open immersion commutes with $\omega^0$, we can replace $S$ with an open neighborhood of $x$. Letting $Z$ be the closure of $\{x\}$ and $U$ be its complement we can therefore assume that the restriction of $f$ to $U$ is quasi-finite. Denoting by $i$ (resp. $j$) the associated closed (resp. open) immersion, $j^*f_!\un_X$ is an Artin motive and thus by \cite[3.3]{plh}, the natural transformation $i^*\omega^0f_! \un_X \to \omega^0i^*f_! \un_X$ is an equivalence. Hence, we can replace $S$ with $Z$ in order to assume that $x$ is the generic point of $S$.

	Then, by \cite[3.3]{plh}, $i_x^*=i_x^!$ commutes with $\omega^0$, which shows that we can assume $S$ to be reduced to a single point. In that case, the result follows from \Cref{Lemme du cas sur un corps} below. 
\end{proof}

\Cref{Lemme du cas sur un corps} will follow from \Cref{omega^0 sur un corps}. To state this result, we need to recall basic information on Artin motives over a field from \cite{AM1}.

\begin{definition} Let $\pi$ be a topological group and let $R$ be a commutative ring. An $R[\pi]$-module $M$ is \emph{discrete} if the action of $\pi$ on $M$ is continuous when $M$ is endowed with the discrete topology. We will denote by $\mathrm{Mod}(\pi,R)$ the Grothendieck abelian category of discrete $R[\pi]$-modules.
\end{definition}

\begin{example}\label{Tate twist} Let $k$ be a field of characteristic exponent $p$. If $n$ is a positive integer coprime to $p$, the absolute Galois group $G_k$ of $k$ acts on the abelian group $\mu_n(\overline{k})$. Let $(T_n)$, where $n$ runs through the set of positive integers which are coprime to $p$, be an inductive system of discrete representations of $G_k$ such that $T_n$ is of $n$-torsion. Set $T=\colim T_n$ and let $m$ be an integer, we can define a discrete representation $$T(m):=\colim T_n \otimes \mu_n^{\otimes m}$$ of $G_k$. We call this representation the \emph{$m$-th discrete Tate twist} of $T$ (beware that it depends on the inductive system and not only on $T$ itself in general). By definition, $T$ is isomorphic to every $T(m)$ as an abelian group but not as a $G_k$-module in general.

For example, let $X\rar \Spec(k)$ be a morphism of finite type, let $\overline{k}$ be the algebraic closure of $k$, let $X_{\overline{k}}$ be the scalar extension of $X$ to $\overline{k}$ and let $n$ be a non-negative integer. Then, the $m$-th discrete Tate twist of the discrete Galois representation $\Hl^n_{\et}(X_{\overline{k}},\Q/\Z[1/p])$ will by convention be the discrete representation $$\colim\limits_{(N,p)=1} \Hl^n_{\et}(X_{\overline{k}},\Z/N\Z)\otimes \mu_N^{\otimes m}.$$
\end{example}
\begin{remark}
    With the notation of the example above, we have in fact $T(m)\simeq T \otimes_{\widehat{\Z}'} \widehat{\Z}'(m)$ where $\widehat{\Z}'=\prod\limits_{\ell \neq p} \Z_\ell$.
\end{remark}

\begin{emptypar}\label{notations alpha_!} Let $R$ be a localization of $\Z$. Let $k$ be a field of characteristic exponent $p$ and let $\sigma\colon R \to R[1/p]$ be the localization morphism.
	By \cite[3.3.7]{AM1}, the functor $\rho_!$ from \Cref{small et site} induces an equivalence 
	\[\Shv(k_{\et},R)\xrightarrow{\sigma_* \rho_!}\DM^A_{\et}(k,R)\]
	and furthermore, letting $\xi\colon\Spec(\overline{k})\rar \Spec(k)$ be the algebraic closure of $k$, composing this equivalence with the inverse of the fiber functor $\xi^*$ associated to $\xi$ yields an equivalence $$\alpha_!\colon \mc{D}(\mathrm{Mod}(G_k,R[1/p]))\rar \mc{DM}^A(k,R).$$
Furthermore, constructible Artin motives are exactly the essential image of the subcategory $\mc{D}^b(\Rep^A(G_k,R))$ of bounded complexes of Artin representations of $G_k$. 
	
	Let $f\colon X\rar S$ be a morphism of finite type and let $X_{\overline{k}}$ be the pull-back of $X$ to $\overline{k}$. We let \[h^n(X,R)=\Hl^n_{\et}(X_{\overline{k}},R[1/p]\otimes_\Z \Q/\Z).\]



	If $N$ is a positive integer, letting $\mu_N$ be the sheaf of $N$-th roots of unity over $k_{\et}$, 
	for $m$ an integer, we define the \emph{$m$-th Tate twist} of the étale sheaf $f_* \left[\underline{R[1/p]\otimes_\Z \Q/\Z}\right]$ as the étale sheaf $$f_* \left[\underline{R[1/p]\otimes_\Z \Q/\Z}\right](m):=\colim_{(n,p)=1} f_*\left[\underline{R[1/p]\otimes_\Z \Z/n\Z}\right]\otimes_\Z \mu_n^{\otimes m}.$$

	Through $\alpha_!$, the $i$-th cohomology group of $f_* \left[\underline{R[1/p]\otimes_\Z \Q/\Z}\right](m)$ seen as a complex of discrete $G_k$-modules is $h^n(X,R)(m)$. 
\end{emptypar}

We can now state the main result of this section:
\begin{theorem}\label{omega^0 sur un corps} Let $R$ be a localization of $\Z$, let $k$ be a field of characteristic exponent $p$, let $\sigma\colon R\rar R[1/p]$ be the localization morphism and let $f\colon X\rar \Spec(k)$ be a morphism of finite type. Assume that the scheme $X$ is regular.
Then, 
\begin{enumerate}
\item The diagram 
$$\begin{tikzcd}\Shv_{\In\lis}(X,R[1/p])\ar[r,"\sigma_*\rho_!"] \ar[d,"f_*"]& \DM_{\et}^{smA}(X,R) \ar[d,"\omega^0f_*"] \\
\Shv(k_{\et},R[1/p]) \ar[r,"\sigma_*\rho_!"] & \DM^A_{\et}(k,R)
\end{tikzcd}$$
is commutative.
\item We have $$\omega^0f_* \un_X=\rho_!f_* \underline{R[1/p]}.$$

\item If $m$ is a positive integer, we have $$\omega^0(f_* \un_X(-m))=\rho_!\left[f_* \left(\underline{R[1/p]\otimes \Q/\Z}\right)[-1](-m)\right].$$
\end{enumerate}


\end{theorem}

\begin{proof} We can assume that $p$ is invertible in $R$. We have pairs of adjoint functors 
$$f^*\colon\mc{DM}^A_{\et}(k,R)\leftrightarrows\mc{DM}^{smA}_{\et}(X,R)\colon \omega^0f_*$$ 
and 
$$f^*\colon\Shv_{\In\lis}(k,R)\leftrightarrows\Shv_{\In\lis}(X,R)\colon f_*.$$ 
Since the functor $\rho_!$ is an equivalence and commutes with the functor $f^*$ by \Cref{442em}, this yields the first assertion.

The second assertion follows from the first one, since we have $$\un_X=\rho_!\underline{R}.$$

We now prove the third assertion. Let $\pi\colon \mb{P}^m_X\rar \Spec(k)$ be the canonical projection. We have $$\pi_*\un_{\mb{P}^m_X}=\bigoplus\limits_{i=0}^mf_*\un_X(-i)[-2i]$$ by the projective bundle formula.
On the other hand, we have an exact triangle étale sheaves
$$\pi_* \left(\underline{R}\right)\rar \pi_* \left(\underline{R\otimes_\Z \Q}\right)\rar \pi_*\left( \underline{R\otimes_\Z \Q/\Z}\right).$$
\cite[2.1]{deninger} yields $$\pi_* \left(\underline{ R\otimes_\Z\Q}\right)=f_* \left(\underline{R\otimes_\Z\Q}\right)$$ and the projective bundle formula in étale cohomology yields $$\pi_*\left( \underline{R\otimes_\Z \Q/\Z}\right)=\bigoplus \limits_{i=0}^m f_* \left(\underline{R\otimes_\Z \Q/\Z}\right)(-i)[-2i].$$
The result follows by induction on $m$.
\end{proof}

\begin{definition}\label{pi_0(X/k)} Let $f\colon X\rar \Spec(k)$ be a morphism of finite type, let $\overline{k}$ be the algebraic closure of $k$ and let $X_{\overline{k}}$ be the pull-back of $X$ to $\overline{k}$. The set $\pi_0(X_{\overline{k}})$ of connected components of $X_{\overline{k}}$ is a finite $G_k$-set and therefore corresponds to a finite étale $k$-scheme through Galois-Grothendieck correspondence \cite[VI]{sga1}. We let $\pi_0(X/k)$ be this scheme and call it the \emph{scheme of connected components of $X$}.
\end{definition}
\begin{remark}When $f$ is proper, this scheme coincides with the scheme $\pi_0(X/k)$ defined in \cite{az} as the Stein factorization of $f$. In general, the scheme $\pi_0(X/k)$ coincides with the Stein factorization of any compactification $\overline{X}$ of $X$ which is normal in $X$ in the sense that $X$ is its own relative normalization in $\overline{X}$.
\end{remark}

\begin{corollary}\label{omega^0 corps} Let $R$ be a localization of $\Z$, let $k$ be a field of characteristic exponent $p$ and let $f\colon X\rar \Spec(k)$ be a morphism of finite type. Assume that the scheme $X$ is regular.
Then,
\begin{enumerate}\item We have $$\Hl^0(\omega^0f_* \un_X)=h_k(\pi_0(X/k)).$$

\item  We have $$\Hl^n(\omega^0f_* \un_X)=\alpha_!\left[\Hl^n_{\et}(X_{\overline{k}},R[1/p])\right]$$
and 
$$\Hl^n_{\et}(X_{\overline{k}},R[1/p])=\begin{cases} R[1/p][\pi_0(X_{\overline{k}})]& \text{ if }n=0\\
h^{n-1}(X,R)& \text{ if }n\geqslant 2\\
0 &\text{ otherwise}
\end{cases}.$$

\item If $m$ is a positive integer and if $n$ is an integer, we have $$\Hl^n(\omega^0f_* \un_X(-m))=\alpha_!\left[\mu^{n-1}(X,R)(-m)\right].$$
\end{enumerate}
\end{corollary}
\begin{remark}
	When $X$ is not regular, we can still compute the cohomology groups of $\omega^0f_* \un_X$ by using resolutions of singularities by alterations and h-descent. This will be done in some cases below.
\end{remark}

We can finish the proof of \Cref{quand c'est qu'c'est constructible}. 

\begin{lemma}\label{Lemme du cas sur un corps} Let $R$ be a localization of $\Z$, let $k$ be a field of characteristic exponent $p$ and let $f\colon X\rar \Spec(k)$ be a morphism of finite type. Assume that the ring $R[1/p]$ is not a $\Q$-algebra and that the morphism $f$ is not finite. Then, the Artin motive $\omega^0 f_! \un_X$ is not constructible.
\end{lemma}
The lemma follows from the following more precise statement.

\begin{lemma}\label{interesting vanishing} Let $R$ be a localization of $\Z$, let $k$ be a field of characteristic exponent $p$, let $f\colon X\rar \Spec(k)$ be a separated morphism of finite type and let $d=\dim(X)$. Assume that the ring $R[1/p]$ is not a $\Q$-algebra and that $d>0$. Then, the motive $\Hl^{2d+1}(\omega^0 f_! \un_X)$ is not the image of an Artin representation through the map $\alpha_!$. Furthermore, if $n>2d+1$, the motive $\Hl^{n}(\omega^0 f_! \un_X)$ vanishes.
\end{lemma}

\begin{example}
    With the notations of the lemma, assume that $R=\Z$ and that $X=\mb{P}^1_k$. Then, by \Cref{omega^0 corps}, the motive $\Hl^3(\omega^0f_!\un_{\mb{P}^1_k})$ is equivalent to $\alpha_![\Q/\Z[1/p](-1)]$. The idea of this lemma is to prove that for a general $k$-scheme of finite type $X$ with $n$ connected components, it still remains true that the top cohomology object of $\omega^0f_!\un_X$ is placed in degree $2\dim(X)+1$ and corresponds through the equivalence $\alpha_!$ to the direct sum of $n$ copies of $\Q/Z\otimes_{\Z} R$ (which is not of finite presentation as an $R$-module) endowed with some action of $G_k$.
\end{example}
\begin{proof} Since change of coefficients preserves constructible objects, we may assume that there is a prime number $\ell\neq p$ such that the ring $R=\Z_{(\ell)}$. We may also assume that the scheme $X$ is connected.

We proceed by induction on the dimension of $X$. By compactifying $X$ using Nagata's theorem \cite[0F41]{stacks}, and by using the localization triangle \eqref{AM.localization} and the induction hypothesis, we may assume that the morphism $f$ is proper. By normalizing $X$, by using cdh-descent and by using the induction hypothesis again, we may also assume that the scheme $X$ is normal.

Using Gabber's $\ell$-alterations \cite[IX.1.1]{travauxgabber}, there is a proper and surjective map $p\colon Y\rar X$ such that $Y$ is regular and integral and $p$ is generically the composition of an étale Galois cover of group of order prime to $\ell$ and of a finite surjective purely inseparable morphism. 
Therefore, there is a closed subscheme $F$ of $X$ of positive codimension such that the induced map $$p\colon Y\setminus p^{-1}(F)\rar X\setminus F$$ is the composition of an étale Galois cover of group $G$, where the order of $G$ is prime to $\ell$, and of a finite surjective purely inseparable morphism. 

Let $X'$ be the relative normalization of $X$ in $Y\setminus p^{-1}(F)$. Then, the map $X'\rar X$ is finite  and the scheme $X'$ is endowed with a $G$-action. Furthermore, the map $X'/G\rar X$ is purely inseparable since $X$ is normal. 

Using \cite[2.1.165]{ayo07}, for any integer $n$ we have a canonical isomorphism $$\Hl^n(\omega^0 h_k(X))\rar \Hl^n(\omega^0 h_k(X'))^G.$$
Furthermore, by cdh-descent, and by using the induction hypothesis, the canonical map

$$\Hl^n(\omega^0 h_k(X'))\rar \Hl^n(\omega^0 h_k(Y))$$ is an isomorphism if $n>2d$. 
Now, by Poincaré Duality in étale cohomology \cite{sga4.5} and by \Cref{omega^0 sur un corps}, since $d$ is positive, we have $$\Hl^{2d+1}(\omega^0 h_k(Y))=\alpha_!\hspace{1mm}\left[\Q/\Z(-d)\otimes_\Z R[\pi_0(Y_{\overline{k}})]\right]$$ and the $G$-action on $\Hl^{2d+1}(\omega^0 h_k(Y))$ is induced by the action of $G$ on the connected components of $Y_{\overline{k}}$. 

Since the map $Y\rar X'$ is an isomorphism in codimension $1$, the map $$\pi_0(Y_{\overline{k}})\rar \pi_0(X'_{\overline{k}})$$ is a bijection. Thus, we get $$\Hl^{2d+1}(\omega^0 h_k(Y))=\alpha_!\hspace{1mm}\left[\Q/\Z(-d)\otimes_\Z \Z_{(\ell)}[\pi_0(X'_{\overline{k}})]\right]^G$$

If $x\in \Q/\Z(-d)$ and if $S$ is a connected component of the scheme $X'_{\overline{k}}$, denote by $[x]_S$ the element $x\otimes e_S$ of $\Q/\Z(-d)\otimes_\Z \Z_{(\ell)}[\pi_0(X'_{\overline{k}})]$ where $e_S$ is the element of the basis of $R[\pi_0(X'_{\overline{k}})]$ attached to $S$. Then, the $G$-fixed points are exactly the elements of the form:

$$\sum\limits_{\substack{T \in \pi_0(X_{\overline{k}}), S \in \pi_0(X'_{\overline{k}}) \\ p(S)=T}}[x_T]_S$$
where $x_T\in \Q/\Z(-d)$. Therefore, this set of $G$-fixed points is isomorphic as a $G_k$-module to $$\Q/\Z(-d)\otimes_\Z \Z_{(\ell)}[\pi_0(X_{\overline{k}})]$$ which is not an Artin representation.
\end{proof}

\subsection{The t-structure for schemes of dimension at most \texorpdfstring{$2$}{2}}

To prove that the t-structure still exists in some cases, the key point will actually be our first version of Artin vanishing. We show that it still works integrally. In order to make sense of this, recall the following definition from \cite{AM1}. 

\begin{definition}
	Let $S$ be a scheme and let $R$ be a commutative ring. The category $\DM_{\et,\Q-c}^{(sm)A}(S,R)$ of $\Q$-constructible (smooth) Artin motives is the subcategory of $\DM^{(sm)A}_{\et}(S,R)$ made of those $M$ such that $M\otimes_\Z \Q$ is constructible.  
\end{definition}

Using the same proof as \cite[4.2.4]{AM1}, replacing \cite[4.1.5]{AM1} with \Cref{coeff change perverse} and \cite[4.1.6]{AM1} with \Cref{perverse homotopy torsion} yields
\begin{corollary}\label{Q-cons perv} Let $S$ be an excellent scheme and let $R$ be a commutative ring. Then, the perverse homotopy t-structure on the stable category $\DM_{\et}^A(S,R)$ induces a t-structure on the stable subcategory $\DM_{\et,\Q-c}^A(S,R)$.
\end{corollary}

We can now prove the appropriate generalization of \Cref{lefschetz affine 2Q}.
\begin{theorem}\label{Affine Lefschetz Q-cons} Let $(S,\delta)$ be an excellent scheme, let $f\colon X\rar S$ be a quasi-finite morphism of schemes and let $R$ be a localization of $\Z$. Assume that $X$ is nil-regular and that we are in one of the following cases 
\begin{enumerate}[label=(\alph*)]
    \item We have $\dim(S)\leqslant 2.$
    \item There is a prime number $p$ such that the scheme $S$ is of finite type over $\mb{F}_p$. 
\end{enumerate}

Then, the functor $$f_!\colon \DM^{smA}_{\et,\Q-c}(X,R)\rar \DM^A_{\et,\Q-c}(S,R)$$ is $t_{hp}$-exact.
\end{theorem}
\begin{proof} \Cref{AM.t-adj} implies that the functor $f_!$ is right $t_{hp}$-exact. Hence, it suffices to show that it is left $t_{hp}$-exact. 

Let $M$ be a $\Q$-constructible motive over $X$ which is $t_{hp}$-non-negative. We have an exact triangle $$M\otimes_\Z \Q/\Z[-1]\rar M\rar M\otimes_\Z\Q .$$

By \Cref{coeff change perverse}, the motive $M\otimes_\Z\Q$ is $t_{hp}$-non-negative. Since the subcategory of $t_{hp}$-non-negative is closed under limits, the motive $M\otimes_\Z \Q/\Z[-1]$ is also $t_{hp}$-non-negative.

Furthermore, we have an exact triangle $$f_!(M\otimes_\Z \Q/\Z[-1])\rar f_!(M)\rar f_!(M\otimes_\Z\Q).$$ By \Cref{lefschetz affine 2Q}, the motive $f_!(M\otimes_\Z\Q)$ is $t_{hp}$-non-negative and by \Cref{Affine Lefschetz torsion}, the motive $f_!(M\otimes_\Z \Q/\Z[-1])$ is $t_{hp}$-non-negative. Since the subcategory of $t_{hp}$-non-negative Artin motives is closed under extensions, the motive $f_!(M)$ is also $t_{hp}$-non-negative.
\end{proof}

\begin{lemma}\label{Affine Lefschetz torsion} 
 Keep the same notations. Then, the functor $f_!$ induces a left t-exact functor $$\mc{DM}_{\et,\mathrm{tors}}(X,R)\rar\mc{DM}_{\et,\mathrm{tors}}(S,R)$$ when both sides are endowed with their perverse t-structures.
\end{lemma}
\begin{proof}
 Since the functor of \Cref{perverse torsion} commutes with the six functors by \cite[2.1.4]{AM1}, it suffices to show that if $p$ is a prime number and is invertible on $S$, the functor $$f_!\colon \Shv_{p^\infty-\mathrm{tors}}(X_{\et},R)\rar \Shv_{p^\infty-\mathrm{tors}}(S_{\et},R)$$ is left t-exact when both sides are endowed with their perverse t-structures.

Let $M$ be a $p^\infty$-torsion étale sheaf which is perverse t-non-negative. 
We have $$M=\colim_n M\otimes_\Z \Z/p^n\Z[-1].$$
Let $n$ be a positive integer. 
Since we have an exact triangle
$$M\otimes_\Z \Z/p^n\Z[-1]\rar M\overset{\times p^n}{\rar} M,$$
the sheaf $M\otimes_\Z \Z/p^n\Z[-1]$ is perverse t-non-negative. 
Therefore, by \Cref{affine lefschetz non cons} below, the sheaf $f_!(M\otimes_\Z \Z/p^n\Z[-1])$ is perverse t-non-negative.
Since the functor $f_!$ is a left adjoint, it is compatible with colimits; 
now, the canonical functor $$\colim_n \Map(g_! \und{R}[\delta(Y)],f_!(M)\otimes_\Z \Z/p^n\Z[-1]) \rar \Map(g_! \und{R}[\delta(Y)],\colim_n f_!(M)\otimes_\Z \Z/p^n\Z[-1])$$ is an equivalence for any $g\colon Y\to S$ quasi-finite (this can be proved in the same fashion as \cite[1.2.4 (5)]{AM1}). Therefore, the sheaf $f_!(M)$ is perverse t-non-negative using \Cref{generators perv}.
\end{proof}

\begin{lemma}\label{affine lefschetz non cons} Keep the same notations. Let $p$ be a prime number which is invertible on $S$ and let $n$ be a positive integer. Then, the functor 
$$f_!\colon \Shv(X_{\et},\Z/p^n\Z)\rar \Shv(S_{\et},\Z/p^n\Z)$$
is left t-exact.
\end{lemma}
\begin{proof} By definition of the perverse t-structure, we can assume that the scheme $S$ is the spectrum of a strictly henselian ring. In that case, since the étale sheaves of the form $g_!\underline{\Z/p^n\Z}[\delta(X)]$ with $g\colon X\rar S$ quasi-finite are compact, the perverse t-structure is compatible with filtered colimits. 

Let $M$ be an étale sheave over $X_{\et}$ which is t-non-negative we want to prove that the sheaf $f_!M$ is. The sheaf $M$ is a filtering colimit of constructible sheaves. Therefore, we can assume that the sheaf $M$ is constructible. The result is then exactly Gabber's Artin vanishing Theorem \cite[XV.1.1.2]{travauxgabber}.
\end{proof}

We are now ready for the main theorem of this section.
\begin{theorem}\label{AM.main theorem} Let $(S,\delta)$ be an excellent d-scheme of dimension $2$ or less and let $R$ be a localization of $\Z$, then, the perverse homotopy t-structure on $\DM^A_{\et}(S,R)$ induces a t-structure on $\mc{DM}^A_{\et,c}(S,R)$.
\end{theorem}

In the proof, we will use some notations from \cite[B]{em}.

\begin{definition}\label{p-loc} Let $\mc{C}$ be a $\Z$-linear stable category (meaning that it is enriched over $\D(\Z)$) and let $\mathfrak{p}$ be a prime ideal of $\Z$. 
\begin{enumerate} \item The \emph{naive $\mathfrak{p}$-localization} $\mc{C}_\mathfrak{p}^\mathrm{na}$ of $\mc{C}$ is the category with the same objects as $\mc{C}$ and such that if $M$ and $N$ are objects of $\mc{C}$, we have $$\Map_{\mc{C}_\mathfrak{p}^\mathrm{na}}(M,N)=\Map_\mc{C}(M,N)\otimes_\Z \Z_{\mathfrak{p}}.$$
\item The $\mathfrak{p}$-localization $\mc{C}_{\mathfrak{p}}$ of $\mc{C}$ is the idempotent completion of the naive $\mathfrak{p}$-localization $\mc{C}_\pp^{\mathrm{na}}$ of $\mc{C}$.
\end{enumerate} 
\end{definition}
\begin{emptypar}Let  $\mathfrak{p}$ be a prime ideal of $\Z$.

If $\mc{C}$ is a $\Z$-linear stable category, we have a canonical exact functor $$\mc{C}\rar \mc{C}_{\mathfrak{p}}.$$
and $F\colon \mc{C}\rar \mc{D}$ is an exact functor of such categories, we have a canonical exact functor 
$$F_\mathfrak{p}\colon \mc{C}_\mathfrak{p}\rar \mc{D}_\mathfrak{p}.$$
\end{emptypar}

We can now prove \Cref{AM.main theorem}.
\begin{proof}
We can assume that $S$ is reduced and connected. We proceed by noetherian induction on $S$. If $\dim(S)=0$, the result follows from \Cref{fields}.

Assume now that $\dim(S)>0$. If $\pp$ is a prime ideal of $\Z$, if $X$ is a scheme and if $M$ is an Artin motive over $X$, we denote by $M_\pp$ the image of $M$ in $\DM^A_{\et}(S,R)_\pp$ and we say that an object $N$ of $\DM^A_{\et}(S,R)_\pp$ is constructible when it belongs to $\DM^{A}_{\et,c}(X,R)_\pp$. 
Recall that an Artin motive $M$ is \emph{$\pp$-constructible} when $M_\pp$ is constructible. By \cite[B.1.7]{em}, it suffices to show that for any constructible Artin motive $M$, the Artin motive ${}^{hp}\tau_{\leqslant 0}(M)$ is $\pp$-constructible for any maximal ideal $\pp$ of $\Z$.

Let $\pp$ be a maximal ideal of $\Z$. By \cite[1.1.8]{AM1}, there is a unique t-structure on the stable category $\DM^A_{\et}(S,R)_\pp$ such that the canonical functor $$\DM^A_{\et}(S,R)\rar \DM^A_{\et}(S,R)_\pp$$ 
is t-exact. We still call this t-structure the \emph{perverse homotopy t-structure}. If $M$ is an Artin motive, we get by t-exactness of the functor $M\mapsto M_\pp$ that $${}^{hp}\tau_{\leqslant 0}(M)_\pp={}^{hp}\tau_{\leqslant 0}(M_\pp).$$ 
Hence, if the perverse homotopy t-structure on the stable category $\DM^A_{\et}(S,R)_\pp$ induces a t-structure on the subcategory $\DM^A_{\et,c}(S,R)_\pp$, then the Artin motive ${}^{hp}\tau_{\leqslant 0}(M)$ is $\pp$-constructible for any Artin motive $M$. 
Therefore, it suffices to show that the perverse homotopy t-structure induces a t-structure on the subcategory $\DM^A_{\et,c}(S,R)_\pp$ for any maximal ideal $\pp$ of $\Z$.

The properties of $\DM^A_{\et,(c)}(-,R)$ transfer to $\DM^A_{\et,(c)}(-,R)_\pp$: we get functors of the form $f^*$ for any morphism $f$, functors of the form $\omega^0 f_*$ for any morphism of finite type $f$, functors of the form $f_!$ and $\omega^0 f^!$ for any quasi-finite morphism $f$ and localization triangles. The t-exactness properties of \Cref{AM.t-adj,AM.t-adj2} and Artin vanishing (\Cref{Affine Lefschetz Q-cons}) remain true in this setting.

Let $p$ be a generator of $\pp$. If $p$ is invertible on a regular connected scheme $X$, the functor $\rho_!$ of \Cref{small et site} induces a t-exact fully faithful functor $$\rho_!\colon \Shv_{\lis}(X,R)_\pp\rar \DM_{\et}^A(X,R)_\pp$$ when the left hand side is endowed with the ordinary t-structure shifted by $\delta(X)$ and the right hand side is endowed with the perverse homotopy t-structure. Indeed, the full faithfulness follows from \cite[4.2.8]{AM1} while the t-exactness follows from \Cref{AM.locality} and \cite[1.1.9]{AM1}.

On the other hand, if a regular scheme $X$ is of characteristic $p$, the functor $\rho_!$ induces a t-exact fully faithful functor $$\rho_!\colon \Shv_{\lis}(X,R[1/p])_\pp\rar \DM_{\et}^A(X,R)_\pp$$ when the left hand side is endowed with the ordinary t-structure shifted by $\delta(X)$ and the right hand side is endowed with the perverse homotopy t-structure. Indeed, the full faithfulness follows from \Cref{smooth Artin motives} while the t-exactness follows from \Cref{AM.locality}.

We say that a connected scheme $X$ has the property $\mr{P}$ if one of the two following properties hold:
\begin{itemize}
    \item The scheme $X$ is of characteristic $p$.
    \item The prime number $p$ is invertible on $X$.
\end{itemize}

We say that a scheme $X$ has the property $\mr{P}$ if all its connected components have the property $\mr{P}$. The above discussion implies that in that case, the perverse homotopy t-structure induces a t-structure on $\DM^{smA}_{\et,c}(X,R)_\pp$ and we let $\mathrm{M}_{\mathrm{perv}}^{smA}(X,R)_{\pp}$ be the heart of this t-structure. 

Notice furthermore that if $F$ is any closed subscheme of $S$ which is of positive codimension, the induction hypothesis and \cite[1.1.8]{AM1} imply that the perverse homotopy t-structure induces a t-structure on $\DM^A_{\et,c}(F,R)_\pp$. We let $\mathrm{M}_{\mathrm{perv}}^{A}(F,R)_{\pp}$ be the heart of this induced t-structure. 

Using \Cref{AM.outil 1} below, it suffices to show that if $j\colon U\rar S$ is a dense affine open immersion with $U$ regular and having the property $\mr{P}$, letting $i\colon F\rar S$ be the reduced closed complementary immersion, for any object $M$ of $\mathrm{M}_{\mathrm{perv}}^{A}(F,R)_{\pp}$, any object $N$ of $\mathrm{M}_{\mathrm{perv}}^{smA}(U,R)_{\pp}$ and any map $$f\colon M\rar \Hlh^{-1}(\omega^0 i^*j_*N),$$ 
the kernel $K$ of $f$ is a constructible Artin motive.

If $\dim(S)=1$, the scheme $F$ is $0$-dimensional. For any profinite group $G$ and any noetherian ring $\Lambda$, the subcategory $\Rep^A(G,\Lambda)$ of $\mathrm{Mod}(G,\Lambda)$ is Serre (meaning that it is closed under sub-quotients and extensions). Therefore, \Cref{fields} and \Cref{smooth Artin motives} imply that the subcategory $\mathrm{M}_{\mathrm{perv}}^{A}(F,R)_{\pp}$ of $\DM^A_{\et}(F,R)_\pp^\heart$ is Serre.
Since $K$ is a subobject of $M$, it is an object of $\mathrm{M}_{\mathrm{perv}}^{A}(F,R)_{\pp}$ and thus it constructible.

If $\dim(S)=2$, the scheme $F$ is of dimension at most $1$. 
By \Cref{AM.outil 2} below, we have an exact sequence
$$0\overset{u}{\rar} M_1\rar \Hlh^{-1}(\omega^0 i^*j_*N)\overset{v}{\rar} M_2$$ in $\DM^A_{\et}(F,R)_\pp^\heart$ such that $M_1$ and $M_2$ are perverse Artin motives. Let $K_1$ be the kernel of the map $$M\overset{f}{\rar} \Hlh^{-1}(\omega^0 i^*j_*N)\rar M_2,$$ it is a perverse Artin motive.
Furthermore, the map $f_0\colon K_1\rar \Hlh^{-1}(\omega^0 i^*j_*N)$ induced by $f$ factors through $M_1$ since $$u\circ f_0=0$$ and the kernel of the induced map $$K_1\rar M_1$$ is $K$ which is therefore a perverse Artin motive.
\end{proof}
\begin{lemma}\label{AM.outil 1} Let $S$ be a reduced excellent scheme, let $R$ be a localization of $\Z$, let $\pp=(p)$ be a maximal ideal of $\Z$. 
Assume that:
\begin{enumerate}[label=(\roman*)]\item The perverse homotopy t-structure on $\DM^A_{\et}(F,R)_{\pp}$ induces a t-structure on the subcategory $\DM^A_{\et,c}(F,R)_{\pp}$ for any closed subscheme $F$ of $S$ which is of positive codimension.
\item If $j\colon U\rar S$ is a dense affine open immersion with $U$ regular and having the property $\mr{P}$, letting $i\colon F\rar S$ be the reduced closed complementary immersion, for any object $M$ of $\mathrm{M}^A_{\mathrm{perv}}(F,R)_{\mathfrak{p}}$, any object $N$ of $\mathrm{M}_{\mathrm{perv}}^{smA}(U,R)_{\pp}$ and any map $$f\colon M\rar \Hlh^{-1}(\omega^0i^*j_*N),$$ the kernel of $f$ is a perverse Artin motive on $F$.
\end{enumerate}

Then, the perverse homotopy t-structure induces a t-structure on $\DM^A_{\et,c}(S,R)_\pp$.
\end{lemma}
\begin{proof} 
If $j\colon U\rar S$ is an open immersion, denote by $\DM^A_{\et,c,U}(S,R)_\pp$ the subcategory of $\DM^A_{\et,c}(S,R)_\pp$ made of those objects $M$ such that $j^*M$ belongs to $\DM^{smA}_{\et,c}(U,R)_\pp$. Using \cite[3.5.1]{AM1}, every object of $\DM^A_{\et,c}(S,R)_\pp$ lies in some $\DM^A_{\et,c,U}(S,R)_\pp$ for some  dense open immersion $j\colon U\rar S$. Shrinking $U$ if needed, we can assume that $U$ is regular since $S$ is excellent. Shrinking $U$ further, we can assume that it has the property $\mr{P}$. Finally, using 
\cite[3.3.5]{aps}, 
we may assume that the morphism $j$ is affine.

Therefore, it suffices to show that for any dense affine open immersion $j\colon U\rar S$ with $U$ regular and having the property $\mr{P}$, the perverse homotopy t-structure induces a t-structure on $\DM^A_{\et,c,U}(S,R)_\pp$. Let now $j\colon U\rar S$ be such an immersion, let $i\colon F\rar S$ be the reduced closed complementary immersion. We let \begin{itemize}
\item $\mc{S}=\DM^A_{\et}(S,R)_\pp$,
\item $\mc{U}=\DM^A_{\et}(U,R)_\pp$ and $\mc{U}_0=\DM^{smA}_{\et,c}(U,R)_\pp$,
\item $\mc{F}=\DM^A_{\et}(F,R)_\pp$ and $\mc{F}_0=\DM^A_{\et,c}(F,R)_\pp$.
\end{itemize}

The category $\mc{S}$ is a gluing of the pair $(\mc{U},\mc{F})$ along the fully faithful functors $\omega^0 j_*$ and $i_*$ in the sense of \cite[A.8.1]{ha} and the perverse homotopy t-structure on $\mc{S}$ is obtained by gluing the perverse t-structures of $\mc{U}$ and $\mc{F}$ in the sense of \cite[1.4.10]{bbd} by \Cref{glueing perverse}. In addition, the perverse homotopy t-structure induces a t-structure on $\mc{U}_0$ and on $\mc{F}_0$ by assumption. Furthermore, since the morphism $j$ is affine, the functor $j_!\colon \mc{U}\rar \mc{S}$ is t-exact.

Finally \Cref{boundedness} ensures that the objects of $\DM^A_{\et,c}(S,R)_\pp$ are bounded with respect to the perverse homotopy t-structure. We conclude by using \cite[3.3.6]{aps}
\end{proof}

\begin{lemma}\label{AM.outil 2} Let $S$ be a reduced excellent connected scheme of dimension $2$, let $R$ be a localization of $\Z$, let $j\colon U\rar S$ be a dense affine open immersion with $U$ regular and having property $\mr{P}$, let $i\colon F\rar S$ be the reduced closed complementary immersion and let $N$ be an object of $\mathrm{M}_{\mathrm{perv}}^{smA}(U,R)_{\pp}$. 

Then, the object $\Hlh^{-1}(\omega^0 i^*j_* N)$ fits in $\DM^A_{\et}(F,R)_\pp^\heart$ into an exact sequence of the form $$0\rar M_1\rar \Hlh^{-1}(\omega^0 i^*j_* N) \rar M_2$$ where $M_1$ and $M_2$ belong to $\mathrm{M}^A_{\mathrm{perv}}(F,R)_{\mathfrak{p}}$.
\end{lemma}
\begin{proof} We use the convention that $\delta(S)=2$.
Replacing $S$ with the closure of a connected component of $U$, we may assume that $U$ is connected. 
Thus, since $U$ is regular and has property $\mr{P}$, we have $$\mathrm{M}_{\mathrm{perv}}^{smA}(U,R)_{\pp}=\rho_!\Loc_U(\Lambda_0)_\pp[2]$$ where $\Lambda_0=R$ if $p$ is invertible on $U$ and $\Lambda_0=R[1/p]$ if $U$ is of characteristic $p$.

Notice that the result holds when the ring $R$ is torsion by \Cref{omega^0 torsion} and therefore, we can assume that the ring $R$ is the localization of $\Z$. 

Denote by $\Lambda$ the ring $\Lambda_0\otimes_\Z \Z_\pp$. 
By \cite[3.1.12 \& 1.1.8]{AM1}, the canonical functor $$\Loc_U(\Lambda_0)_\pp\rar \Loc_U(\Lambda)$$ is an equivalence.

\textbf{Step 1:} 
Let $\xi$ be a geometric point of $U$ and let $$\varepsilon_!\colon \Rep^A(\pi_1^{\et}(U,\xi),\Lambda) \rar \Loc_U(\Lambda)$$ be the inverse of the equivalence of \cite[3.1.7]{AM1}, let $\alpha_!=\varepsilon_!\rho_!$ (this is the same functor as in \Cref{notations alpha_!}) and let $\alpha^!$ be the inverse of $\alpha_!$.

Let $P=\alpha^!N[-2]$ and let $P_{\mathrm{tors}}$ be the sub-representation of $P$ made of the torsion elements. There is a positive integer $n$ such that $p^nP_{\mathrm{tors}}$ vanishes; moreover, the $\Lambda[\pi_1^{\et}(U,\xi)]$-module $P_{\mathrm{tors}}$ is an Artin representation. Therefore, the $\Lambda[\pi_1^{\et}(U,\xi)]$-module $P/P_{\mathrm{tors}}$ is also an Artin representation and it is torsion free.

Using the Artin vanishing theorem (\Cref{Affine Lefschetz Q-cons}), the functor $\omega^0 i^*j_*$ is of perverse homotopy cohomological amplitude bounded below by $-1$. Therefore, the exact triangle $$\omega^0 i^*j_* \alpha_!P_{\mathrm{tors}}[2] \rar \omega^0 i^*j_* N \rar \omega^0 i^*j_* \alpha_!(P/P_{\mathrm{tors}})[2]$$ induces an exact sequence 
$$0\rar \Hlh^{1}(\omega^0 i^*j_* \alpha_!P_{\mathrm{tors}}) \rar \Hlh^{-1}(\omega^0 i^*j_* N) \rar \Hlh^{1}(\omega^0 i^*j_* \alpha_!(P/P_{\mathrm{tors}})).$$ 
Using \Cref{omega^0 torsion}, the object $\Hlh^{1}(\omega^0 i^*j_* \alpha_! P_{\mathrm{tors}})$ lies in  $\mathrm{M}^A_{\mathrm{perv}}(F,R)_{\mathfrak{p}}$ and therefore, it suffices to prove that the object $\Hlh^{1}(\omega^0 i^*j_* \alpha_! (P/P_{\mathrm{tors}}))$ of $\DM^A_{\et}(F,R)_\pp^\heart$ is a subobject of an object of $\mathrm{M}^A_{\mathrm{perv}}(F,R)_{\mathfrak{p}}$.

Since the Artin representation $P/P_{\mathrm{tors}}$ is torsion free, the canonical map $$P/P_{\mathrm{tors}}\rar P/P_{\mathrm{tors}}\otimes_{\Lambda}\Q$$ is injective. The category $\Rep^A(\pi_1^{\et}(U,\xi),\Q)$ is semi-simple by Maschke's lemma. 
Therefore, the representation $P/P_{\mathrm{tors}}\otimes_{\Lambda}\Q$ is a direct factor of a representation $$Q_\Q:=\Q[G_1]\bigoplus \cdots \bigoplus \Q[G_n]$$ where $G_1,\ldots,G_n$ are finite quotients of $\pi_1^{\et}(U,\xi)$. 

Denote by $\lambda_\Q$ the map $P/P_{\mathrm{tors}}\otimes_{\Lambda}\Q\rar Q_\Q$. There is an integer $m$ such that the image of $P/P_{\mathrm{tors}}$ in $Q_\Q$ is contained in the lattice $$Q_0:=\frac{1}{p^m}\left(\Lambda[G_1]\bigoplus \cdots \bigoplus \Lambda[G_n]\right)$$ in $Q_\Q$. The representation $Q_0$ is isomorphic to $\Lambda[G_1]\bigoplus \cdots \bigoplus \Lambda[G_n]$. Thus, we get an embedding $$P/P_{\mathrm{tors}}\rar \Lambda[G_1]\bigoplus \cdots \bigoplus \Lambda[G_n].$$

Since the functor $\omega^0 i^*j_*$ is of perverse homotopy cohomological amplitude bounded below by $-1$, we get an embedding $$\Hlh^{1}(\omega^0 i^*j_* \alpha_! (P/P_{\mathrm{tors}}))\rar \bigoplus_{i=1}^n \Hlh^{1}(\omega^0 i^*j_* \alpha_! \Lambda[G_i]).$$
Furthermore, if $1\leqslant i\leqslant n$, the motive $\alpha_! \Lambda[G_i]$ is of the form $f_*\un_V$ where $f\colon V\rar U$ is a finite étale map. 
Therefore, it suffices to show that the following claim: "If $f\colon V\rar U$ is a finite étale map, and if $N=f_*\un_V$, the object $\Hlh^{1}(\omega^0 i^*j_* N)$ belongs to $\mathrm{M}^A_{\mathrm{perv}}(F,R)_{\mathfrak{p}}$."

\textbf{Step 2:} In this step we show that to prove the above claim, we can assume $S$ to be normal and that $N=\un_U$. 

Let $\overline{V}$ be the relative normalization of $S$ in $V$. Since $S$ is excellent and noetherian, it is Nagata by \cite[033Z]{stacks}. Thus, \cite[0AVK]{stacks} ensures that the structural map $\nu:\overline{V}\rar S$ is finite. Furthermore, the scheme $\overline{V}$ is normal by \cite[035L]{stacks}.
Consider the following commutative diagram
$$\begin{tikzcd} V \ar[d,"f"]\ar[r,"\gamma"] & \overline{V} \ar[d,"\nu"]& F' \ar[l,sloped,"\iota"] \ar[d,"\nu_F"] \\
U \ar[r,"j"]& S & F \ar[l,sloped,"i"]
\end{tikzcd}$$
made of cartesian squares.
Now, we have $$\omega^0 i^*j_*N=\omega^0i^*j_*f_*\un_V=\omega^0 i^*\nu_* \gamma_*\un_V=(\nu_F)_* \omega^0 \iota^*\gamma_*\un_V.$$
Since $\nu_F$ is finite, it is perverse homotopy t-exact. Thus, we have $$\Hlh^{1}(\omega^0i^*j_* N)= (\nu_F)_*\Hlh^{1}(\omega^0\psi^*\gamma_* \un_V).$$
Therefore, we can replace $S$ with $\overline{V}$ and $U$ with $V$ in order to assume that $S$ is normal and that $N=\un_U$.

\textbf{Step 3:} In this step, we study the singularities of $S$. 

First, notice that if $Y\rar S$ is a closed immersion and if $F \cap Y=\varnothing$, we can factor $i$ into a closed immersion $u\colon F\rar S\setminus Y$ and an open immersion $\xi\colon S\setminus Y \rar S$. Letting $j'\colon U\setminus Y\rar S \setminus Y$ be the canonical immersion, smooth base change implies that $$\omega^0i^*j_*\un_U=\omega^0 u^* \xi^*j_*\un_U=\omega^0 u^* j'_*\un_{U\setminus Y}.$$

We can therefore remove any closed subset which does not intersect $F$ from $S$ without changing $\Hlh^{1}(\omega^0 i^*j_* \un_U).$
Since $S$ is normal, its singular locus lies in codimension $2$ and we can therefore assume that it is contained in $F$. 

Since $S$ is excellent, Lipman's Theorem on embedded resolution of singularities applies (see \cite[0BGP,0BIC,0ADX]{stacks} for precise statements). We get a cdh-distinguished square

$$\begin{tikzcd}
    E\ar[r,"i_E"]\ar[d,"p"] &\widetilde{S}\ar[d,"f"]\\
    F \ar[r,"i"] & S
\end{tikzcd}$$
such that the map $f$ induces an isomorphism over $U$, the scheme $\widetilde{S}$ is regular and the subscheme $E$ of $\widetilde{S}$ is a simple normal crossing divisor.

Let $\gamma\colon U\rar \widetilde{S}$ be the complementary open immersion of $i_E$. We have $j=f\circ \gamma$. 
Therefore, we get $$\omega^0 i^*j_* \un_{U}=\omega^0 i^*f_*\gamma_*\un_U=\omega^0 p_* i_E^*\gamma_*\un_{U}.$$

We have a localization exact triangle
$$(i_E)_*i_E^!\un_{\widetilde{S}}\rar \un_{\widetilde{S}}\rar \gamma_* \un_{U},$$
applying $\omega^0 p_*i_E^*$, we get an exact triangle
$$\omega^0 p_*i_E^!\un_{\widetilde{S}}\rar \omega^0 p_*\un_{E}\rar \omega^0 i^*j_* \un_{U}$$
and therefore, we get an exact sequence 
$$\Hlh^1\left(\omega^0 p_*i_E^!\un_{\widetilde{S}}\right)\rar \Hlh^1(\omega^0 p_*\un_{E})\rar \Hlh^1(\omega^0 i^*j_* \un_{U})\rar \Hlh^2\left(\omega^0 p_*i_E^!\un_{\widetilde{S}}\right)$$
and an equivalence $$\Hlh^0\left(\omega^0 p_*i_E^!\un_{\widetilde{S}}\right)= \Hlh^0(\omega^0 p_*\un_{E})$$ (to get this last equivalence, we use the Artin vanishing theorem).

Since $E$ is a simple normal crossing divisor, there is a finite set $I$ and regular 1-dimensional closed subschemes $E_i$ of $\widetilde{X}$ for $i\in I$, such that $E=\bigcup\limits_{i \in I} E_i$. Choose a total order on the finite set $I$. If $i<j$, write $E_{ij}=E_i\cap E_j$. Consider for $J\subseteq I$ of cardinality at most $2$ the obvious diagram: 
$$\begin{tikzcd}
    E_J \ar[r,"u_{E_J}"] \ar[rd,swap,"p_J"] & E \ar[d,"p"] \\
    & F
\end{tikzcd}$$

By cdh-descent and absolute purity, we have an exact triangle 
$$\bigoplus_{i<j} (u_{E_{ij}})_* \un_{E_{ij}}(-2)[-4] \rar \bigoplus_{i \in I} (u_{E_i})_* \un_{E_i}(-1)[-2] \rar i_E^!\un_{\widetilde{S}}$$
and applying $\omega^0p_*$ yields an exact triangle 
$$\bigoplus_{i<j} \omega^0\left((p_{ij})_* \un_{E_{ij}}(-2)\right)[-4]\rar \bigoplus_{i \in I} \omega^0\left((p_i)_* \un_{E_i}(-1)\right)[-2] \rar \omega^0p_*i_E^!\un_{\widetilde{S}}.$$

Let $J\subseteq I$ be of cardinality at most $2$ and let $c$ be the cardinality of $J$. We have an exact triangle 
$$\omega^0\left((p_J)_* (\un_{E_J}\otimes_\Z \Q/\Z)(-c)\right)[-1]\rar \omega^0\left((p_J)_*\un_{E_J}(-c)\right)\rar \omega^0\left((p_J)_* (\un_{E_J}\otimes_\Z \Q)(-c)\right)$$

\cite[3.9]{plh} ensures that the motive $\omega^0\left((p_J)_* (\un_{E_J}\otimes_\Z \Q)(-c)\right)$ vanishes. Therefore, using \Cref{omega^0 torsion}, we get an equivalence $$\omega^0\left((p_J)_*\un_{E_J}(-c)\right)=(p_J)_* (\un_{E_J}\otimes_\Z \Q/\Z)(-c)[-1]$$

But then, using \Cref{perverse homotopy torsion} and the usual properties of the perverse t-structure described in \cite[4.2.4]{bbd}, we get $$(p_J)_* (\un_{E_J}\otimes_\Z \Q/\Z)(-c)\geqslant_{hp} 0.$$ Therefore, we have $$\omega^0\left((p_J)_*\un_{E_J}(-c)\right)\geqslant_{hp}1$$ 
and thus, we have $$\omega^0\left((p_J)_*\un_{E_J}(-c)\right)[-2c]\geqslant_{hp}1+2c$$ which yields $$\omega^0p_*i_E^!\un_{\widetilde{S}}\geqslant_{hp} 3.$$ 

Hence, we get an equivalence $$\Hlh^1(\omega^0 p_*\un_{E})= \Hlh^1(\omega^0 i^*j_* \un_{U})$$
and the motive $\Hlh^0(\omega^0 p_*\un_{E})$ vanishes.

\textbf{Step 4:} In this step we compute the motive $\Hlh^1(\omega^0 p_*\un_{E})$ and we finish the proof. By cdh-descent, we have an exact triangle

$$\un_E \rar\bigoplus_{i \in I}(u_{E_i})_*(u_{E_i})^* \un_E \rar  \bigoplus\limits_{i< j} (u_{E_{ij}})_*(u_{E_{ij}})^* \un_E.$$
and applying $\omega^0 p_*$, we get an exact triangle
$$\omega^0 p_* \un_E \rar\bigoplus_{i \in I}\omega^0(p_i)_* \un_{E_i} \rar  \bigoplus\limits_{i< j} (p_{ij})_* \un_{E_{ij}}.$$

Let $$I_0=\{ i \in I \mid \delta(p(E_i))=0\}.$$ 
If $i \in I\setminus I_0$, we have $\delta(E_i)=1$ and the morphism $p_i$ is finite. If $i\in I_0$, the map $p_i$ factors through the singular locus $Z$ of $S$ which is $0$-dimensional and we can write $p_i=\iota \circ \pi_i$ with $\iota\colon Z\rar S$ the closed immersion. Finally, notice that $\delta(Z)=0$ and thus, for any $i<j$, we have $\delta(E_{ij})=0$. Hence, since the motive $\Hlh^0(\omega^0 p_*\un_{E})$ vanishes, we get an exact sequence 
$$0\rar P \rar \bigoplus\limits_{i< j} (p_{ij})_* \un_{E_{ij}} \rar  \Hlh^1(\omega^0 p_* \un_E) \rar\bigoplus_{i \in I\setminus I_0}(p_i)_* \un_{E_i} \bigoplus Q \rar 0 $$
where $P=\iota_*\bigoplus_{i \in I_0}\Hlh^0(\omega^0(\pi_i)_* \un_{E_i})$ and $Q=\iota_*\bigoplus_{i \in I_0}\Hlh^1(\omega^0(\pi_i)_* \un_{E_i})$. 

If $i\in I_0$, let $$E_i\rar F_i \overset{q_i}{\rar} Z$$ be the Stein factorization of $\pi_i$. \Cref{omega^0 corps} yields 
$$P=\iota_*\bigoplus_{i \in I_0}(q_i)_* \un_{F_i}$$ and $$Q=0.$$

Hence, the objects $P$ and $Q$ belong to $\mathrm{M}^A_{\mathrm{perv}}(F,R)_{\mathfrak{p}}$ and therefore, so does the object $\Hlh^1(\omega^0 p_* \un_E)$ which finishes the proof.
\end{proof}

\begin{proposition}\label{Affine Lefschetz} Let $S$ be an excellent scheme of dimension $2$ or less, let $f\colon X\rar S$ be a quasi-finite affine morphism of schemes and let $R$ be a regular ring. Assume that $X$ is nil-regular.

Then, the functor $$f_!\colon \DM^{smA}_{\et,c}(X,R)\rar \DM^A_{\et,c}(S,R)$$ is t-exact when both sides are endowed with the perverse homotopy t-structure.
\end{proposition}
\begin{proof} This follows from \Cref{Affine Lefschetz Q-cons}.    
\end{proof}

\begin{example}\label{superbe exemple} Let $X$ be a normal scheme of dimension $4$ with a single singular point $x$ of codimension $4$, let $k$ be the residue field at $x$, let $p$ be the characteristic exponent of $k$ and let $f\colon \widetilde{X}\rar X$ be a resolution of singularities of $X$.

Assume that the exceptional divisor above $x$ is isomorphic to a smooth $k$-scheme $E$ such that $b_1(E,\ell)\neq 0$ for some prime number $\ell$ distinct from the characteristic exponent of $E$ and that the ring of coefficients is $\Z$. Take the convention that $\delta(X)=4$.

We have a cdh-distinguished square
$$\begin{tikzcd}E \ar[r,"i_E"] \ar[d,"\pi"] & \widetilde{X}\ar[d,"f"]\\
\Spec(k) \ar[r,"i"] & X
\end{tikzcd}$$

which yields an exact triangle
$$\un_X\rar \omega^0 f_*\un_{\widetilde{X}} \oplus i_* \un_k \rar i_* \omega^0 \pi_*\un_E.$$

Furthermore, letting $U=X\setminus\{ x\}= \widetilde{X}\setminus E$ and letting $j\colon X\setminus\{ x\}\rar X$ be the open immersion, the localization property \eqref{AM.colocalization} yields an exact triangle
$$i_* \omega^0 \pi_* (i_E)^! \un_{\widetilde{X}} \rar \omega^0 f_* \un_{\widetilde{X}} \rar \omega^0j_* \un_U.$$

Since by \Cref{AM.t-adj,delta(S)}, the Artin motive $\omega^0j_* \un_U$ lies in degree at least $4$, we get that $$\Hlh^i(\omega^0 f_*\un_{\widetilde{X}})=i_* \Hlh^i(\omega^0 \pi_* (i_E)^! \un_{\widetilde{X}})$$
for $i<4$. 

Using the absolute purity property, we have $$\omega^0 \pi_* (i_E)^! \un_{\widetilde{X}}=\omega^0 \left( \pi_* \un_{E}(-1)\right)[-2]$$ and therefore, by \Cref{omega^0 corps}, we get $$\Hlh^i(\omega^0 f_*\un_{\widetilde{X}})=\begin{cases} 0 &\text{ if }i<3 \\
i_* \left(\alpha_!\hspace{1mm} \Q/\Z[1/p](-1) \right) &\text{ if }i=3,
\end{cases}
$$
using \Cref{notations alpha_!}.
On the other hand, we have $$\Hlh^2(\omega^0 \pi_*\un_E)=i_*\alpha_! \hspace{1mm} h^1(E,\Z)$$ by \Cref{omega^0 corps}. Moreover, the Artin motive $i_* \un_k$ is in the heart of the perverse homotopy t-structure. 
Hence, we get an exact sequence $$0\rar i_* \alpha_! \mu_1(E,\Z) \rar \Hlh^3(\un_X)\rar i_* M\rar 0,$$ where $M$ is the kernel of the map $\Hlh^3(\omega^0\pi_* (i_E)^! \un_{\widetilde{X}})\rar \Hlh^3(\omega^0\pi_* \un_E)$. 
Thus, the motive $\Hlh^3(\un_X)$ is in the image of the functor $i_*$.
Since the representation $h^1(E,\Z)$ of $G_k$ is not of finite type as a $\Z$-module (it has $\Z(\ell^\infty)^{b_1(E,\ell)}$ as a sub-module), the Artin motive $\Hlh^3(\un_X)$ cannot be constructible.
\end{example}
\begin{example}\label{exemple dim 4} Let $X$ be the affine cone over an abelian threefold $E$ embedded in some $\mb{P}^N$. The scheme $X$ has a single singular point $x$ and the blow up of $X$ at $x$ is a resolution of singularities of $X$ with exceptional divisor $E$.

By Noether's normalization lemma, there is a finite map $X\rar \mb{A}^4_k$. Thus, the perverse homotopy t-structure does not induce a t-structure on $\mc{DM}^A_{\et,c}(\mb{A}^4_k,\Z)$. As $\mb{A}^4_k$ is a closed subscheme of $\mb{A}^n_k$ for $n\geqslant 4$, the perverse t-structure does not induce a t-structure on $\mc{DM}^A_{\et,c}(\mb{A}^n_k,\Z)$ if $n \geqslant 4$.
\end{example}

\begin{remark} Loosely speaking, if the scheme $X$ becomes more singular, the perverse homotopy cohomology sheaves of $\un_X$ should become more complicated. Here, the simplest possible singularity on a scheme of dimension $4$ over a finite field already renders the cohomology sheaves not constructible. Therefore, \Cref{AM.main theorem} should not hold for schemes of dimension $4$ or more. The case of $3$-folds, however, remains open.
\end{remark}

\subsection{Gluing and t-exactness of the realization}
The following description of perverse Artin motives is very close to that of Artin perverse sheaves from \cite{aps}. It is very loosely inspired by Beilinson's gluing method from \cite{howtoglue}. 

Let $M$ be an Artin perverse motive over an excellent base scheme $S$ of dimension at most $2$. \cite[3.5.1]{AM1} gives a stratification of $S$ with regular strata and such that for any stratum $T$, the motive $M|_T$ is a constructible smooth Artin motive. We can recover the motive $M$ from the motives $M|_T$ and additional data.

We take the convention that $\delta(S)=2$ and assume for simplicity that the scheme $S$ is reduced. Let $U$ be the (disjoint) union of the strata containing the generic points of $S$. Without loss of generality, we can assume that $U$ is a single stratum over which the motive $M$ a constructible smooth Artin motive. Shrinking the open subset $U$, we can also assume that the immersion $j\colon U\rar S$ is affine. Write $i\colon F\rar S$ the reduced complementary closed immersion. The motive $M|_U$ is a perverse Artin motive and the constructible Artin motive $M|_F$ is in perverse degree $[-1,0]$.

By localization, we have an exact triangle $$i_*M|_F[-1] \rar j_!M|_U\rar M.$$

Thus, by \Cref{Affine Lefschetz}, we have an exact sequence $$0\rar i_*\Hlh^{-1}M|_F\rar j_! M|_U\rar M \rar i_* \Hlh^0 M|_F\rar O$$
Hence, we can recover $M$ as the cofiber of a map $i_*M|_F[-1]\rar j_!M|_U$ such that the induced map $i_*\Hlh^{-1}M|_F\rar j_!M|_U$ is a monomorphism. 

Conversely, a smooth Artin perverse sheaf $M_U$ over $U$, a constructible Artin motive $M_F$ in perverse degrees $[-1,0]$ over $F$ and a connecting map $i_*M_F[-1]\rar j_!M_U$ such that the induced map $i_*\Hlh^{-1}M_F\rar j_*M_U$ is a monomorphism, give rise to a unique perverse Artin motive $M$ over $S$.

Now, recall that we have an exact triangle

$$j_!\rar \omega^0j_* \rar i_* \omega^0i^* j_*$$ and that $j^*i_*=0$. Therefore, we have an exact sequence 
$$0\rar i_*\Hlh^{-1}(\omega^0 i^*j_* M_U) \rar j_! M_U \rar \Hlh^0(\omega^0 j_* M_U) \rar i_* \Hlh^0(\omega^0 i^*j_* M_U)\rar 0.$$
and we have $$\Hom(i_*M_F[-1],j_!M_U)=\Hom(M_F,\omega^0 i^*j_* M_U).$$

Therefore, if $\phi\colon i_*M_F[-1]\rar j_!M_U$ is any map, the map $\Hlh^0(\phi)$ is a monomorphism if and only if the induced map $$\Hlh^{-1}(M_F)\rar \Hlh^{-1}(\omega^0 i^*j_* M_U)$$ is a monomorphism.

Similarly, since the stratification of $S$ induces a stratification of $F$, the Artin motive $M_F$ is uniquely encoded by the data of a smooth constructible Artin motive $M_1$ concentrated in perverse degrees $[-1,0]$ over an affine open subset $U_1$ of $F$, a constructible Artin motive $N_1$ in perverse degrees $[-2,0]$ on $F_1=F\setminus U_1$ and a connecting map with a similar condition.

This discussion provides the following proposition:

\begin{proposition}\label{description d=2} Let $S$ be an excellent scheme of dimension $2$ or less and let $R$ be a regular ring. Then, perverse Artin étale motive over $S$ are uniquely encoded by the following data.

\begin{itemize}
    \item A triple $(U_0,U_1,U_2)$ of locally closed subschemes of $S$ such that each $U_k$ is nil-regular, purely of codimension $k$, is open in the scheme $F_{k-1}:=S\setminus (U_0\cup \cdots \cup U_{k-1})$ and such that the open immersion $$j_k\colon U_k\rar F_{k-1}$$ is affine (take the conventions that $F_{-1}=S$). Denote by $$i_k\colon F_k\rar F_{k-1}$$ the closed immersion.
    \item A triple $(M_0,M_1,M_2)$ of constructible smooth Artin motives on $U_i$ such that each $M_i$ is placed in perverse degree $[-i,0]$.
    \item Two connecting maps 
    	$$ \begin{cases} \phi_2\colon M_{2}\rar \omega^0 i_1^*(j_1)_*M_1 \\
    	\phi_1\colon M_{12} \rar \omega^0 i_0^*(j_0)_*M_0
    	\end{cases}$$ 
such that $\Hlh^{-k}(\phi_k)$ is a monomorphism and where $M_{12}$ is the cofiber of the map $$(i_1)_* M_{2}[-1]\rar (j_1)_! M_1$$ induced by $\phi_2$.
    \end{itemize}
\end{proposition}

\begin{remark} If $\dim(S)\leqslant 1$, $U_2=\emptyset$, $M_2=0$ and the data of $\phi_1$ is trivial so that $M_{12}=M_1$. 
\end{remark}

This description yields the t-exactness of the $\ell$-adic realization functor.

\begin{theorem}\label{t-ex of l-adic real} Let $(S,\delta)$ be an excellent d-scheme of dimension $2$ or less, let $R$ be a localization of $\Z$, let $\ell$ be a prime number. Then, the reduced $\ell$-adic realization functor 
$$\overline{\rho}_\ell\colon \DM_{\et,c}^A(S,R)\rar \mc{D}^b_c(S[1/\ell],R\otimes_\Z \Z_\ell)$$ of \Cref{reduced l-adic real} is t-exact when the left hand side is endowed with the perverse homotopy t-structure and the right hand side is endowed with the perverse t-structure.
\end{theorem}

\begin{proof} First, notice that $\rho_\ell$ is right t-exact because the generators of the perverse homotopy t-structure are sent to perverse t-non-positive objects. Using the same trick as in the proof of \cite[4.2.7]{AM1}, we can assume that $R=\Q$. We may also assume that $\ell$ is invertible on $S$. Now, it suffices to show that any perverse Artin étale motive on $S$ is sent to a complex which is non-negative for the perverse t-structure. Let $M$ be such a motive on $S$.

By \Cref{description d=2}, the motive $M$ is given by
\begin{itemize}
    \item A triple $(U_0,U_1,U_2)$ of locally closed subschemes of $S$ such that each $U_k$ is nil-regular, purely of codimension $k$, is open in the scheme $F_{k-1}:=S\setminus (U_0\cup \cdots \cup U_{k-1})$ and such that the open immersion $$j_k\colon U_k\rar F_{k-1}$$ is affine (take the conventions that $F_{-1}=S$). Denote by $$i_k\colon F_k\rar F_{k-1}$$ the closed immersion.
    \item A triple $(M_0,M_1,M_2)$ of constructible smooth Artin motives on $U_i$ such that each $M_i$ is placed in perverse degree $[-i,0]$.
    \item Two connecting maps 
    	$$ \begin{cases} \phi_2\colon M_{2}\rar \omega^0 i_1^*(j_1)_*M_1 \\
    	\phi_1\colon M_{12} \rar \omega^0 i_0^*(j_0)_*M_0
    	\end{cases}$$ 
such that $\Hlh^{-k}(\phi_k)$ is a monomorphism and where $M_{12}$ is the cofiber of the map $(i_1)_* M_{2}[-1]\rar (j_1)_! M_1$ induced by $\phi_2$.
    \end{itemize}

By \Cref{AM.locality}, the $\ell$-adic complex $\rho_{\ell}(M)|_{U_i}=\rho_{\ell}(M_i)$ is in perverse degree $[-i,0]$. To show that $\rho_{\ell}(M)$ is a perverse sheaf, it suffices to show that the natural maps $$ \begin{cases} \psi_2\colon \rho_{\ell}(M_{2})\rar \rho_{\ell}(i_1^*(j_1)_*M_1) \\
    	\psi_1\colon \rho_{\ell}(M_{12}) \rar  \rho_{\ell}(i_0^*(j_0)_*M_0)
    	\end{cases}$$ 
are such that ${}^{p} \Hl^{-k}(\psi_k)$ is a monomorphism.

Notice now that if $k\in \{ 1,2\}$ and $M_{\geqslant k}=\begin{cases} M_{12} \text{ if }k=1 \\ M_2 \text{ if }k=2\end{cases}$, we have a commutative diagram
$$\begin{tikzcd}\rho_{\ell}(M_{\geqslant k}) \ar[r,"\rho_{\ell}(\phi_{k})"] \ar[rd,"\psi_k"] & \rho_{\ell}(\omega^0 i^*_{k-1}(j_{k-1})_* M_{k-1})\ar[d,"\rho_{\ell}(\eta(i^*_{k-1}(j_{k-1})_* M_{k-1}))"] \\
& \rho_{\ell}(i^*_{k-1}(j_{k-1})_* M_{k-1})
\end{tikzcd}$$
with $\eta\colon \iota \circ \omega^0 \to \mathrm{id}$ the counit of the adjunction that defines $\omega^0$ (see \Cref{def omega^0}).

Working by induction on $\dim(S)$, we may assume that the maps $$\DM^A_{\et,c}(F_{k-1},\Q)\rar \mc{D}^b_c(F_{k-1},\Q_\ell)$$ are t-exact. Thus, the map ${}^{p} \Hl^{-k}(\rho_{\ell}(\phi_{k}))$ is a monomorphism. To prove the theorem, it therefore suffices to show that the map ${}^p \Hl^{-k}\left[\rho_{\ell}(\eta(i^*_{k-1}(j_{k-1})_* M_{k-1}))\right]$
is a monomorphism.

Write $$M_1=A[1]\bigoplus B[0],$$ which is possible since the category of smooth Artin motives with coefficients in $\Q$ is semi-simple by Maschke's theorem. Then, using \Cref{Affine Lefschetz Q-cons} and the induction hypothesis, the $\ell$-adic complex $\rho_{\ell}(\omega^0 i^*_1(j_1)_*B)$ is in degree at least $-1$ with respect to the perverse t-structure. By the usual Artin vanishing theorem, so is $\rho_{\ell}(i^*_1(j_1)_*B)$, since  by \Cref{AM.locality}, the $\ell$-adic complex $\rho_{\ell}(B)$ is a perverse sheaf. Hence, the map ${}^{p} \Hl^{-2}\left[\rho_{\ell}(\eta(i^*_1(j_1)_* M_{1}))\right]$ is a monomorphism if and only if the map
$${}^{p} \Hl^{-1}\left[\rho_{\ell}(\eta)\right]\colon {}^{p} \Hl^{-1}\left[\rho_{\ell}(\omega^0 i_1^*(j_1)_* A)\right]\longrightarrow {}^{p} \Hl^{-1}\left[\rho_{\ell}(i_1^*(j_1)_* A)\right]$$ is a monomorphism. Hence, to prove the theorem, it suffices to show that the maps $${}^{p} \Hl^{-1}\left[\rho_{\ell}(\eta)\right]\colon {}^{p} \Hl^{-1}\left[\rho_{\ell}(\omega^0 i_1^*(j_1)_* A)\right]\longrightarrow {}^{p} \Hl^{-1}\left[\rho_{\ell}(i_1^*(j_1)_* A)\right]$$
and 
$${}^{p} \Hl^{-1}\left[\rho_{\ell}(\eta)\right]\colon {}^{p} \Hl^{-1}\left[\rho_{\ell}(\omega^0 i_0^*(j_0)_* M_0)\right]\longrightarrow {}^{p} \Hl^{-1}\left[\rho_{\ell}(i_0^*(j_0)_* M_0)\right]$$ are monomorphisms which follows from \cref{634} below.
\end{proof}
\begin{lemma}\label{634}  Let $(S,\delta)$ be an excellent scheme of dimension $2$ or less, let $\ell$ be a prime number, let $i\colon F\rar S$ be a closed immersion and let $j\colon U\rar S$ be the open complementary immersion. Assume that the prime number $\ell$ is invertible on $S$, that the morphism $j$ is affine and that the scheme $U$ is nil-regular. Let $M$ be an object of $\mathrm{M}^{smA}(U,\Q)$. Then, the map $${}^{p} \Hl^{-1}\rho_{\ell}(\delta)\colon {}^{p} \Hl^{-1}\left[\rho_{\ell}(\omega^0 i^*j_* M)\right]\longrightarrow {}^{p} \Hl^{-1}\left[\rho_{\ell}(i^*j_* M)\right]$$ is a monomorphism.
\end{lemma}
\begin{proof}Let $d=\dim(S)$. Take the convention that $\delta(S)=d$. As in the proof of \Cref{lefschetz affine 2Q}, we can assume that $M=\un_U[d]$, that the scheme $S$ is normal and that the scheme $U$ is regular. 

\textbf{Case 1:} Assume that $\dim(S)=1$. Since the closed subscheme $F$ is $0$-dimensional, we can assume it to be reduced to a point $\{ x\}$. Recall that the $2$-morphism $$\Theta\colon \rho_!i^*j_*\rar \omega^0 i^*j_*\rho_!$$ of \Cref{nearby cycle} is defined so that we have a commutative diagram
$$\begin{tikzcd} \rho_!i^*j_* \ar[r,"\Theta"] \ar[rd,"Ex"] & \omega^0 i^*j_* \rho_!\ar[d,"\eta(i^*j_*\rho_!)"] \\
& i^*j_*\rho_!
\end{tikzcd}$$ where $Ex$ is the exchange map.

Moreover, the map $\Theta(\underline{\Q})$ is an equivalence by \Cref{nearby cycle}. Therefore, it suffices to show that the map $${}^{p} \Hl^0 \rho_{\ell}(Ex)\colon {}^{p} \Hl^{0}\rho_{\ell}(\rho_! i^*j_* \underline{\Q})\longrightarrow {}^{p} \Hl^{0}(i^*j_* \underline{\Q_\ell})$$ is a monomorphism. Now, this map can be by \cite[VIII.5]{sga4} identified with the canonical map
$$\underline{\Q_\ell}\rar \Hl^0_{\et}(\mc{O}_{S,x}^{sh} \times_S U,\Q_\ell)$$ where $\mc{O}_{S,x}^{sh}$ is the strict henselization of the local ring of $S$ at $x$. Since $S$ is normal, this map is an isomorphism by \cite[2.1]{deninger}.

\textbf{Case 2:} Assume that $\dim(S)=2$. Since $S$ is normal its singular points are of codimension $2$. Let $f\colon \widetilde{S}\rar S$ be a resolution of singularities of $S$ such that the inverse image of $F$ is a simple normal crossing divisor (which exists by Lipman's theorem). We have a commutative diagram

$$\begin{tikzcd}
U \arrow[d, equal] \arrow[r, "\gamma"] & \widetilde{S} \arrow[d, "f"] & E \arrow[d, "p"] \arrow[l, swap,"\iota"] \\
U \arrow[r, "j"]                               & S                           & F \arrow[l, swap,"i"]                  
\end{tikzcd}$$
 made of cartesian squares which yields a commutative diagram
$$\begin{tikzcd}\rho_{\ell}(\omega^0i^*j_*\un_U)\ar[r] \ar[d,"\rho_{\ell}(\eta)"] & \rho_{\ell}(\omega^0 p_*\iota^*\gamma_*\un_U) \ar[d,"\rho_{\ell}(\eta)"] \\
\rho_{\ell} (i^*j_*\un_U) \ar[r]& \rho_{\ell}(p_*\iota^*\gamma_* \un_U)
\end{tikzcd}$$ 
where the horizontal maps are equivalences. Thus, it suffices to show that the map $${}^{p} \Hl^{1}\rho_{\ell}(\eta)\colon {}^{p} \Hl^{1}\rho_{\ell}(\omega^0 p_*\iota^*\gamma_*\un_U)\longrightarrow {}^{p} \Hl^{1}\rho_{\ell}(p_*\iota^*\gamma_* \un_U)$$ is a monomorphism.

The localization triangle \eqref{AM.colocalization} induces a morphism of exact triangles:

\[\begin{tikzcd}
	{\rho_{\ell}(\omega^0p_*\iota^!\un_{\widetilde{S}})} & {\rho_{\ell}(\omega^0p_*\un_E)} & {\rho_{\ell}(\omega^0 p_*\iota^*\gamma_*\un_U)} \\
	{\rho_{\ell}(p_*\iota^!\un_{\widetilde{S}})} & {\rho_{\ell}(p_*\un_E)} & {\rho_{\ell}(p_*\iota^*\gamma_*\un_U)}
	\arrow["{\rho_{\ell}(\eta)}", from=1-3, to=2-3]
	\arrow[from=1-2, to=1-3]
	\arrow["{\rho_{\ell}(\eta)}"', from=1-2, to=2-2]
	\arrow[from=2-2, to=2-3]
	\arrow[from=1-1, to=1-2]
	\arrow["{\rho_{\ell}(\eta)}"', from=1-1, to=2-1]
	\arrow[from=2-1, to=2-2]
\end{tikzcd}\]

By \Cref{az211}, we have $$\omega^0p_*\iota^!\un_{\widetilde{S}}=0.$$ Moreover, using the cdh-descent argument which appears the last step of the proof of \Cref{lefschetz affine 2Q} we can show that $${}^{p} \Hl^1(\rho_{\ell}(p_*\iota^!\un_{\widetilde{S}}))=0.$$  Hence, we get a commutative diagram with exact rows

\[\begin{tikzcd}
	{0} & {{}^p\Hl^1\rho_{\ell}(\omega^0p_*\un_E)} & {{}^p\Hl^1\rho_{\ell}(\omega^0 p_*\iota^*\gamma_*\un_U)} & 0\\
	{0} & {{}^p\Hl^1 \rho_{\ell}(p_*\un_E)} & {{}^p\Hl^1 \rho_{\ell}(p_*\iota^*\gamma_*\un_U)}&
	\arrow[from=1-3, to=1-4]
	\arrow["{{}^p\Hl^1\rho_{\ell}(\eta)}", from=1-3, to=2-3]
	\arrow[from=1-2, to=1-3]
	\arrow["{{}^p\Hl^1\rho_{\ell}(\eta)}"', from=1-2, to=2-2]
	\arrow[from=2-2, to=2-3]
	\arrow[from=1-1, to=1-2]
	\arrow[from=2-1, to=2-2]
\end{tikzcd}\]

Therefore, it suffices to show that the map $${}^p\Hl^1\rho_{\ell}(\eta)\colon {}^p\Hl^1\rho_{\ell}(\omega^0p_*\un_E) \rar {}^p\Hl^1 \rho_{\ell}(p_*\un_E)$$ is a monomorphism.

Write $E=\bigcup\limits_{i \in J} E_i$ with $J$ finite, where for any index $i$, the scheme $E_i$ is regular and of codimension $1$, where for any distinct indices $i$ and $j$ the scheme $E_{ij}=E_i\cap E_j$ is of codimension $2$ and regular and where the intersections of $3$ distinct $E_i$ are empty.

By cdh-descent, we have a morphism of exact triangles
$$\begin{tikzcd}\rho_{\ell}\omega^0p_*\un_E \ar[r] \ar[d,"\rho_{\ell}(\eta)"]& \bigoplus\limits_{i \in J} \rho_{\ell}\omega^0 (p_i)_*\un_{E_i}\ar[r]\ar[d,"\rho_{\ell}(\eta)"]& \bigoplus\limits_{\{ i,j\} \subseteq J} \rho_{\ell} (p_{ij})_* \un_{E_{ij}}\ar[d,equals] \\ 
\rho_{\ell}p_*\un_E\ar[r] & \bigoplus\limits_{i \in J} \rho_{\ell}(p_i)_*\un_{E_i}\ar[r]& \bigoplus\limits_{\{ i,j\} \subseteq J} \rho_{\ell} (p_{ij})_* \un_{E_{ij}} 
\end{tikzcd}$$
which yields a commutative diagram with exact rows
\smallskip

\adjustbox{scale=0.8,center}{%
\begin{tikzcd}
\bigoplus\limits_{i \in J} {}^{p}\Hl^0\rho_{\ell}\omega^0 (p_i)_*\un_{E_i}\ar[r]\ar[d,"{}^{p}\Hl^0\rho_{\ell}(\eta)"]\ar[r]&
\bigoplus\limits_{\{ i,j\} \subseteq J} \rho_{\ell} (p_{ij})_* \un_{E_{ij}}\ar[d,equals] \ar[r]& 
{}^{p}\Hl^1\rho_{\ell}\omega^0p_*\un_E \ar[r] \ar[d,"{}^{p}\Hl^1\rho_{\ell}(\eta)"]\ar[r]& 
\bigoplus\limits_{i \in J} {}^{p}\Hl^1\rho_{\ell}\omega^0 (p_i)_*\un_{E_i}\ar[d,"{}^{p}\Hl^1\rho_{\ell}(\eta)"]\ar[r]&
0 \\ 
\bigoplus\limits_{i \in J}{}^{p}\Hl^0 \rho_{\ell}(p_i)_*\un_{E_i}\ar[r]& 
\bigoplus\limits_{\{ i,j\} \subseteq J} \rho_{\ell} (p_{ij})_* \un_{E_{ij}} \ar[r]&
{}^{p}\Hl^1\rho_{\ell}p_*\un_E\ar[r] &
\bigoplus\limits_{i \in J} {}^{p}\Hl^1\rho_{\ell}\omega^0 (p_i)_*\un_{E_i}\ar[r]&
0 \\ 
\end{tikzcd}}
The proof then follows from the Five Lemma and from \Cref{635} below.
\end{proof}
\begin{lemma}\label{635} Let $(S,\delta)$ be an excellent $1$-dimensional d-scheme and let $f\colon X\rar S$ be a proper map, let $\ell$ be a prime number. Assume that the scheme $X$ is regular, connected and at most $1$-dimensional. Then, \begin{enumerate}\item The map $${}^{p} \Hl^0(\rho_{\ell} \omega^0 f_*\un_X) \rar {}^{p} \Hl^0(\rho_{\ell} f_*\un_X)$$ is an isomorphism.
\item The map $${}^{p} \Hl^1(\rho_{\ell} \omega^0 f_*\un_X) \rar {}^{p} \Hl^1(\rho_{\ell} f_*\un_X)$$ is a monomorphism.
\end{enumerate}
\end{lemma}
\begin{proof} If $f$ is finite, then, $\omega^0 f_*\un_X=f_*\un_X$ and both statements hold. Otherwise, the image of $X$ is a point and the lemma follows from \Cref{omega^0 corps}.
\end{proof}

\section*{Appendix: The six functors for cohomological motives (following Gabber)}

\renewcommand{\thetheorem}{A.\arabic{theorem}}
\setcounter{theorem}{0}

\hypertarget{lol}{In} \cite[1.12]{plh}, Pepin Lehalleur proved that (constructible) cohomological motives are stable under the six functors. However, one has to assume that $R=\Q$ and to have the functors $f^!$ and $f_*$ one also has to assume the existence of
resolutions of singularities. This is not necessary: one can mimic the method of Gabber \cite[6.2]{em} in the cohomological case. Thus, we can prove that if $X$ (resp. $Y$) is a quasi-excellent scheme, and $f\colon Y\rar X$, $f_*$ (resp. $f^!$) preserves cohomological constructible motives.  We will outline how to do this.
The following proposition is straightforward:
\begin{proposition}\label{obvious_stability}
	Let $R$ be a commutative ring. The fibred subcategory $\DM^{\coh}_{\et,c}(-,R)$ of $\DM_{\et}(-,R)$ is stable under tensor product, negative Tate twists, $f^*$ if $f$ is any morphism and $f_!$ if $f$ is separated and of finite type.
\end{proposition}

\begin{corollary}
	Let $i\colon D\to X$ be the closed immersion of a simple normal crossing divisor, with $X$ regular. Then the motive $i^!\un_X$ is cohomological constructible.
\end{corollary}
\begin{proof}
If $D$ only has one branch, the result follows from absolute purity and stability by negative Tate twist. The general case then follows by induction on the number of branches and cdh-descent.
\end{proof}

\begin{definition}
	Let $S$ be a scheme, let $R$ be a commutative ring and let $\pp$ be a prime ideal of $\Z$. 
	An object $M$ of $\DM_{\et}(S,R)$ is called \emph{$\pp$-constructible cohomological} when its image in $\DM_{\et}(S,R)_\pp$ (with the notations of \Cref{p-loc}) belongs to $\DM^{\coh}_{\et,c}(X,R)_\pp$. 
\end{definition}

\begin{proposition}\label{p-locality}
	An étale motive is constructible cohomological if and only if it is $\pp$-constructible cohomological for any prime ideal $\pp$ of $\Z$.
\end{proposition}
\begin{proof}
	This follows from \cite[B.1.7]{em}.
\end{proof}

\begin{proposition}\label{char_p_case}
	Let $p$ be a prime number. Let $S$ be a scheme of characteristic $p$ and let $R$ be a ring. Then, an étale motive an étale motive over $S$ is $(p)$-constructible cohomological if and only if it is $(0)$-cohomological constructible.
\end{proposition}
\begin{proof}
	The proof is the same as in \cite[6.2.4]{em}: using the Artin-Schreier exact sequence \cite[A.3.1]{em}, the natural map $\DM_{\et}(S,R)_{(p)}\to \DM^{\et}(S,R)_{(0)}$ is invertible.
\end{proof}

The following lemma can be compared with \cite[6.2.7, 6.2.9]{em} and \cite[XIII.3.1.3]{travauxgabber}:
\begin{lemma}\label{Gabber_lemma}(Gabber's lemma)
	Let $\mc{T}_1$ be a fibered subcategory of $\DM^{\et}(-,R)$ over the category of (noetherian finite-dimensional) schemes. Assume that:
	\begin{enumerate}[label=(\alph*)]
		\item For any scheme $X$, the subcategory $\mc{T}_1(X)$ of $\DM_{\et}(X,R)$ is thick and contains the unit object $\un_X$.
		\item For any morphism of finite type $f \colon X'\to X$, the fibered category $\mc{T}_1$ is stable under $f_!$.
		\item For any dense open immersion $j \colon V \to Y$ , with $Y$ regular, which is the complement of a strict normal crossing divisor, the motive $j_*(\un_V)$ lies in $\mc{T}_1(Y)$.
	\end{enumerate}
Let $\pp$ be a prime ideal of $\Z$ and $\mc{T}_0$ be the fibered subcategory of $\DM_{\et}(-,R)$ defined for any $X$ by
\[\mc{T}_0(X)=\{ M \in \DM_{\et}(X,R)\mid M_\pp \in \mc{T}_1(X)_{\pp}\}\]
Then, for any dense open immersion $j \colon U \to X$ with $X$ quasi-excellent such that for any point $x$ of $X$,
the exponent characteristic of the residue field $k(x)$ is not in $\pp$, the motive $j_*(\un_U)$ lies in $\mc{T}_0(X)$.
\end{lemma}
\begin{proof}
	The fibered subcategory $\mc{T}_0$ also satisfies (a), (b) and (c). 
The proof is a noetherian induction on $X$. We may assume that $X$ is reduced, and it is sufficient to prove by induction on $c\geqslant 0$\footnote{There are two inductions going on; the induction hypothesis on $X$ will only be used once later in the proof.} that there exists a closed subscheme $T\subseteq X$ of codimension $>c$ such that $j_*(\un_U)|_{(X\setminus T)}$ lies in $\mc{T}_0(X\setminus T)$. The case where $c = 0$ is clear: we can choose $X\setminus T= U$ and use (a). If $c > 0$, choose a closed subscheme $T$ of $X$ of codimension $> c-1$, such that $j_*(\un_U)|_{(X\setminus T)}$ lies in $\mc{T}_0(X\setminus T)$. Using Mayer-Vietoris, (a) and (b), it is sufficient to find a dense open subscheme $V$ of $X$, which contains all the generic points of $T$, and such that $j_*(\un_U)|_{V}$ lies in $\mc{T}_0(V)$. In particular, we can always replace $X$ by a generic neighborhood of $T$ so that $T$ is purely of codimension $c$. Using the localization triangle \ref{AM.localization} and (b), it is then sufficient to prove that $u_*u^* j_*(\un_U)$ belongs to $\mc{T}_0(X)$ with $u\colon T\to X$ the inclusion.

	As in \cite{em}, we introduce the following notation: given a morphism $i \colon Z \to W$ of $X$-schemes, consider commutative diagram:
\[\begin{tikzcd}
	Z & W & {W_U} \\
	& X & U
	\arrow["f", from=1-1, to=1-2]
	\arrow["\pi"', from=1-1, to=2-2]
	\arrow[from=1-2, to=2-2]
	\arrow["{j_W}"', from=1-3, to=1-2]
	\arrow["\lrcorner"{anchor=center, pos=0.125, rotate=-90}, draw=none, from=1-3, to=2-2]
	\arrow[from=1-3, to=2-3]
	\arrow["j"', hook', from=2-3, to=2-2]
\end{tikzcd}\]
where the right hand square is Cartesian, and define an étale motive \[\varphi(W,Z)=\pi_* f^*j_*(\un_{W_U}).\]
Note that this defines a contravariant functor from the category of morphisms of $X$-schemes to $\DM_{\et}(X,R)$. The goal is to prove that $\varphi(X,T)$ belongs to $\mc{T}_0(X)$.

Now, \cite[6.2.8]{em} yields the following data:
\begin{enumerate}[label=(\roman*)]
	\item a finite h-cover $\{f_i\colon  Y_i\to X\}_{i\in I}$ such that for all $i$ in $I$, $f_i$ is of finite type, the scheme $Y_i$ is regular, and $f_i^{-1}(U)$ is either $Y_i$ itself or the complement of a strict normal crossing divisor in $Y_i$ ; 
	\item letting $f \colon Y = \bigsqcup_{i \in I}Y_i \to X$ be the induced global h-cover, a commutative diagram
\[\begin{tikzcd}
	{X'''} && Y \\
	{X''} & {X'} & X
	\arrow["g", from=1-1, to=1-3]
	\arrow["q"', from=1-1, to=2-1]
	\arrow["f", from=1-3, to=2-3]
	\arrow["u", from=2-1, to=2-2]
	\arrow["p", from=2-2, to=2-3]
\end{tikzcd}\]
	in which: $p$ is a proper birational morphism, $u$ is a Nisnevich cover, and $q$ is a flat finite surjective morphism of degree not in $\pp$.
\end{enumerate}
Let $j' \colon  U' \to X '$ denote the pullback
of $j$ along $p$. Then, we can find, by induction on $c$, a closed subscheme $T'$
in $X'$, of codimension $>c-1$, such that $j'_*(\un_{U'})_{X'\setminus T'}$ is in
$\mc{T}_0(X')$. Replacing $X$ further by a generic neighborhood of $T$, we can assume by \cite[6.2.8]{em} that 
\begin{enumerate}[label=(\roman*)]
\setcounter{enumi}{2}
\item $p(T') \subseteq T$ and the induced map $T'\to T$ is finite and sends any generic point to a generic point;
\item if we write $T'' = u^{-1}(T')$, the induced map $T'' \to T'$ is an isomorphism.
\end{enumerate}
Then, \cite[6.2.11]{em} ensures (using (iv)) that $\varphi(X',T')\to \varphi(X'',T'')$ is an isomorphism and \cite[6.2.12]{em} ensures that $\varphi(X'',T'')\to \varphi(X''',T''')$ admits a $\pp$-quasi-section. 

The same proof as \cite[6.2.10]{em} also ensures that the cone of the map $\varphi(X,T)\to \varphi(X',T')$ lies in $\mc{T}_0(X)$; we explain the key points for completeness. As the map $\varphi(X,T)\to \varphi(X',T')$ factors as 
\[\varphi(X,T)\to \varphi(X',p^{-1}(T))\to \varphi(X',T')\] and it suffices to prove that the cone of both above maps lie in $\mc{T}_0(X)$. 
For the first one, let $V$ be a dense open subscheme of $U$ such that $p^{-1}(V ) \to V$ is an isomorphism; write $k \colon Z \to U$ for the complement closed immersion. Letting $\overline{Z}$ be the reduced closure of $Z$ in $X$, cdh-descent yields a cartesian square:
\[\begin{tikzcd}
	{\varphi(X,T)} & {\varphi(X',p^{-1}(T))} \\
	{\varphi(\overline{Z},\overline{Z}\cap T)} & {\varphi(p^{-1}(\overline{Z}),p^{-1}(\overline{Z}\cap T))}
	\arrow[from=1-1, to=1-2]
	\arrow[from=1-1, to=2-1]
	\arrow[from=1-2, to=2-2]
	\arrow[from=2-1, to=2-2]
\end{tikzcd}\]
and the noetherian induction hypothesis on $X$ proves that both bottom objects lie in $\mc{T}_0(X)$. 
For the second map, given (iii) above, using the localization triangle \ref{AM.localization} gives an exact triangle:
\[p_*[j'_*(\un_{U'})|_{p^{-1}(T)\setminus T'}] \to \varphi(X',p^{-1}(T))\to \varphi(X',T')\]
and $p_*[j'_*(\un_{U'})|_{p^{-1}(T)\setminus T'}]$ belongs to $\mc{T}_0(X)$ because $j'_*(\un_{U'})_{X'\setminus T'}$ is in
$\mc{T}_0(X')$ is in $\mc{T}_0(X')$ and we can use (b).

Thus, we have proven so far that the image of the map $\alpha\colon \varphi(X,T)\to \varphi(X''',T''')$ through the canonical functor
\[\Xi\colon \DM_{\et}(X,R)\to \left(\DM_{\et}(X,R)/\DM_{\et,c}^{\coh}(X,R)\right)_{\pp}\] 
is a split monomorphism. Using \cite[B.1.7]{em}, what we want to prove is that $\Xi(\varphi(X,T))$ vanishes. But the map $\alpha$ factors as 
\[\varphi(X,T)\to\varphi(Y,T''') \to\varphi(X''',T''')\] so it suffices to prove that $\Xi(\varphi(Y,T'''))$ vanishes, \textit{i.e.} that $\varphi(Y,T''')$ is in $\mc{T}_0(X)$ but this is true because the map $T'''\to X$ is finite (using (iii), (iv) and the definition of $q$) and writing the definition of $\varphi(Y,T''')$, the result follows from (b), (c) and (i).
\end{proof}

\begin{corollary}
	Let $X$ be a quasi-excellent scheme and let $\pp$ be a prime ideal of $\Z$. Assume that, for any point $x$ of $X$, the exponent characteristic of the residue field $k(x)$ is not in $\pp$. Then, for any dense open immersion $j \colon U \to X$, the motive $j_*(\un_U)$ is $\pp$-constructible.
\end{corollary}

\begin{theorem}
	Let $f \colon Y \to X$ be a morphism of finite type such that X is a quasi-excellent scheme (of finite dimension). Then for any constructible cohomological étale motive
$M$ in $\DM_{\et}(Y,R)$, the motive $f_*(M)$ is constructible in $\DM_{\et}(X,R)$.
\end{theorem}
\begin{proof}
Using Nagata's compactification theorem, it is classical to reduce $f=j$ with $j\colon U \to X$ a dense open immersion and $M=\un_U$ (see the beginning of the proof of \cite[6.2.13]{em}). The rest of the proof is also contained in \cite[6.2.13]{em}: by \Cref{p-locality}, we have to show that for any prime ideal $\pp$ of $\Z$, the motive $j_*(\un_U)$ is $\pp$-constructible cohomological. If $\pp=(0)$, this is given by \Cref{Gabber_lemma}. Otherwise $\pp=(p)$ with $p$ a prime number and we reduce to schemes of characteristic $p$ or schemes with $p$ invertible by using the localization exact triangle \ref{AM.localization} for the closed immersion $X\times_{\Spec(\Z)}\Spec(\Z/p\Z)\to X$. If $p$ is invertible, we use \Cref{Gabber_lemma}; if $p=0$ on $X$, we use \Cref{char_p_case}.
\end{proof}

\begin{corollary}\label{main_thm_appendix}
	The six operations preserve constructible cohomological motives in $\DM_{\et}(-, R)$ over
quasi-excellent (noetherian) schemes (of finite dimension).
\end{corollary}
\bibliographystyle{alpha}
\bibliography{biblio.bib}
\end{document}

%% file: Resume.tex
\begin{abstract}
In this text, we are mainly interested in the existence of the perverse motivic t-structures on the category of Artin étale motives with integral coefficients. We construct the \emph{perverse homotopy} t-structure which is the best possible approximation to a perverse t-structure on Artin motives with rational coefficients. The heart of this t-structure has properties similar to those of the category of perverse sheaves and contains the Ayoub-Zucker motive. With integral coefficients, we construct the perverse motivic t-structure on Artin motives when the base scheme is of dimension at most $2$ and show that it cannot exist in dimension $4$. This construction relies notably on an analogue for Artin motives of the Artin Vanishing Theorem.
\end{abstract}